\documentclass[11pt]{article}
\usepackage{graphicx} 

\newcommand\DMO[2]{\DeclareMathOperator{#1}{#2}}

\usepackage{pdfsync}
\usepackage{amsmath,amsthm,amsfonts,amscd,amssymb}
\usepackage{color}

\usepackage[all]{xy} 

\usepackage{makeidx}

\newcounter{enumitemp}
\newenvironment{enumeratecontinue}{
  \setcounter{enumitemp}{\value{enumi}}
  \begin{enumerate}
  \setcounter{enumi}{\value{enumitemp}}
}
{
  \end{enumerate}
}

\numberwithin{equation}{section}

\newcommand\pref[1]{(\ref{#1})}

\newcommand{\nb}[1]{#1\nobreakdash-}

\newcommand\ds\displaystyle

\theoremstyle{plain}

\newtheorem*{maintheorem}{Main Theorem}
\newtheorem{theorem}{Theorem}[section]
\newtheorem{proposition}[theorem]{Proposition}
\newtheorem{lemma}[theorem]{Lemma}
\newtheorem{sublemma}[theorem]{Sublemma}
\newtheorem{corollary}[theorem]{Corollary}

\theoremstyle{definition}
\newtheorem{definition}[theorem]{Definition}

\DeclareMathOperator{\Out}{Out}
\DeclareMathOperator{\Aut}{Aut}
\DeclareMathOperator{\rank}{rank}
\DeclareMathOperator{\Length}{Length}

\DeclareMathOperator{\Energy}{Energy}
\DeclareMathOperator{\Stab}{Stab}
\DeclareMathOperator\diam{diam}
\DeclareMathOperator\image{image}

\newcommand\reals{{\mathbf R}}

\newcommand\I{{\mathcal I}}

\newcommand{\bdy}{\partial}
\newcommand{\from}{\colon}
\newcommand\composed{\circ}
\newcommand\suchthat{\bigm|}
\newcommand\inv{{-1}}
\newcommand\union{\cup}

\newcommand\abs[1]{\left| #1 \right|}

\newcommand\intersect{\cap}

\newcommand\restrict{\bigm|}
\newcommand\subgroup{<}
\newcommand\cross{\times}

\newcommand\C{\mathcal C}
\newcommand\D{\mathcal D}
\newcommand\E{\mathcal E}
\newcommand\F{\mathcal F}

\newcommand\K{{\mathcal K}}
\newcommand\M{\mathcal M}

\newcommand\<\langle
\renewcommand\>\rangle
\newcommand\wh\widehat
\newcommand\disjunion\sqcup

\DeclareMathOperator\interior{int}

\DeclareMathOperator\Fr{Fr}

\newcommand\act\curvearrowright

\newcommand\X{\mathcal{X}}
\newcommand\CV\X

\newcommand\BookZero{\cite{BFH:laminations}}
\newcommand\BookOne{\cite{BFH:TitsOne}}

\newcommand\ti {\tilde}

\DMO\Core{Core}
\DMO\ACore{{\hat{\mathcal{C}}}}
\DMO\truss{truss}
\newcommand\wt\widetilde

\newcommand\FS{\mathcal{FS}}

\newcommand\collapsesto\succ
\newcommand\collapse\collapsesto
\newcommand\collapses\collapsesto
\newcommand\expandsto\prec
\newcommand\expand\expandsto
\newcommand\expands\expandsto

\newcommand\Hull{H}

\newcommand\indexemph[1]{\index{#1}\emph{#1}}

\title{The free splitting complex of a free group I:  Hyperbolicity}

\author{Michael Handel and Lee Mosher}

\makeindex

\begin{document}

\maketitle

Given a free group $F_n$ of finite rank $n \ge 2$, a \indexemph{free splitting} over $F_n$ is a minimal, simplicial action of the group $F_n$ on a simplicial tree $T$ such that the stabilizer of each edge of $T$ is the trivial subgroup of $F_n$. A free splitting is denoted $F_n \act T$, or just $T$ when the group and its action are understood. Although the tree $T$ is allowed to have vertices of valence~$2$, there is a unique \emph{natural cell structure} on $T$ the vertices of which are the points of valence~$\ge 3$. We say that $T$ is a \emph{$k$-edge free splitting} if $k$ is the number of natural edge orbits, a number which can take on any value from $1$ to $3n-3$. The equivalence relation amongst free splittings is \indexemph{conjugacy}, where two free splittings of $F_n$ are conjugate if there exists an $F_n$-equivariant homeomorphism between them. See the beginning of Section~\ref{SectionFreeSplittingComplex} for the details of these definitions.

The \emph{free splitting complex} of $F_n$, denoted $\FS(F_n)$, is a simplicial complex of dimension $3n-4$ having a simplex $\<T\>$ of dimension $k$ for each conjugacy class of $k+1$-edge free splittings $F_n \act T$. Given another free splitting $F_n \act S$, the simplex $\<S\>$ is a face of $\<T\>$ if and only if there is a \emph{collapse map}\index{map!collapse} $T \mapsto S$, which collapses to a point each edge in some $F$-invariant set of edges of $T$. We write $T \collapsesto S$ for the relation ``$T$ collapses to $S$'', and $S \expandsto T$ for the inverse relation ``$S$ expands to $T$''. There is a natural left action of the outer automorphism group $\Out(F_n)$ on $\FS(F_n)$, where $\phi \in \Out(F_n)$ acts on the conjugacy class of a free splitting $F_n \act T$ by precomposing the action by an automorphism of $F_n$ representing $\phi$. The free splitting complex was introduced by Hatcher in \cite{Hatcher:HomStability} in its role as the sphere complex of a connected sum of $n$ copies of the 3-manifold $S^2 \cross S^1$. A careful construction of an isomorphism between the \nb{1}skeletons of $\FS(F_n)$ and Hatcher's sphere complex can be found in \cite{AramayonSouto:FreeSplittings}, and that proof extends with little trouble to the entire complexes. In Section~\ref{SectionFSInTermsOfCollapsing} we shall give a rigorous construction of the free splitting complex given purely in tree language. 

The complex $\FS(F_n)$ is regarded as one of several $\Out(F_n)$ analogues of the curve complex of a surface --- another competing analogue is the free factor complex of $F_n$ introduced by Hatcher and Vogtmann in \cite{HatcherVogtmann:FreeFactors}. The analogies are imperfect in each case: Hatcher and Vogtmann showed that the free factor complex, like the curve complex, has the homotopy type of a wedge of spheres of constant dimension \cite{HatcherVogtmann:FreeFactors}; by contrast, Hatcher showed that $\FS(F_n)$ is contractible \cite{Hatcher:HomStability}. On the other hand we showed in \cite{HandelMosher:distortion} that simplex stabilizers of $\FS(F_n)$ are all undistorted subgroups of $\Out(F_n)$, just as simplex stabilizers of the curve complex of a surface are undistorted subgroups of its mapping class group; by constrast, we also showed that the simplex stabilizers of the free factor complex of $F$ are, most of them, distorted in $\Out(F_n)$.

\bigskip

Here is our main result, an analogue to the theorem of Masur and Minsky \cite{MasurMinsky:complex1} on the hyperbolicity of the curve complex:

\begin{maintheorem}
The free splitting complex $\FS(F_n)$, with its geodesic simplicial metric, is Gromov hyperbolic.
\end{maintheorem}

By comparison Bestvina and Feighn have proved that the free factor complex is Gromov hyperbolic \cite{BestvinaFeighn:FFCHyp}.

In rank $n=2$, the Main Theorem is well known, because the simplicial complex $\FS(F_2)$ contains the Farey graph as a coarsely dense subcomplex, and the Farey graph is quasi-isometric to an $\reals$-tree and is therefore Gromov hyperbolic (see e.g.\ Example~5.2 of \cite{Manning:pseudocharacters}).

One should contrast the Main Theorem with the result of Sabalka and Savchuk \cite{SabalkaSavchuk:NotHyperbolic} which says that the ``edge splitting graph'' of $F_n$ is not hyperbolic --- this is the \nb{1}dimensional subcomplex of $\FS(F_n)$ spanned by the $0$-simplices corresponding to those $1$-edge free splittings $F \act T$ that have $2$ vertex orbits. Their result has an analogue in a theorem of Schleimer \cite{Schleimer:CurveComplexNotes} that on a closed, oriented surface of genus~$\ge 3$, the subcomplex of the curve complex spanned by separating curves is not hyperbolic. 

\bigskip

In Part II of this work we shall determine the dynamics of the action of elements of $\Out(F_n)$ on $\FS(F_n)$, showing in particular that $\phi \in \Out(F_n)$ acts loxodromically on $\FS(F_n)$ if and only if, in the terminology and notation of \BookOne, there exists an element $\Lambda$ of the set $\mathcal{L}(\phi)$ of attracting laminations such that the free factor support of $\Lambda$ is the whole group $F_n$.

\bigskip

\subsection*{Outline of the proof} 

Outside of applying the hyperbolicity axioms of Masur and Minsky our methods of proof, although intricate, are mostly self contained, depending on basic tools from the theory of group actions on trees including Bass-Serre theory and Stallings folds. Beyond the methods there are important motivations coming from the proof of Masur and Minsky, in particular the definition of the projection maps that play a role in verifying the Masur--Minsky axioms.

\subparagraph{Section~\ref{SectionFreeSplittingComplex}.} We give the basic concepts underlying the construction of the free splitting complex $\FS(F_n)$, including definitions of collapse maps, and Lemma~\ref{LemmaFSSimplices} which contains the technical results about free splittings that are needed to verify that $\FS(F_n)$ is, indeed, a simplicial complex. The proof of that lemma is given in Section~\ref{SectionFSSimplicesProof}. Collapse maps are also needed to understand the first barycentric subdivision $\FS'(F_n)$, which is what we actually use in our proof of hyperbolicity. In brief, $\FS'(F_n)$ has a \emph{vertex} for each conjugacy class of free splitting $F \act T$, and an oriented edge for each collapse relation $T \collapsesto S$. Since the composition of two collapse maps is a collapse map, the collapse relation is transitive, from which it follows that each geodesic in the \nb{1}skeleton of $\FS'(F_n)$ is a ``zig-zag path'' that alternates between collapses and expansions.

\subparagraph{Sections~\ref{SectionFoldPaths} and~\ref{SectionMasurMinsky}.} Following Stallings method \cite{Stallings:folding} as extended by Bestvina and Feighn \cite{BestvinaFeighn:bounding}, we define a system of paths in $\FS'(F)$ called \emph{fold paths}. We also review the criterion for hyperbolicity due to Masur and Minsky \cite{MasurMinsky:complex1}, which is concerned with familes of paths and projection maps to those paths that satisfy certain axioms, which we refer to as the \emph{Coarse Retraction}, \emph{Coarse Lipschitz}, and \emph{Strong Projection} axioms. 

The first step of progress on the Main Theorem is the statement of Proposition~\ref{PropFoldContractions} which asserts the existence of a system of projection maps, one such map from the ambient space $\FS'(F_n)$ to each fold path, that satisfy the Masur-Minsky axioms.

\subparagraph{Section~\ref{SectionCombing}.} We introduce the concept of combing of fold paths. The combing process has as input a fold path $S_0 \mapsto\cdots\mapsto S_K$ plus a single edge in $\FS'(F_n)$ with one endpoint $S_K$ and opposite endpoint denoted $S'_K$, which can be either a collapse $S_K \collapsesto S'_K$ or an expand $S_K \expandsto S'_K$. The output is a fold path (roughly speaking) from some $S'_0$ to $S'_K$ which stays a uniformly bounded distance from the input path, and which has the following rather strong asynchronous fellow traveller property: every free splitting along the input fold path from $S_0$ to $S_K$ is connected by a single edge to some free splitting along the output path from $S'_0$ to $S'_K$. The result of the combing process is a \emph{combing rectangle}, the general form of which is depicted in Figure~\ref{FigureCombingRectangle}. These rectangles are certain commutative diagrams of fold maps and collapse maps that can be viewed as living in the \nb{1}skeleton of $\FS'(F_n)$. We use many such diagrams throughout the paper, both as formal tools and as visualization aids.

Section~\ref{SectionCombingRectangles} contains basic definitions and properties regarding combing rectangles. In this section we also take the next step of progress in the proof of the Main Theorem, by using combing to define the system of projections maps to fold paths, and we state Proposition~\ref{PropProjToFoldPath} which asserts that these particular projection maps satisfy the Mazur Minsky axioms. Section~\ref{SectionCombingConstructions} contains the statements and proofs of various useful constructions of combing rectangles.

\subparagraph{Section~\ref{SectionFSU}.} We introduce \emph{free splitting units} as a way of subdividing a fold path into subpaths each of which has uniformly bounded diameter in $\FS'(F_n)$ (see Proposition~\ref{LemmaUnitsLipschitz}) but which nevertheless measure progress through $\FS'(F_n)$ (as stated later in Proposition~\ref{PropFoldPathQuasis}). Section~\ref{SectionDiamBounds} contains important diameter bounds for subsegments of fold paths. Section~\ref{SectionFSUDefsPropsApps} uses these diameter bounds to formulate the definition of free splitting units. Once they are defined, we are able to use the diameter bounds to quickly verify the \emph{Coarse Retraction} axiom; see Proposition~\ref{PropCoarseRetract}. 

\subparagraph{Section~\ref{SectionMainProof}.} We verify the \emph{Coarse Lipschitz} and \emph{Strong Projection} axioms, completing the proof of the Main Theorem. In this section we also verify that when a fold path is parameterized by free splitting units it becomes a quasigeodesic in $\FS'(F_n)$; see Proposition~\ref{PropFoldPathQuasis}. See the beginning of Section~\ref{SectionMainProof} for a sketch of the proof of the Main Theorem.

\section{The free splitting complex}
\label{SectionFreeSplittingComplex}

We begin with some basic notations used throughout the paper. 

For the rest of the paper we shall fix a free group $F$ of finite rank~$\ge 2$. 

A \emph{graph} is a 1-dimensional simplicial complex equipped with the CW topology. A \emph{tree} $T$ is a contractible graph. \emph{Simplicial maps} between graphs and trees are maps taking each vertex to a vertex, and taking each edge to a vertex or to another edge preserving barycentric coordinates. We use $G \act T$ to denote an action of a group $G$ on~$T$, which by definition is a homomorphism $G \mapsto \Aut(T)$ from $G$ to the group of simplicial automorphisms of $T$. The action associates to each $\gamma \in G$ a simplicial automorphism of $T$ denoted $x \mapsto \gamma \cdot x$, a notation that extends to subsets of $T$ by $\gamma \cdot A = \{\gamma \cdot x \suchthat x \in A\}$. The \emph{stabilizer} of a subset $A \subset T$ is the subgroup $\Stab_T(A) = \{\gamma \in G \suchthat \gamma \cdot A = A\}$. Given two actions $G \act S,T$, a function $f \from S \to T$ is said to be \emph{equivariant} if $f(\gamma \cdot x) = \gamma \cdot f(x)$ for all $x \in S$, $\gamma \in G$. 

Given a set $A$ and a subset $B \subset A$ we denote the set theoretic complement as~$A-B$. Given a graph $X$ and a subgraph $Y \subset X$ we denote the graph theoretic complement as $X \setminus Y$, whose topological description is the closure of $X-Y$.

\subsection{Free splittings, maps, natural vertices and edges, edgelets}
\label{SectionBasic}
Recall from the introduction that a \emph{free splitting of $F$}\index{free splitting} is an action $F \act T$ where $T$ is a tree that is not a point, the action is \emph{minimal} meaning that there is no proper $F$-invariant subtree, and for every edge $e \subset T$ the subgroup $\Stab_T(e)$ is trivial. We use without comment the basic fact that every homeomorphism of a tree $T$ either fixes a point or translates along a properly embedded copy of $\reals$ called its \emph{axis}, and that minimality of an action $F \act T$ is equivalent to the statement that $T$ is the union of the axes of the elements of $F$ that have no fixed point in~$T$. We also use without comment the fact that every free splitting is cocompact, that is, there is a finite number of orbits of vertices and of edges; this follows from Bass-Serre theory \cite{ScottWall} combined with the fact that the rank of $F$ is finite.

Given a free splitting $F \act T$, from Bass-Serre theory \cite{ScottWall} it follows that the set of conjugacy classes in $F$ of nontrivial vertex stabilizers of $T$ forms a free factor system in the sense of \BookOne, which means that by appropriate choice of representatives $H_1 = \Stab_T(v_1),\ldots,H_k=\Stab_T(v_k)$ of each conjugacy class --- where $v_1,\ldots,v_k$ are the corresponding vertex orbit representatives --- there exists a free factorization of the form $F = H_1 * \cdots * H_k * B$, with $B$ possibly trivial. We refer to this free factor system as the \emph{vertex group system} of~$F \act T$, and denote it $\F(T)$. Notice that a free splitting $F \act T$ is properly discontinuous if and only if $\F(T) = \emptyset$, if and only if every vertex has finite valence.

\begin{definition}[Maps between free splittings]
Given free splittings $F \act S,T$, a \indexemph{map} from $S$ to $T$ is defined to be an $F$-equivariant simplicial map $f \from S \to T$.
\end{definition}
We will encounter several different kinds of maps, most commonly ``collapse maps'' defined in Section~\ref{SectionCollapseMaps}, ``foldable maps'' defined in Section~\ref{SectionFoldableMaps}, and ``folds'' defined in Section~\ref{SectionFoldFactorizations}. The category of maps will usually suffice for much of this paper, but we will occasionally have to consider more general equivariant continuous functions between free splittings, for example conjugacies.

We will sometimes emphasize the role of the action of $F$ by referring to a ``free splitting over $F$'' or a ``map over $F$'', and we shall use similar terminology for more complicated objects introduced later on that are built out of free splittings and maps over $F$.

Recall from the introduction that a \indexemph{conjugacy} between free splittings $F \act S,T$ is an equivariant homeomorphism between $S$ and $T$. A conjugacy \emph{need not be} a map as just defined, i.e.\ it need not take vertices to vertices or edges to edges, and even if it does it need not preserve barycentric coordinates. Notice that if one is given a map $f \from S \to T$ as just defined --- an equivariant simplicial map --- then $f$ is a conjugacy if and only if it is locally injective: for if $f$ is locally injective then it is evidently injective, and it is surjective by minimality of the action $F \act T$, and so $f$ is a simplicial isomorphism and hence a homeomorphism.

Given a free splitting $F \act T$, recall also from the introduction the \indexemph{natural cell structure} on $T$, a CW structure whose $0$-skeleton is the set of \emph{natural vertices} which are the vertices of valence~$\ge 3$. Implicit in the definition of the natural cell structure is the fact that each point of~$T$ which is not a natural vertex is contained in the interior of a unique \emph{natural edge}, which is an arc of $T$ each of whose endpoints is a natural vertex and none of whose interior points is a natural vertex. If this fact were not true then $T$ would contain a valence~$1$ vertex, violating minimality, or $T$ would contain arbitrarily long simplicial arcs with no natural vertices. In the latter case, by cocompactness it would follow that $T$ is homeomorphic to a line: but then either the action would be properly discontinuous implying that $F$ has rank~$1$ which is a contradiction; or the kernel of the action would be a free factor of corank~$1$, contradicting that edge stabilizers are trivial. We have also defined the notion of a $k$-edge free splitting $F \act T$ meaning that $T$ has $k$ orbits of natural edges; this notion is invariant under conjugacy. In terms of Bass-Serre theory \cite{ScottWall}, the number of orbits of natural vertices of a free splitting $F \act T$ equals the number of points in the quotient graph of groups $T/F$ which either have a nontrivial group or have valence~$\ge 3$. 

The word ``natural'' in this context refers to naturality in the category of free splittings and conjugacies: every conjugacy is an automorphism of the natural cell structure, and in particular preserves the numbers of orbits of natural vertices and edges. On this basis one might have wished to refer to a valence~$1$ vertex as ``natural'', were it not for the fact that $T$ has no vertices of valence~$1$, by virtue of minimality of the action $F \act T$.

\paragraph{Remark on terminology.} Outside of discussions involving natural cell structures and nonsimplicial conjugacies, we work primarily in the simplicial category: a free splitting $F \act T$ comes equipped with a simplicial structure on the tree $T$ which is invariant under the action of~$F$; maps between free splittings are $F$-equivariant simplicial maps. This will be particularly convenient when we encounter subcomplexes of the simplicial structure which are not subcomplexes of the natural cell structure, for example in the results of Sections~\ref{SectionPushingDownPeaks} and~\ref{SectionProofFSUContraction} where the heart of the proof of the Main Theorem resides.

For any free splitting $F \act T$, in order to distinguish between the natural edges of~$T$ and the edges of the given simplicial structure on $T$ we shall refer to the latter as the \emph{edgelets}\index{edgelet} of $T$. This word is meant to evoke the phenomenon that, fairly often, there are many, many, many edgelets in a single natural edge, and we often visualize the edgelets as being very, very, very tiny.

\subsection{Collapse maps.} 
\label{SectionCollapseMaps}
In order to define the free splitting complex of $F$ rigorously we need some preliminaries regarding collapse maps. 

Given two free splittings $F \act S,T$, a map $f \from S \to T$ is called a \emph{collapse map}\index{map!collapse}\index{collapse} if $f$ is injective over the interior of each edgelet of $T$. The \emph{collapsed subgraph} $\sigma \subset S$ is the $F$-equivariant subgraph which is the union of those edgelets of $F$ which are collapsed to a vertex by the map $f$. We put $\sigma$ into the notation by writing $f \from S \xrightarrow{[\sigma]} T$, the square brackets highlighting that $\sigma$ is the name of the collapsed graph, whereas the notation $S \xrightarrow{f} T$ tells us the name of the collapse map $f$ itself. Note that $\sigma \subset S$ is a \emph{proper} subgraph, meaning that $\sigma \ne S$. 

Here are some basic facts about collapse maps. Items~\pref{ItemCollapseComponents} and~\pref{ItemNondegIsComponent} will be used without mention throughout the paper. Item~\pref{ItemCollapseFrontierHull} will be needed for the proof of Proposition~\ref{PropCBE}.

\begin{lemma}\label{LemmaCollapseProps}
For any free splittings $F \act S,T$, any collapse map $f \from S \xrightarrow{[\sigma]} T$, and any vertex $v \in T$, the following hold:
\begin{enumerate}
\item \label{ItemCollapseComponents}
The subgraph $f^\inv(v)$ is connected.
\item \label{ItemNondegIsComponent}
$f^\inv(v)$ does not degenerate to a point if and only if it is a component of $\sigma$.
\item \label{ItemCollapseFrontierHull}
$f^\inv(v)$ is the convex hull of its frontier in $S$.
\end{enumerate}
\end{lemma}

\begin{proof} Denote $\sigma_v = f^\inv(v)$. Given vertices $w_1 \ne w_2 \in \sigma_v$, if the segment $[w_1,w_2]$ does not map to $v$ then $f[w_1,w_2]$ is a nondegenerate finite tree and there must exist two edgelets in $[w_1,w_2]$ with the same image in that tree, contradicting the definition of a collapse~map; this proves that $\sigma_v$ is connected. If $\sigma_v$ is nondegenerate, i.e.\ if it contains an edgelet, then each of its edgelets being in $\sigma$ it follows by connectivity that $\sigma_v$ is a subset of $\sigma$. It is moreover a maximal connected subset of $\sigma$ --- a component of $\sigma$ --- because any edgelet of~$S$ incident to a vertex of $\sigma_v$ but not in $\sigma_v$ does not have constant image under $f$ and so is not contained in $\sigma$. This proves~\pref{ItemCollapseComponents} and~\pref{ItemNondegIsComponent}.

To prove \pref{ItemCollapseFrontierHull}, let $\Fr$ be the frontier of $\sigma_v$ in $S$ and let $\Hull \subset S$ be the convex hull of $\Fr$. By connectivity we have $\Hull \subset \sigma_v$. If the opposite conclusion did not hold then there would be an edgelet $e \subset \sigma_v \setminus \Hull$. Only one of its two complementary components $S \setminus e = S_0 \disjunion S_1$ can contain a point of $\Fr$, and so up to interchanging indices we have $\Hull \subset S_0$. Since $S_1$ is disjoint from $\Fr$ but contains the point $x = e \intersect S_1 \subset e \subset \sigma_v$, it follows that $S_1 \subset \sigma_v \subset \sigma$. The point $x$ is the unique frontier point of $S_1$. Choose $\gamma \in F$ having an axis $L$ contained in $S_1$. Let $z$ be the point of $L$ closest to $x$. For each $y \in S \setminus S_1$, $z$ is also the point of $L$ closest to $y$, and so $\gamma(z)$ is the point of $L$ closest to $\gamma(y)$. But $\gamma(z) \ne z$ and so $\gamma(y) \in S_1 \subset \sigma$, implying that $y \in \sigma$ and contradicting properness of $\sigma$.
\end{proof}

From Lemma~\ref{LemmaCollapseProps}~\pref{ItemCollapseComponents}, given a collapse map $f \from S \xrightarrow{[\sigma]} T$ it follows that $\sigma$ determines $T$ up to simplicial conjugacy, in that the map $S \mapsto T$ induces a simplicial isomorphism between $T$ and the quotient tree obtained from $S$ by collapsing each component of $\sigma$ to a point, and furthermore this simplicial isomorphism is $F$-equivariant. In this situation we often say that \emph{$T$ is obtained by collapsing $\sigma$}.

Furthermore, any choice of collapsed subgraph may be used, in the sense that for any free splitting $F \act S$ and any $F$-equivariant, proper subgraph $\sigma \subset S$ there exists a free splitting $T$ and a collapse map $S \xrightarrow{[\sigma]} T$. The tree $T$ is defined as the quotient of $S$ obtained by collapsing to a point each component of $\sigma$. Since $\sigma$ is proper, $T$ is not a point. Since $\sigma$ is equivariant, the action $F \act S$ descends to an action $F \act T$. This action is minimal because $T$ is a union of axes of elements of $F$: for each edge $e \subset T$ there exists a unique pre-image edge $e' \subset S$ such that $e'$ maps to $e$, and there exists $\gamma \in F$ whose axis in $S$ contains $e'$, so the axis of $\gamma$ in $T$ contains~$e$. The stabilizer of an edge $e \subset T$ equals the stabilizer of the pre-image edge and so is trivial. This shows that $F \act T$ is a free splitting, and by construction the quotient map $S \xrightarrow{[\sigma]} T$ is a collapse map.

The (nonsimplicial) conjugacy type of the collapsed tree actually depends only on the ``natural core'' of the collapsed subgraph. To be precise, given a free splitting $F \act S$ and a proper, $F$-equivariant subgraph $\sigma \subset S$, define the \emph{natural core}\index{natural core} of $\sigma$ to be the largest natural subcomplex of $S$ contained in $\sigma$ whose components are all nondegenerate. For any collapse maps $S \xrightarrow{[\sigma]} T$, $S \xrightarrow{[\sigma']} T'$, if $\sigma,\sigma'$ have the same natural core then there exists a conjugacy $T \to T'$, although this conjugacy need not be a simplicial map with respect to the given simplicial structures of $T,T'$.

\bigskip

Given free splittings $F \act S,T$, we say that $S$ \emph{collapses} to $T$\index{collapse} or that $T$ \emph{expands} to~$S$,\index{expansion} denoted $S \collapsesto T$ or $T \expandsto S$, if there exists a function $S \mapsto T$ which is a collapse map with respect to some simplicial subdivisions of the natural cell structures on $S$ and~$T$. These relations are well-defined on the conjugacy classes of $S,T$, indeed $S \collapsesto T$ if and only if there exist a function $S \mapsto T$ which is a collapse map with respect to the natural cell structures themselves. Even when it is known that $S \collapsesto T$, notice that there might not exist a collapse map $S \mapsto T$ without first changing the simplicial structures on $S$ and/or~$T$, for example if $T$ is subdivided so finely that it has more edgelet orbits than~$S$. The collapse and expand relations are transitive, e.g.\ if $S \collapsesto S' \collapsesto S''$ then $S \collapsesto S''$, for if $S \mapsto S' \mapsto S''$ are collapse maps of natural cell structures then the composition $S \mapsto S''$ is a collapse map of natural cell structures. 

In several places throughout the paper we use without comment the fact that every free splitting $F \act T$ has a \emph{properly discontinuous expansion} $T \expands S$, meaning that the free splitting $F \act S$ is properly discontinuous; see \cite{HandelMosher:distortion}, Section~3.2 for a proof, under the heading ``How to construct trees in $\K^T_n$'', Steps~1 and~2. When a properly discontinuous expansion $T \expands S$ is chosen, with collapse map $S \xrightarrow{[\sigma]} T$, the vertex group system of $T$ is represented in $S$ as the conjugacy classes of the stabilizers of the infinite components of~$\sigma$.

\subsection{The free splitting complex in terms of collapse maps.}
\label{SectionFSInTermsOfCollapsing}
The following result contains the technical facts needed to justify the construction of the simplicial complex $\FS(F)$. For any free splitting $F \act T$ and any proper $F$-invariant natural subgraph $\sigma \subset T$ let $T \xrightarrow{[\sigma]} T_\sigma$ be the corresponding collapse map, the quotient map obtained by collapsing to a point each component of $\sigma$. If $T$ is a $(K+1)$-edge free splitting then for each $k=0,\ldots,K$ let $\F_k(T)$ be the set of conjugacy classes of $(k+1)$-edge free splittings of the form $T_\sigma$, indexed by those natural subgraphs $\sigma \subset T$ that contain exactly $K-k$ natural edge orbits of $T$. There are exactly $\binom{K+1}{k+1}=\frac{(K+1)!}{(k+1)! (K-k)!}$ choices of such $\sigma$, although a priori one does not know where the cardinality of the set $\F_k(T)$ lies in the interval from $1$ to $\binom{K+1}{k+1}$, because one does not know whether collapsing two distinct $F$-invariant natural subgraphs results in nonconjugate free splittings. Furthermore one does not know a priori how the conjugacy class of $T$ depends on, say, the set $\F_0(T)$ of conjugacy classes of 1-edge collapses of~$T$. The following lemma resolves these issues as one might hope; the lemma will be proved in Section~\ref{SectionFSSimplicesProof}.

\begin{lemma}\label{LemmaFSSimplices} For any free splittings $F \act T,T'$ the following hold:
\begin{enumerate}
\item\label{ItemSimplexExists}
For any two $F$-equivariant natural subgraphs $\sigma_1,\sigma_2 \subset T$ we have $\sigma_1 = \sigma_2$ if and only if $T_{\sigma_1}$, $T_{\sigma_2}$ are conjugate.
\item\label{ItemSimplexUnique}
$\F_0(T)=\F_0(T')$ if and only if $T,T'$ are conjugate.
\end{enumerate}
\end{lemma}

By applying item~\pref{ItemSimplexExists} of this lemma we may define a collapse $T \collapse U$ to be \emph{proper}\index{collapse!proper and improper} if it satisfies any of the following equivalent conditions: $U,T$ are not conjugate; \emph{for any} map $T \xrightarrow{[\sigma]} U$ which is a collapse map with respect to some subdivision of the natural cell structures, the natural core of $\sigma$ is nonempty. We also refer to the collapse maps of the latter type as \emph{proper collapse maps}. Notice that properness of a collapse relation $T \collapse U$ is also equivalent to the statement that \emph{there exists} a map $T \xrightarrow{[\sigma]} U$ which is a collapse map with respect to some subdivision of the natural structures, such that the natural core of $\sigma$ is nonempty. A collapse relation $T \collapse U$ which is not proper is~\emph{improper}.

Before proving this lemma we apply it to the construction of $\FS(F)$. From item~(1) it follows that we can associate an abstract $K$-simplex denoted $\<T\>$ to the conjugacy class of each $(K+1)$-edge free splitting $F \act T$, where the $k$-dimensional faces of $\<T\>$ are labelled by the conjugacy classes of those free splittings of the form~$T_\sigma$ such that $\sigma$ contains exactly $K-k$ natural edge orbits of $T$, and where $T_\sigma$ is a face of $T_{\sigma'}$ if and only if $\sigma' \subset \sigma$. We can then glue these simplices together, where for each collapse relation $T \collapses U$ the simplex $\<U\>$ is glued to the unique face of the simplex $\<T\>$ that is labelled by the conjugacy class of $U$ and where the gluing preserves the labelling of subfaces. From item~(2) it follows that the result of these gluings is a simplicial complex. We have proved:

\begin{corollary}
There exists a simplicial complex $\FS(F)$ whose $K$-simplices $\<T\>$ are in one-to-one correspondence with the conjugacy classes of $K+1$-edge free splittings $F \act T$, such that for any pair of simplices $\<T\>$, $\<U\>$ we have $\<U\> \subset \<T\>$ if and only if $U \expandsto T$.
\end{corollary}

The alternate and more well known approach to this corollary is to appeal to Hatcher's construction of the sphere complex \cite{Hatcher:HomStability}; see for example Aramayona--Souto \cite{AramayonSouto:FreeSplittings} which constructs the 1-skeleton of $\FS(F)$ in this manner.

The dimension of $\FS(F)$ equals $3 \cdot \rank(F) - 4$, the number $3 \cdot \rank(F)-3$ being the maximum number of natural edge orbits of a free splitting $F \act T$, the maximum occuring if and only if every natural vertex of $T$ has valence~$3$ (which implies that the action $F \act T$ is properly discontinuous).

We usually work with the first barycentric subdivision of $\FS(F)$, denoted $\FS'(F)$. Gromov hyperbolicity of $\FS(F)$ and $\FS'(F)$ are equivalent because, as with any connected simplicial complex, the identity map is a quasi-isometry between their geodesic simplicial metrics (connectivity follows from Hatcher's proof of contractibility \cite{Hatcher:HomStability}, or from the construction of Stallings fold paths reviewed in Section~\ref{SectionFoldPaths}). The simplicial complex $\FS'(F)$ has one $0$-simplex associated to each conjugacy class of free splittings, and it has a $k$-simplex associated to each sequence of conjugacy classes of free splittings obtained from any chain of $k$ proper expansions $T_0 \expandsto T_1 \expandsto \cdots \expandsto T_k$. In particular, an edge in $\FS'(F)$ oriented from $S$ to $T$ can be written uniquely as either an expand $S \expandsto T$ or a collapse $S \collapsesto T$; uniqueness follows from asymmetry of the collapse relation, which is a consequence of Lemma~\ref{LemmaFSSimplices}~\pref{ItemSimplexExists}. 

As mentioned earlier, the relations of collapse and expand are transitive. It follows that every geodesic in the one-skeleton of $\FS'(F)$ can be written as an alternating sequence of expands and collapses, for example starting with an expand $T_0 \expand T_1 \collapse T_2 \expand T_3 \collapse T_4 \expand T_5 \collapse \cdots$ or starting with a collapse $T_0 \collapse T_1 \expand T_2 \collapse T_3 \expand T_4 \collapse T_5 \expand \cdots$. Any edge path in $\FS'(F)$ that alternates between expands and collapses is called a \emph{zig-zag path}\index{zig-zag} in $\FS'(F)$.

Throughout the paper, given free splittings $F \act S,T$, we use the notation $d(S,T)$ to denote the length of the shortest edge path in the simplicial complex $\FS'(F)$ between the vertices represented by $S$ and $T$. We must prove that this metric is Gromov hyperbolic.

\subsection{Proof of Lemma \ref{LemmaFSSimplices}}
\label{SectionFSSimplicesProof}
While the proof is surely standard, we are not aware of any proof in the literature, so we provide the details.

To each free splitting $F \act S$ and each oriented natural edge $\eta \subset S$ we associate a clopen decomposition $\bdy F = \C_-(\eta) \disjunion \C_+(\eta)$ as follows. Choose a proper expansion $S \expandsto R$ with collapse map $f \from R \to S$. Let $\eta_R \subset R$ be the unique oriented natural edge that maps to $\eta$ under the collapse $R \mapsto S$. The subgraph $R \setminus \eta_R$ has two components, incident to initial and terminal vertices of $\eta_R$, whose end spaces are $\C_-(\eta), \C_+(\eta) \subset \bdy F$, respectively. If one chooses any other proper expansion $S \expandsto R'$ with oriented natural edge $\eta_{R'}$ mapping to $\eta$, then as shown in \cite{HandelMosher:distortion} Lemma~17 there exists a sequence of collapses and expansions $R=R_0,\ldots,R_K=R'$ and oriented natural edges $\eta_R=\eta_{R_0} \subset R_0, \eta_{R_1} \subset R_1, \ldots, \eta_{R'}=\eta_{R_K} \subset R_K$ such that for each $k=1,\ldots,K$ the edges $\eta_{k-1}$, $\eta_k$ correspond to each other under the collapse map between $R_{k-1}$ and $R_k$ (whichever direction that map goes). It immediately follows that $\C_-(\eta_k)$, $\C_+(\eta_k)$ are each constant along this sequence. This shows that $\C_-(\eta),\C_+(\eta)$ are both well-defined independent of the choice of $R$. Denote the unordered pair by $\C(\eta)=\{\C_-(\eta),\C_+(\eta)\}$.

Note that for each $\gamma \in F$ and $\eta \subset S$ we have $\C(\gamma \cdot \eta) = \gamma \cdot \C(\eta)$. Also, given natural edges $\eta \ne \eta' \subset S$ we have $\C(\eta) \ne \C(\eta')$: for the proof we may assume $S$ is proper, so $S \setminus (\eta \union \eta')$ has three components, each infinite; choosing a ray in each we see that three of the four sets $\C_-(\eta) \intersect \C_-(\eta')$, \, $\C_-(\eta) \intersect \C_+(\eta')$, \, $\C_+(\eta) \intersect \C_-(\eta')$, \, $\C_+(\eta) \intersect \C_+(\eta')$ are nonempty, and so $\C(\eta) \ne \C(\eta')$. Also, for any collapse $S \xrightarrow{g} T$ and any edges $\eta_S \subset S$, $\eta_T \subset T$ such that $g(\eta_S)=\eta_T$, we have $\C(\eta_S) = \C(\eta_T)$, for in defining $\C(\eta_S)$ we can choose any proper expansion with collapse $R \xrightarrow{f} S$, in defining $\C(\eta_T)$ we can choose the same $R$ with collapse $R \xrightarrow{f} S \xrightarrow{g} T$, and one sees that the same edge of $R$ maps to $\eta_S$ and to $\eta_T$ under these collapse maps.

Consider now $T$, $T_{\sigma_1}$, and $T_{\sigma_2}$ as in~\pref{ItemSimplexExists} and suppose there exists a conjugacy $T_{\sigma_1} \to T_{\sigma_2}$, inducing a bijection of natural edges. If $e_i \subset T_{\sigma_i}$ $(i=1,2)$ correspond under this bijection, pull back under the collapse maps $T \to T_{\sigma_i}$ to obtain natural edges $e'_i \subset T$. From the previous paragraph it follows that $\C(e'_1) = \C(e_1) = \C(e_2) = \C(e'_2)$ which implies that $e'_1=e'_2$. Thus, a natural edge of $T$ is collapsed by $T \mapsto T_{\sigma_1}$ if and only if it is collapsed by $T \mapsto T_{\sigma_2}$, which implies that $\sigma_1=\sigma_2$. This proves~\pref{ItemSimplexExists}.

\bigskip

To prove \pref{ItemSimplexUnique}, given a free splitting $F \act T$, let $\C(T) = \union_{\eta\subset T} \{\C_-(\eta), \C_+(\eta)\}$ taken over all oriented natural edges $\eta \subset T$. The set $\C(T)$ is an $F$-invariant set of clopens in $\bdy F$ depending only on the conjugacy class of $T$. Since $\C(T) = \union_{T' \in \F_0(T)} \C(T')$, it follows that $\F_0(T)$ determines $\C(T)$, and so it suffices to show that $\C(T)$ determines the conjugacy class of $T$. The set $\C(T)$ does determine the oriented edges of $T$, which are in bijective, $F$-equivariant correspondence with $\C(T)$ itself via $\eta \leftrightarrow \C_+(\eta)$. Also, the unoriented edges of $T$ are in bijective, $F$-equivariant correspondence with subsets of $\C(T)$ of cardinality~$2$ which are partitions of $\bdy F$. It remains to show that $\C(T)$ also determines the vertices of $T$ and the ``initial vertex'' relation between oriented edges and vertices. 

Associated to each natural vertex $v \in T$ there is a subset $\D(v) \subset \C(T)$ consisting of all $\C_+(\eta) \in \C(T)$ such that $v$ is the initial vertex of $\eta$. If we can show that $\C(T)$ determines the collection $\{\D(v) \suchthat \text{$v$ is a natural vertex of $T$}\}$ then we will be done, because the initial vertex relation is then also determined: $v$ is an initial vertex of $\eta$ if and only if $\C_+(\eta) \in \D(v)$. Noting that the valence of $v$ equals the cardinality of $\D(v)$, we show first that $\C(T)$ determines the finite cardinality sets~$\D(v)$.

Define a relation on the set of subsets of $\C(T)$: given two subsets $\D,\C \subset \C(T)$ we write $\D \sqsubset \C$ if for every $D \in D$ there exists $C \in \C$ such that $D \subset C$.

If $v \in T$ is a natural vertex of finite valence then $\D(v)$ is a partition of $\bdy F$ of finite cardinality $\ge 3$. Furthermore, for every cardinality~2 subset $\C \subset \C(T)$ which is a partition of $\bdy F$ --- i.e.\ every subset of the form $\C=\{\C_-(\eta),\C_+(\eta)\}$ for some oriented natural edge $\eta \subset T$ --- if $\D(v) \sqsubset \C$ then there exists $D \in \D(v)$ and $C \in \C$ such that $D=C$. 

We claim that the converse holds: suppose $\D \subset \C(T)$ is a partition of $\bdy F$ of finite cardinality $\ge 3$, and suppose that $\D$ satisfies the property that for every $\C \subset \C(T)$ of cardinality~$2$ which is a partition of $\bdy F$, if $\D \sqsubset \C$ then there exists $D \in \D$ and $C \in \C$ such that $D=C$; then it follows that there exists a natural vertex $v \in T$ such that $\D=\D(v)$. To prove this claim, write $\D = \{\C_+(\eta_i)\}_{i \in I}$ for some finite set~$I$. Note that if $i \ne j \in I$ then $\eta_i,\eta_j$ have disjoint interiors, because otherwise they are opposite orientations of the same edge $\eta$ and $\D=\{\C_+(\eta),\C_-(\eta)\}$, contradicting that $\D$ has cardinality $\ge 3$. Also, if $i \ne j$ then the shortest path in~$T$ intersecting both of $\eta_i,\eta_j$ intersects them in their respective initial vertices, because $\C_+(\eta_i) \intersect \C_+(\eta_j) = \emptyset$. It follows that $T - \union_{i \in I} \interior(\eta_i)$ has a component $\tau$ that intersects each $\eta_i$ in its initial endpoint. If $\eta$ is any oriented natural edge such that $\tau\intersect \eta$ is the initial vertex of $\eta$ then $\C_+(\eta) \in \D$, for otherwise $\C_+(\eta) \subset \bdy F - \union \D$, contradicting that $\D$ is a partition of $\bdy F$. It follows that $\{\eta_i\}$ is precisely the set of oriented edges not in $\tau$ but with initial vertex in $\tau$. Suppose that  $\tau$ is a nondegenerate tree. If $\tau$ has finite diameter, pick any natural edge $\eta \subset \tau$, and note that $\D \sqsubset \C(\eta)$ but there does not exist any $D \in \D$ and $C \in \C(\eta)$ for which $D=C$, a contradiction. If $\tau$ has infinite diameter, any ray in $\tau$ determines an element of $\bdy F - \union \D$, contradicting that $\D$ partitions $\bdy F$. It follows that $\tau$ is a degenerate tree, a natural vertex $v \in T$ of cardinality $\ge 3$, and that $D=\D(v)$, proving the claim.

To summarize, the natural, finite valence vertices of $T$ are determined by $\C(T)$ in the following manner: they are in $F$-equivariant bijective correspondence, via the correspondence $v \leftrightarrow \D(v)$, with the subsets $\D \subset \C(T)$ which are partitions of $\bdy F$ of finite cardinality $\ge 3$, having the property that for every every two-element subset $\E \subset \C(T)$ which is a partition of $\bdy F$, if $\D \sqsubset \E$ then there exists $D \in \D$ and $E \in \E$ such that $D=E$.

It remains to describe a similar scheme by which $\C(T)$ determines the infinite valence vertices of~$T$. If $v \in T$ is a natural vertex of infinite valence then $\D(v)$ is an infinite partition of $\bdy F - \bdy \Stab_T(v)$, it is invariant under the action of the free factor $\Stab_T(v) \subgroup F$ on $\bdy F$, and it has the property that for any cardinality~$2$ subset $\C \subset \C(T)$ which is a partition of $\bdy F$, if $\D(v) \sqsubset \C$ then there exists $D \in \D(v)$ and $C \in \C$ such that $D=C$. Conversely, let $\D \subset \C(T)$ be an infinite subset for which there exists a proper, nontrivial free factor $A \subgroup F$ such that $\D$ is a clopen partition of $\bdy F - \bdy A$ and $\D$ is invariant under the action of $A$ on $\C(T)$, and for any cardinality~$2$ subset $\E \subset \C(T)$ which is a partition of $\bdy F$, if $\D \sqsubset \E$ then there exists $D \in \D$ and $E \in \E$ such that $D=E$. Under these conditions we must prove that there exists a vertex $v \in A$ such that $A = \Stab_T(v)$ and $\D=\D(v)$. Just as in the finite valence case, writing $\D = \{\C_+(\eta_i)\}_{i \in I}$ where the index set $I$ is now infinite, there is a component $\tau$ of $T - \union_{i \in I} \interior(\eta_i)$ that intersects each $\eta_i$ in its initial endpoint. Since $\D$ is $A$-invariant, the collection $\{\eta_i\}_{i \in I}$ is also $A$-invariant, and so $\tau$ is $A$-invariant. The set $\{\eta_i\}_{i \in I}$ is precisely the set of oriented natural edges not in $\tau$ but with initial vertex in $\tau$, for if $\eta$ is an oriented natural edge such that $\tau\intersect \eta$ is the initial vertex of $\eta$ then $\C_+(\eta) \subset \bdy F - \bdy A$, and if $\C_+(\eta) \not\in \D$ then $\C_+(\eta) \subset (\bdy F - \bdy A) - \union \D$, contradicting that $\D$ is a partition of $\bdy F - \bdy A$. 

If $\tau$ is nondegenerate and of finite diameter then we obtain the same contradiction as in the case where $\D$ is finite. Suppose $\tau$ is nondegenerate and of infinite diameter. The action of the free factor $A$ on $T$ has a unique, minimal invariant subtree $T^A$, and so $T^A \subset \tau$. If $T^A$ is nondegenerate then for any edge $\eta \subset T^A$ we have $\D \sqsubset \C_\pm(\eta)$ but no $D \in \D$ equals any $C \in \C(\eta)$, a contradiction. The tree $T^A$ is therefore degenerate, $T^A = \{v\}$ where $v \in T$ is the unique vertex for which $\Stab_T(v)=A$. Any ray in $\tau$ therefore defines an element of $\bdy F - \bdy A$, but the element defined is not in $\union\D$, a contradiction. It follows that $\tau$ must be degenerate, $\tau=\{v\}$ and $\D=\D(v)$ for some natural vertex $v \in T$, and the proof of Lemma \ref{LemmaFSSimplices} is~complete.

\subparagraph{Remark.} For any free splitting $F \act S$, any self-conjugacy $f \from S \to S$ restricts to the identity map on the vertex set of $S$, because $f$ maps each natural edge $\eta \subset S$ to itself preserving orientation. This is true because, as shown at the beginning of the proof of the corollary, if $\eta \ne \eta' \subset S$ are natural edges then $\C(\eta) \ne \C(\eta')$, and $\C_-(\eta) \ne \C_+(\eta)$.

\section{Fold paths}
\label{SectionFoldPaths}
We define the class of fold paths between vertices of $\FS'(F)$, using a method pioneered by Stallings \cite{Stallings:folding} for factoring maps of graphs into products of folds. This method was extended to the category of group actions on trees by Bestvina and Feighn \cite{BestvinaFeighn:bounding}. We refer to the latter paper for some details, although these details are considerably simplified in the category of free splittings.

\subsection{Directions, gates, and foldable maps} 
\label{SectionFoldableMaps}
First we set up some of the basic definitions which are used throughout the paper. We will also prove a tree-theoretic version of the First Derivative Test, Lemma~\ref{LemmaFDT}.

Given any graph $X$ and a vertex $v \in X$, the set of \emph{directions of $X$ at $v$},\index{direction} denoted $D_v X$, is defined to be the set of germs of oriented arcs in $X$ with initial vertex~$v$. Each direction at $v$ is uniquely represented by an oriented edgelet with initial vertex $v$. The union of the sets $D_v X$ over all vertices $v \in X$ is denoted $DX$. Given a subgraph $Y \subset X$, the subset of $DX$ represented by oriented edgelets $e \subset X \setminus Y$ having initial vertex in $Y$ is denoted $D_Y X$.

Given two free splittings $F \act S,T$ and a map $f \from S \to T$, the \emph{derivative}\index{derivative} of $f$ is a partially defined map $df \from DS \to DT$ whose domain is the set of directions of oriented edgelets $e$ on which $f$ is nonconstant, and whose value on the direction of $e$ is the direction of the oriented edgelet $f(e)$. Given a subgraph $W \subset S$, if $f$ is nonconstant on each edgelet representing a direction in the set $D_W S$ then we obtain by restriction a map $d_W f \from D_W S \to DT$; as a special case, when $W=\{v\}$ is a vertex we obtain a map $d_v f \from D_v S \to D_{f(v)} T$. 

Suppose now that the map $f \from S \to T$ is nonconstant on all edgelets of $S$, so $df \from DS \to DT$ has full domain of definition. For each vertex $v \in S$ the set $D_v S$ partitions into \emph{gates}\index{gate} which are the nonempty subsets of the form $(d_v f)^\inv(\delta)$ for $\delta \in D_{f(v)} T$. Every gate is a finite set, indeed we have:

\begin{lemma}\label{LemmaGateBound}
For any free splittings $F \act S,T$, for any map $f \from S \to T$ which is nonconstant on each edgelet of $S$, and for any vertex $v \in S$, the cardinality of each gate of $D_v S$ is $\le 2 \rank(F)$.
\end{lemma}

\begin{proof}
Let $e_1,\ldots,e_M \subset S$ be oriented edgelets with initial vertex $v$ representing a gate of $D_v S$. These oriented edgelets are all in distinct orbits under the action of $F$, for otherwise their common image in $T$ would have a nontrivial stabilizer. It follows that in the quotient graph of groups $S/F$, the quotients of $e_1,\ldots,e_M$ represent $M$ distinct directions at the quotient of $v$. It therefore suffices to bound the valence of each vertex in the quotient graph of groups of a free splitting. Without decreasing the valence at the quotient of $v$, one can blow up all other vertex orbits so that the only vertex orbit with nontrivial stabilizers is the orbit of $v$. Then, still without decreasing quotient valence, one can inductively collapse natural edges whose endpoints are in different vertex orbits. When this process stops, the quotient graph of groups is a rose with one natural vertex (possibly having nontrivial vertex group) and with $\le \rank(F)$ edges, whose natural vertex has valence $\le 2 \rank(F)$.
\end{proof}

\begin{definition}[Foldable maps and edgelets]
\label{DefinitionFoldableMaps}
A map $f \from S \to T$ is \emph{foldable}\index{map!foldable}\index{foldable map} if it satisfies either of the following two equivalent statements:
\begin{description}
\item[Natural edge definition of foldable:] $f$ is injective on each natural edge of $S$ and $f$ has $\ge 3$ gates at each natural vertex of $S$.
\item[Edgelet definition of foldable:] $f$ is injective on every edgelet, $f$ has $\ge 2$ gates at every vertex, and $f$ has $\ge 3$ gates at every natural vertex. 
\end{description}
We will without warning switch between these two definitions whenever it is convenient. Notice that the restrictions on the number of gates are significant only at vertices of finite valence, because every gate is a finite set; for example, if every natural vertex of $S$ has nontrivial stabilizer then every map defined on $S$ which is injective on natural edges is foldable. Notice also that foldability of $f$ depends only on the natural cell structures on $S$ and $T$, not on the given simplicial structures; to put it more formally, foldability is an invariant of $f$ in the category of equivariant continuous functions between free splittings of $F$. 

Given free splittings $F \act S,T$, a foldable map $f \from S \to T$, and an edgelet $e \subset T$, an \emph{$e$-edgelet of $f$}\index{edgelet!of a foldable map} is an edgelet of $S$ that is mapped to $e$ by~$f$.

In Lemma~\ref{LemmaFoldableExistence} below we shall prove the existence of foldable maps in the appropriate context.
\end{definition}

\subparagraph{Remark.} In other treatments of Stallings folds we have not seen any analogue of our gate~$\ge 3$ condition on natural vertices. This condition is crucial to the diameter bound obtained in Lemma~\ref{LemmaBROneB}, as well as in the heart of the proof of the Main Theorem, particularly in the proof of Proposition~\ref{PropPushdownInToto}, Step 3.

\subparagraph{The First Derivative Test.} The first derivative test of calculus implies that if the derivative of a function has no zeroes then local extreme values occur only at endpoints of the domain.

\begin{lemma}[The First Derivative Test]\label{LemmaFDT}
Suppose that $f \from S \to T$ is a foldable map of free splittings. Given a connected subgraph $W \subset S$ and a vertex $v \in W$, if $f(v)$ has valence~1 in the subgraph $f(W) \subset T$ then $v$ is a frontier point of W.
\end{lemma}

\begin{proof} If $v$ is an interior point of $W$ then $D_v W = D_v S$, and since $f$ has $\ge 2$ gates at~$v$ it follows that $d_v f(D_v W)$ has cardinality $\ge 2$, implying that $f(v)$ has valence~$\ge 2$ in~$f(W)$.
\end{proof}

\subsection{Construction of foldable maps} 
Given free splittings $F \act S,T$, a fold path from $S$ to $T$ will be defined by factoring a foldable map $S \mapsto T$. Although a foldable map does not always exist, one will exist after moving $S$ a distance at most~2 in $\FS'(F)$.

\begin{lemma}
\label{LemmaFoldableExistence}
For any free splittings $F \act S, T$ there exist free splittings $S',S''$ and a foldable map $S'' \mapsto T$ such that $S \expandsto S' \collapsesto S''$.
\end{lemma}

\begin{proof} Fix the free splitting $F \act T$. Given a free splitting $F \act R$, let $\M(R,T)$ denote the set of all equivariant continuous functions $f \from R \to T$ that take each natural vertex of $R$ to a vertex of $T$ and whose restriction to each natural edge of $R$ is either injective or constant. It follows that $f$ is a map with respect to the \emph{pullback} simplicial structure on $R$ whose vertex set consists of all points that map to vertices of $T$ and that are not in the interior of a natural edge of $R$ that is collapsed by $f$. The edges of this simplicial structure on $R$ will be referred to as \emph{pullback edgelets of $f$}.

Choose any expansion $S \expandsto S'$ so that $F \act S'$ is properly discontinuous, which implies that the set $\M(S',T)$ is nonempty. Amongst all elements of $\M(S',T)$ choose $f \from S' \to T$ to maximize the number of orbits of natural edges of $S'$ on which $f$ is constant. By collapsing each such natural edge we define a collapse map $S' \mapsto S''$ and an induced function which is an element of the set $\M(S'',T)$. By maximality of $f$ it follows that \emph{any} element of $\M(S'',T)$ is injective on each natural edge of $S''$, for otherwise by composing the collapse map $S' \mapsto S''$ with an element of $\M(S'',T)$ that collapses some natural edge of $S''$ we obtain an element of $\M(S',T)$ that collapses a larger number of natural edge orbits than $f$ does, a contradiction.

We find a foldable element of $\M(S'',T)$ by solving optimization problems. First we prove that if $g \in \M(S'',T)$ minimizes the number of orbits of pullback edgelets then $g$ has $\ge 2$ gates at each vertex of $S''$. Suppose there is a vertex $v \in S''$ at which $g$ has only~$1$ gate. Let $K$ be the valence of $v$; note that $K \ge 3$ because $g$ is injective on natural edges. Let $\eta_1,\ldots,\eta_K$ be the oriented natural edges of $S''$ with initial vertex~$v$. Let $e_1,\ldots,e_K$ be the initial pullback edgelets of $\eta_1,\ldots,\eta_K$, and let $w_1,\ldots,w_K$ be the terminal endpoints of $e_1,\ldots,e_K$, respectively. We have $f(e_1)=\cdots=f(e_K)=e$ for some oriented edge $e \subset T$ with initial vertex $f(v)$ and opposite vertex $w=f(w_1)=\ldots=f(w_K)$. Consider first the case that $e_i \ne \eta_i$ for each $i$, and so we can isotope each restricted map $g \restrict \eta_i$ by pushing $g(v)$ across $e$ to $w$ by an isotopy supported in a neighborhood of $e_i$, and we can extend these isotopies to an equivariant homotopy of $g$, to produce an element of $\M(S'',T)$ that has $K$ fewer orbits of pullback edgelets than $g$ has, a contradiction. Consider next the case that $e_i=\eta_i$ for certain values of $i=1,\ldots,K$. If $v,w_i$ are in distinct $F$-orbits for each such $i$ then we can equivariantly homotope $g$, pushing $g(v)$ across $e$ to $w$, so as to collapse each $e_i$ for which $e_i=\eta_i$, to produce an element of $\M(S'',T)$ that collapses each of the natural edges $\eta_i$ such that $e_i=\eta_i$, a contradiction. In the remaining case there exists some $i=1,\ldots,K$ such that $e_i=\eta_i$ and $w_i=\gamma(v)$ for some $\gamma \in F$, and it follows that $w=\gamma(v)$. The edges $e_i \subset S''$ and $e \subset T$ are therefore fundamental domains for the actions of $\gamma$ on its axes in $S''$, $T$, respectively. It follows that the direction of $\gamma^\inv(e_i)$ at $v$ maps to the direction of $\gamma^\inv(\eta)$ at $g(v)$ which is \emph{not equal to} the direction of $\eta$ at $g(v)$, contradicting that $g$ has a single gate at $v$.

Next we prove that among all $g \in \M(S'',T)$ that minimize the number of orbits of pullback edges, there is at least one which is foldable, having $\ge 3$ gates at each natural vertex. This is achieved, mostly, by solving another optimization problem. Define the \emph{edgelet vector} of $g$ to be the vector of positive integers $L_g$ indexed by the natural edge orbits of $S$, whose entry $L_g(e)$ corresponding to a natural edge $e \subset S$ is the number of pullback edgelets in $e$. Define $\Length(L_g)$ to be the sum of its entries, which equals the number of pullback edgelet orbits of $g$, a number which has already been minimized so as to guarantee~$\ge 2$ gates at each vertex. Define $\Energy(L_g)$ to be the sum of the squares of its entries. We have the inequality $\Energy(L_g) \le (\Length(L_g))^2$. Amongst all $g \in \M(S'',T)$ with minimal value of $\Length(L_g)$, choose $g$ so as to maximize $\Energy(L_g)$.

We claim that with energy maximized as above, one of the following holds:
\begin{enumerate}
\item \label{ItemNotHoop}
$g$ has $\ge 3$ gates at each natural vertex, and so $g$ is foldable.
\item \label{ItemHoop}
$S''$ has exactly one natural vertex orbit, $g$ has two gates at every natural vertex, and each natural edge of $S''$ has its two directions lying in distinct gate orbits. 
\end{enumerate}
To prove this dichotomy, suppose that $g$ has exactly two gates at some natural vertex~$v$. The gates must have the same cardinality: otherwise, by doing a valence~2 homotopy, pushing $g(v)$ across one edge of $T$ in the image direction of the larger of the two gates at~$v$,  one reduces the total number of pullback edgelets. Now consider $g_1,g_2 \in \M(S'',T)$ defined by the two possible valence~2 homotopies at $v$, pushing $g(v)$ across the two edges of $T$ in the two image directions of the two gates at $v$. Note that the average of the two vectors $L_{g_1}, L_{g_2}$ is the vector $L_g$. It follows that $L_g = L_{g_1} = L_{g_2}$, for otherwise, by convexity of energy, one of $\Energy(L_{g_1})$ or $\Energy(L_{g_2})$ would be larger than $\Energy(g)$. It also follows that $S''$ has exactly one natural vertex orbit, for otherwise $v$ would be connected across a natural edge $e$ to some natural vertex in a different orbit, implying that one of $L_{g_1}(e)$, $L_{g_2}(e)$ equals $L_g(e)+1$ and the other equals $L_g(e)-1$. It also follows that each natural edge $e$ has one end in the orbit of one gate at $v$ and opposite end in the orbit of the other gate at $v$, for otherwise one of $L_{g_1}(e)$, $L_{g_2}(e)$ would equal $L_g(e)+2$ and the other equals $L_g(e)-2$. This shows that $g$ satisfies item~\pref{ItemHoop}.

\medskip

To finish up we show that if $g$ satisfies \pref{ItemHoop} then there exists $g' \in \M(S'',T)$ which satisfies~\pref{ItemNotHoop}. Item~\pref{ItemHoop} implies that there is an orientation of the natural edges of $S''$ such that at each natural vertex $v \in S''$, the directions with initial vertex $v$ form one gate of $g'$ at $v$ denoted $D^+_v$, and the directions with terminal vertex $v$ form the other gate denoted~$D^-_v$. 

Pick a natural vertex $v \in S''$. Let $\tau$ be the subtree of $S''$ consisting of the union of all oriented rays in $S''$ with initial vertex $v$. The restriction of $g$ to each such ray is injective and proper, and their initial directions all map to the same direction in $T$, so it follows that the subtree $g(\tau) \subset T$ has a valence~$1$ vertex at $g(v)$ and no other valence~$1$ vertex. Also, if we orient each edge of $g(\tau)$ to point away from the vertex $g(v)$ then the map $g \from \tau \to g(\tau)$ preserves orientation. Furthermore $g(\tau)$ is not itself just a ray, for if it were then $T$ would be just a line, an impossibility for a free splitting of a free group of rank~$\ge 2$. Let $w \in g(\tau)$ be the vertex of $g(\tau)$ of valence~$\ge 3$ which is closest to $g(v)$. Define $g' \from S'' \to T$ by mapping $v$ to $w$, extending equivariantly to the orbit of $v$, and extending equivariantly to an embedding on each edge of~$S''$. 

We claim that $g'$ has one gate at $v$ corresponding to each direction of $g(\tau)$ at $w$, which implies that $g'$ is foldable. To see why, first note that the set $D^-_v$ is mapped by $d_v g'$ to the unique direction of the segment $[w,g(v)]$ at $w$. Next note that each direction in the set $D^+_v$ is mapped by $d_v g'$ to one of the directions of $T$ at $w$ distinct from the direction of $[w,g(v)]$; furthermore each such direction is in the image of $d_v g'$ because $g'$ maps $\tau$ onto $f(\tau) \setminus [w,g(v)]$ by an orientation preserving map. 

This completes the proof of Lemma~\ref{LemmaFoldableExistence}.
\end{proof}

\subsection{Folds}
\label{SectionFoldFactorizations}
Given free splittings $F \act S,T$ and a foldable map $f \from S \to T$, we say that $f$ is a \emph{fold}\index{map!fold}\index{fold}\index{fold map} if there exist oriented natural edges $\eta,\eta' \subset S$ with the same initial vertex $v$, and there exist nondegenerate initial segments $e \subset \eta$, $e' \subset \eta'$ which are subcomplexes of~$S$ with the same positive number of edgelets, such that if we let $\phi \from e \to e'$ denote the unique orientation preserving simplicial isomorphism, then for all $x \ne x' \in S$ we have $f(x)=f(x')$ if and only if there exists $\gamma \in F$ such that (up to interchanging $x,x'$) $\gamma\cdot x \in e$ and $\phi(\gamma \cdot x)=\gamma \cdot x' \in e'$. We also say that the map $f$ \emph{folds the segments $e$ and~$e'$}. 

The pair of segments $e,e'$ determines the free splitting $F \act T$ up to simplicial conjugacy, namely $F \act T$ is conjugate to the equivariant quotient complex of $S$ obtained by equivariantly identifying $e$ and $e'$ via $\phi \from e \to e'$. In this context we shall say that the free splitting $T$ is determined by \emph{folding the segments $e,e'$}. 
Letting $d,d' \in D_v S$ denote the initial directions of $e,e'$ respectively, we also say that $f$ \emph{folds the directions $d,d'$}, although $d,d'$ do not determine the segments $e,e'$ and they need not determine $T$ up to conjugacy. Notice that $d,d'$ are in different orbits under the action $\Stab_S(v) \act D_v S$ (equivalently under the action $F \act D S$), for otherwise the segment $f(e)=f(e') \subset T$ would have nontrivial stabilizer.
Folds are classified according to the properness of the inclusions $e \subset \eta$, $e' \subset \eta'$, as follows. If $e,e'$ are both proper initial segments of $\eta,\eta'$ then we say that $f$ is a \emph{partial} fold; otherwise $f$ is a \emph{full fold}.\index{fold!full} If $f$ is a full fold and exactly one of $e,e'$ is proper then we say that $f$ is a \emph{proper} full fold;\index{fold!full!proper} otherwise, when $e=\eta$ and $e'=\eta'$, we say that $f$ is an \emph{improper} full fold.\index{fold!full!improper} For later purposes we note that if $f$ is a full fold then every natural vertex of $T$ is the image of a natural vertex of $S$; and even when $f$ is a partial fold, every natural vertex of $T$ which is not in the orbit of the image of the terminal endpoints of the folded edges $e,e'$ is the image of a natural vertex of $S$.

In the terminology of \cite{BestvinaFeighn:bounding}, folds between free splittings can also be classified into two types as follows. If the opposite vertices $w,w'$ of $e,e'$ are in different $F$-orbits one gets a type IA fold;\index{fold!type IA} in this case the stabilizer of the vertex $W=f(w)=f(w')$ is the subgroup generated by the stabilizers of $w,w'$, which (if nontrivial) is a free factor whose rank is the sum of the ranks of the stabilizers of $w$ and $w'$. If $w,w'$ are in the same $F$-orbit then one gets a type IIIA fold,\index{fold!type IIIA} and the stabilizer of the vertex $W$ is the subgroup generated by the stabilizer of $w$ and an element $\gamma \in F$ such that $\gamma(w)=w'$, which is a free factor whose rank is $1$ plus the rank of the stabilizer of $w$. Notice that a type IIIA fold is only possible if $f$ is a partial fold or an improper full fold, because a natural and an unnatural vertex can never be in the same orbit. We refer to \cite{BestvinaFeighn:bounding} for an understanding of the map on quotient graphs of groups $S/F \to T/F$ which is induced by a fold $f \from S \to T$.

The following lemma and its proof are well known in the narrower context of the first barycentric subdivision of the spine of outer space. 

\begin{lemma} 
\label{LemmaFoldDistance}
For any fold $f \from S \to T$, the distance in $\FS'(F)$ from $S$ to $T$ equals $1$ or~$2$.
\end{lemma}

\begin{proof} Let $f$ fold oriented segments $e,e'$ with common initial endpoint $v$ and opposite endpoints $w,w'$. After possibly subdividing $S$ and $T$ so that $e,e'$ each contain $\ge 2$ edgelets, the map $f$ can be factored into two maps as $S \xrightarrow{g} U \xrightarrow{h} T$, where $g$ folds the initial edgelets $e_0 \subset e$, $e'_0 \subset e'$, and $h$ folds the $g$-images of the terminal segments $e_1 = e \setminus e_0$, $e'_1 = e' \setminus e'_0$. Letting $\hat e = g(e_0)=g(e'_0) \subset U$ and $\sigma_0 = F \cdot \hat e \subset U$, resubdividing $S$ there is an expansion $S \expands U$ defined by a collapse map $U \xrightarrow{[\sigma_0]} S$. Also, letting $\sigma_1 = F \cdot (g(e_1) \union g(e'_1)) \subset U$, after resubdividing $T$ there is a collapse $U \collapsesto T$ defined by a collapse map $U \xrightarrow{[\sigma_1]} T$. It follows that $d(S,T) \le 2$ in $\FS'(F)$.

It remains to show that $d(S,T) \ne 0$, that is, $S,T$ are not conjugate free splittings. Since each fold map is foldable, the natural vertex $v$ has~$\ge 3$ gates with respect to $f$. It therefore has $\ge 3$ gates with respect to $g$, and so $g(v) \in U$ is natural. It follows that $\hat e$ is a natural edge of $U$, having one endpoint at $g(v)$ and opposite endpoint of valence~$3$ in $U$. The subgraph $\sigma_0 \subset U$ is therefore natural, and it follows from Lemma~\ref{LemmaFSSimplices} that $S$ is not conjugate to $U$. The free splittings $U,T$ may or may not be conjugate, depending on whether at least one of $g(e_1), g(e_2) \subset U$ is a natural edge. If neither of $g(e_1)$, $g(e_2)$ is natural then $T$ is conjugate to $U$, and so $T$ is not conjugate to $S$. If one or both of $g(e_1)$, $g(e_2)$ is natural then (after resubdividing $T$) the collapse $U \collapsesto T$ may also defined by collapsing the natural subgraph $\hat\sigma_1 \subset U$ which is the union of the $F$ orbits of whichever of $g(e_1)$, $g(e_2)$ is natural; but $\sigma_0 \ne \hat\sigma_1$ and so by Lemma~\ref{LemmaFSSimplices} we conclude that $S,T$ are not conjugate.
\end{proof}

\subsection{Fold sequences and fold paths} 
Consider free splittings $F \act S,T,U$ and a sequence of maps of the form $S \xrightarrow{h} U \xrightarrow{g} T$. Letting $f = g \composed h \from S \to T$, we say that $h$ is a \emph{maximal fold factor of $f$} if the following hold: $h$ is a fold map that folds oriented initial segments $e,e' \subset S$ of oriented natural edges $\eta,\eta' \subset S$, respectively, and $e,e'$ are the maximal initial subsegments of $\eta,\eta'$ such that in $T$ we have $f(e)=f(e')$. Recall from the definition of a fold that $e,e'$ are edgelet paths with the same number of edgelets.

\subparagraph{Fold sequences.} Consider a sequence of free splittings and maps of the form $S_0 \xrightarrow{f_1} S_1 \xrightarrow{f_2}\cdots\xrightarrow{f_K} S_K$, $K \ge 0$. In this context we will always denote 
$$f^i_j = f_j \composed \cdots \composed f_{i+1} \from S_i \to S_j, \quad\text{for}\quad 0 \le i < j \le K.
$$
We say that this is a \emph{fold sequence}\index{fold sequence} if the following holds:
\begin{enumerate}
\item \label{ItemOuterFoldable}
$f^0_K \from S_0 \to S_K$ is a foldable map.
\item \label{ItemConjOrMaxFold} 
Each map $f_{i+1} \from S_{i} \to S_{i+1}$ is a maximal fold factor of the map $f^i_K \from S_{i} \to S_K$, for $0 \le i < K$.
\end{enumerate}

It follows from \pref{ItemOuterFoldable} and \pref{ItemConjOrMaxFold} that
\begin{enumeratecontinue}
\item \label{ItemEachFoldable}
$f^i_j \from S_i \to S_j$ is a foldable map for each $0 \le i < j \le K$.
\end{enumeratecontinue}
To prove \pref{ItemEachFoldable}, starting from the base assumption \pref{ItemOuterFoldable}, and assuming by induction that $f^{i-1}_K = f^i_K \composed f_i$ is foldable, we prove that $f^i_K$ is foldable. By \pref{ItemConjOrMaxFold} the map $f_i$ is a maximal fold factor of $f^{i-1}_K$. The map $f^{i-1}_K$ is injective on each edgelet of $S_{i-1}$, and each edgelet of $S_i$ is the $f_i$ image of some edgelet of $S_{i-1}$, so $f^i_K$ is injective on each edgelet. Consider a vertex $v \in S_i$ and a vertex $u \in S_{i-1}$ for which $f_i(u)=v$. The number of $f^i_K$-gates at $v$ is greater than or equal to the number of $f^{i-1}_K$ gates at $u$ which is $\ge 2$, and furthermore if $u$ is natural then this number is $\ge 3$. This covers all cases except for when $v$ is natural and each such $u$ has valence~$2$. Since $f_i$ is a maximal fold factor of $f^{i-1}_K$, this is only possible if $f$ is a partial fold that folds segments $e,e' \subset S_{i-1}$ such that if $w,w'$ denote the terminal endpoints of $e,e'$ then $v=f_i(w)=f_i(w')$. If $f_i$ is a type IA fold, that is if $w,w'$ are in different orbits, then $v$ has valence~$3$, and by maximality of the fold $f_i$ it follows that the three directions at $v$ are all in different gates with respect to $f^i_K$. If $f_i$ is a type IIIA fold, that is if $w,w'$ are in the same orbit, say $\gamma \cdot w = w'$ for a nontrivial $\gamma \in F$, then $\Stab_{S_i}(v)$ contains $\gamma$ and so is nontrivial, and hence $v$ has infinitely many gates with respect to~$f^i_K$. This proves by induction that each $f^i_K$ is foldable. Next, to prove that $f^i_j$ is foldable, given a vertex $v \in S_i$ we simply note that the decomposition of $D_v S_i$ into $f^i_j$-gates is a refinement of the decomposition into $f^i_K$ gates, of which there are $\ge 2$, and $\ge 3$ if $v$ is natural. This completes the proof that \pref{ItemOuterFoldable} and \pref{ItemConjOrMaxFold} imply~\pref{ItemEachFoldable}.

In this proof we have shown the following fact which will be useful in Lemma~\ref{LemmaFoldSequenceConstruction} below when we construct fold sequences:

\begin{lemma}\label{LemmaMaxFactor} For any foldable map $S \xrightarrow{f} T$ and any factorization of $f$ into two maps of the form $S \xrightarrow{k} U \xrightarrow{g} T$, if $k$ is a maximal fold factor of~$f$ then the map $g \from U \to T$ is also foldable. 
\qed\end{lemma}

The implication of this lemma is \emph{false} if one allows $k$ to be a partial fold which is not a maximal fold factor of $f$, for in that case the map $g \from U \to T$ will have only 2~gates at the valence~$3$ vertex which is the $k$-image of the terminal endpoints of oriented segments $e,e'$ that are folded by $k$.

\subparagraph{Fold paths.} A \emph{fold path}\index{fold path} in $\FS'(F)$ is any sequence of vertices represented by free splittings $F \act S_0,S_1,\ldots,S_K$ for which there exists a fold sequence $S_0 \mapsto S_1 \mapsto\cdots\mapsto S_K$; we also say that this fold path has \emph{$K$-steps}.

Strictly speaking a fold path need not be the sequence of vertices along an actual edge path in the simplicial complex $\FS'(F)$, because the size of the step from $S_{i-1}$ to $S_i$ is either $1$ or~$2$; see Lemma~\ref{LemmaFoldDistance}. If one so desires one can easily interpolate the gap between $S_{i-1}$ and $S_i$ by an edge path of length~$1$ or $2$, to get an actual edge path from $S_0$ to $S_K$. 

We define two fold sequences to be \emph{equivalent}\index{fold sequence!equivalence} if they have the same length and there is a commutative diagram of the form
$$\xymatrix{
S_0 \ar[r] \ar[d] & S_1 \ar[r] \ar[d] & \cdots \ar[r] & S_{K-1} \ar[r] \ar[d] & S_K \ar[d] \\
S'_0 \ar[r]          & S'_1 \ar[r]         & \cdots \ar[r] & S_{K'-1} \ar[r]          & S'_K
}$$
where the top and bottom rows are the two given fold sequences and each vertical arrow is a conjugacy. Note that the vertical arrows are \emph{not} required to be ``maps'' as we have defined them, in that they need not be simplicial. For example, if the bottom row is obtained by taking the $400^{\text{th}}$ barycentric subdivision of each 1-simplex in the top row then the two fold sequences are equivalent.

Equivalent fold sequences determine the same fold path, but the converse is false. A counterexample consisting of a $1$-step fold path is given at the end of this section.

\subparagraph{Construction of fold factorizations.} Having constructed many foldable maps in Lemma~\ref{LemmaFoldableExistence}, to construct many fold paths it suffices to factor each foldable map as a fold sequence.

Given free splittings $F \act S,T$ and a foldable map $S \xrightarrow{f} T$, a \emph{fold factorization}\index{fold factorization} of $f$ is any fold sequence $S_0 \mapsto S_1 \mapsto\cdots\mapsto S_K$ which factors $f$ as shown in the following commutative diagram:
$$\xymatrix{
S \ar@{=}[r] \ar@/^2pc/[rrrrr]^{f} & S_0 \ar[r]^{f_1} & S_1 \ar[r]^{f_2} & \cdots \ar[r]^{f_K} & S_K \ar@{=}[r] & T
}$$
A fold factorization of any foldable map can be constructed by an inductive process described in \cite{BestvinaFeighn:bounding}, with considerable simplification arising from the fact that all edgelet stabilizers are trivial in $T$. We give this simplified argument here. 

\begin{lemma}
\label{LemmaFoldSequenceConstruction}
For any free splittings $F \act S,T$, every foldable map $f \from S \mapsto T$ has a fold factorization.
\end{lemma}

\begin{proof} If $f$ is a simplicial isomorphism then we are done, with a fold factorization of length~$K=0$. Otherwise, we use the following obvious but key fact:
\begin{description}
\item[Local to global principle:] Any simplicial map between trees which is locally injective is globally injective. If furthermore it is surjective then it is a simplicial isomorphism.
\end{description}
\noindent
For the inductive step we show that every foldable map $S \xrightarrow{f} T$ which is not a homeomorphism factors into maps as $S \xrightarrow{k} U \xrightarrow{g} T$ where $k$ is a maximal fold factor of~$f$. By the \emph{Local to global principle}, plus the fact that $F \act T$ is minimal, it follows that $f$ is surjective and so $f$ is not locally injective. We may therefore find a vertex $v \in S$ and two directions $d,d' \in D_v S$ such that $d_v f(d)=d_v f(d')$. Let $\eta,\eta'$ be the oriented natural edges with initial vertex~$v$ and initial directions $d,d'$. Let $e \subset \eta$, $e' \subset \eta'$ be the maximal initial segments such that $f(e)=f(e')$. Noting that $e,e'$ are subcomplexes with the same number of edgelets, let $h \from e \to e'$ be the unique orientation preserving simplicial homeomorphism. Define $k \from S \to U$ to be the quotient map obtained by equivariantly identifying $e$ and $e'$, and let $g \from U \to T$ be the induced map. As indicated in \cite{BestvinaFeighn:bounding}, $U$ is a tree and the induced action $F \act U$ is minimal. The map $k$ is simplicial by construction, from which it follows that $g$ is simplicial as well. The stabilizer of each edgelet of $U$ is trivial because it is contained in the stabilizer of its image in $T$ under $g$ which is trivial, and so $F \act U$ is a free splitting. By construction the map $k \from S \to U$ is a maximal fold factor of the foldable map $f$. 

To support the inductive step we must prove that $U$ has fewer edgelet orbits than~$S$, which follows from the fact that the initial edgelets of $e$ and $e'$ are in different orbits of the action $F \act S$, because they have the same image edgelet in $T$ and its stabilizer is trivial. 

The fold factorization of $f=f^0_T \from S=S_0 \to T$ may now be constructed as follows. Assuming $f^0_T$ is not locally injective, factor $f^0_T$ into maps as $S_0 \xrightarrow{f_1} S_1 \xrightarrow{f^1_T} T$ where $f_1$ is a maximal fold factor of $f^0_T$. The induced map $f^1_T$ is foldable by Lemma~\ref{LemmaMaxFactor}, and the number of edgelet orbits of $S_1$ is smaller than the number of edgelet orbits of $S_0$. The process therefore continues by induction on the number of edgelet orbits, stopping at $S=S_0 \xrightarrow{f_1} S_1 \xrightarrow{f_2}\cdots\xrightarrow{f_K} S_K \xrightarrow{f^K_T} T$ when $f^K_T$ is locally injective and therefore a simplicial conjugacy, and we identify $S_K=T$.
\end{proof}

\paragraph{Remark.} The \emph{Local to global principle} may be used to construct fold factorizations with various special properties. In particular, if $\beta \subset S$ is a subtree on which $f$ is not locally injective then we may choose the folded edges $\eta,\eta'$ to lie in $\beta$. This is used in the proof of Lemma~\ref{LemmaBRNatural}.

\paragraph{Counterexample: inequivalent folds.} \qquad\qquad  We describe two inequivalent folds $\ti f, \ti f' \from S_0 \to S_1$ that determine the same $1$~step fold path $S_0,S_1$ in $\FS'(F)$. Both of the actions $F \act S_0,S_1$ are properly discontinuous. We first describe the quotient marked graphs $G_0 = S_0/F$, $G_1=S_1/F$ and the induced homotopy equivalences $f, f' \from G_0 \to G_1$. The marked graph $G_0$ has a valence~$4$ vertex $v$ with the following incident directions: directed natural edges $a,b$ with initial vertex $v$, and a directed natural edge $c$ with initial and terminal vertex $v$; subject to this description, $G_0$ is then filled out to be a marked graph in an arbitrary manner. The marked graph $G_1$ is defined to have the same underlying unmarked graph as $G_0$. The homotopy equivalences $f,f' \from G_0 \to G_1$ are defined so that $f(a)=ca$, $f'(b) = c^\inv b$, and $f,f'$ are the identity elsewhere. Clearly $f,f'$ are homotopic, by a homotopy which spins the $c$ loop once around itself and is stationary on $G_0 \setminus (a \union b \union c)$. The marking on $G_1$ is defined by pushing forward the marking on $G_0$ via either of $f,f'$, and so each of $f,f'$ preserves marking.  Consider the universal covering maps $S_i \mapsto G_i$, $i=0,1$. We may choose $F$-equivariant lifts $\ti f, \ti f' \from S_0 \to S_1$ which are the two fold maps at issue. If they were equivalent then, since any self-conjugacy of $S_0$ or of $S_1$ fixes each vertex and each oriented natural edge (see the \emph{Remark} at the end of Section~\ref{SectionFreeSplittingComplex}), each direction in $D S_0$ would have the same image in $D S_1$ under $d \ti f$ and $d \ti f'$. However, fixing a lift $\ti v$ and lifts $\ti a, \ti b, \ti c$ of $a,b,c$ with initial vertex $\ti v$ and a lift $\ti c'$ of $c$ with terminal vertex $\ti v$, we have $d \ti f(\ti a) = d \ti f(\ti c)$ but $d \ti f'(\ti a) \ne d \ti f'(\ti c)$.

\section{The Masur-Minsky axioms}
\label{SectionMasurMinsky}
Our proof that $\FS(F)$ is hyperbolic uses the axioms introduced by Masur and Minsky \cite{MasurMinsky:complex1} for their proof that the curve complex of a finite type surface is hyperbolic. The axioms require existence of a family of paths which satisfy a strong projection property. For this purpose we shall use fold paths: Proposition~\ref{PropFoldContractions} stated at the end of this section says, roughly speaking, that fold paths in $\FS'(F)$ satisfy the Masur-Minsky axioms. 

First we give an intuitive explanation of the content of Proposition~\ref{PropFoldContractions} by giving an outline of the Masur-Minsky axioms as they would apply to fold paths. The axioms require that a map be defined which is a kind of projection from $\FS'(F)$ to each fold path $S_0, S_1, \ldots,  S_K$. To make things work the range of the projection is taken to be the parameter interval $[0,K]$ of the fold path, giving the projection map the form $\pi \from \FS'(F) \to [0,K]$. When one projects two vertices of $\FS'(F)$ to two parameters $l \le k \in [0,K]$, one is interested in the ``diameter (of the subpath) between these two parameters'', which means the diameter of the set $\{S_l, S_{l+1}, \ldots, S_k\}$ in $\FS'(F)$. There are three axioms. The \emph{Coarse Retraction} bounds the diameter between each $k \in [0,K]$ and its projection~$\pi(S_k) \in [0,K]$. The \emph{Coarse Lipschitz} axiom bounds the diameter between the projections $\pi(T),\pi(T') \in [0,K]$ of two nearby vertices $T,T' \in \FS'(F)$. The \emph{Strong Contraction} axiom says roughly that, for each metric ball in $\FS'(F)$ that stays a bounded distance away from the fold path, if one takes the sub-ball having a certain proportion of the total radius, the diameter between the projections of any two vertices in the subball is bounded. All the bounds occurring in this discussion must be uniform, depending only on the rank of~$F$.

In fact rather than using fold paths themselves, we use fold sequences. As we have seen in the counterexample at the end of Section~\ref{SectionFoldPaths}, the same fold path $S_0,\ldots,S_K$ can be represented by inequivalent fold sequences, and the projection maps $\FS'(F) \to [0,K]$ of these two fold sequences may be different. This kind of situation is handled formally be expressing the Masur-Minsky axioms in terms of ``families'' of paths which allow a path to occur repeatedly in the family.

\bigskip

Given integers $i, j$ we adopt interval notation $[i,j]$ for the set of all integers between $i$ and $j$ inclusive, regardless of the order of $i,j$. 

Consider a connected simplicial complex $X$ with the simplicial metric.  A \emph{path} in $X$ is just a finite sequence of $0$-simplices $\gamma(0),\gamma(1),\ldots,\gamma(K)$, which we write in function notation as $\gamma \from [0,K] \to X$. A \emph{family of paths} in $X$ is an indexed collection $\{\gamma_i\}_{i \in \I}$ of paths in $X$; we allow repetition in the family. A family of paths in $X$ is said to be \emph{almost transitive} if there exists a constant $A$ such that for any $0$-simplices $v,w$ there is a path $\gamma \from [0,K] \to X$ in the family such that all of the distances $d(v,\gamma(0))$, $d(\gamma(0),\gamma(1))$, \ldots, $d(\gamma(K-1),\gamma(K))$, $d(\gamma(K),w)$ are $\le A$. 

Given a path $\gamma \from [0,K] \to X$ and a function $\pi \from X \to [0,K]$, called the ``projection map'' to the path $\gamma$, and given constants $a,b,c > 0$, consider the following three axioms:
\begin{description}
\item[Coarse retraction:] For all $k \in [0,K]$ the set $\gamma[k,\pi(\gamma(k))]$ has diameter $\le c$.
\item[Coarse Lipschitz:] For all vertices $v,w \in X$, if $d(v,w) \le 1$ then the set $\gamma[\pi(v),\pi(w)]$ has diameter~$\le c$.
\item[Strong contraction:] For all vertices $v,w \in X$, if $d(v,\gamma[0,K]) \ge a$, and if $d(w,v) \le b \cdot d(v,\gamma[0,K])$, then the set $\gamma[\pi(v),\pi(w)]$ has diameter $\le c$.
\end{description}

\begin{theorem}[\cite{MasurMinsky:complex1}, Theorem 2.3]
Given a connected simplicial complex $X$, if there exists an almost transitive family of paths $\{\gamma_i\}_{i \in I}$ in $X$ and for each $i \in I$ a projection map $\pi_i \from X \to [0,K]$ to the path $\gamma_i \from [0,K] \to X$ such that the \emph{Coarse retraction}, the \emph{Coarse Lipschitz}, and the \emph{Strong contraction} axioms all hold with uniform constants $a,b,c>0$ for all $i \in I$, then $X$ is hyperbolic.
\end{theorem}

\subparagraph{Remarks.} Our notion of ``almost transitivity'' is not quite the same as ``coarse transitivity'' used in \cite{MasurMinsky:complex1}, which requires that the paths in the set be continuous and that there is a constant $D$ such that any two points at distance $\ge D$ are connected by a path in the set. However, the proof of equivalence of the two forms of the axioms, one with ``almost transitive'' and the other with ``coarse transitive'', is very simple, and is left to the reader. The set of fold paths in $\FS'(F)$ is almost transitive with constant $A = 2$: for any free splittings $S,T$, by moving $S$ a distance~$\le 2$ one obtains a naturally foldable morphism to $T$ (Lemma~\ref{LemmaFoldableExistence}), which has a fold factorization (Section~\ref{SectionFoldFactorizations}), and consecutive free splittings in such a factorization have distance~$\le 2$ (Lemma~\ref{LemmaFoldDistance}). 

The concept of a ``family of paths'' is left undefined in \cite{MasurMinsky:complex1} but the proof of the above theorem and the application to curve complexes given in \cite{MasurMinsky:complex1} clearly indicate that an indexed family with repetition is allowed. On top of that, given any indexed family satisfying the hypothesis of the theorem, if we removed repetition by kicking out all but one copy of each path then the resulting family would still satisfy the hypotheses of the theorem. In our situation, although we use fold paths in our application of the above theorem, we shall index them by (equivalence classes of) fold sequences; thus, we allow for the possibility that two inequivalent fold sequences representing the same fold path might have somewhat different projection maps.

\bigskip

Notice that the \emph{Strong contraction} axiom, unlike the \emph{Coarse Lipschitz} axiom, is not symmetric in the variables $v,w$. For our proof we shall need to extend the applicability of the \emph{Strong contraction} axiom by further desymmetrizing it:
\begin{description}
\item[Desymmetrized strong contraction:] For all vertices $v,w \in X$, if $\pi(w) \le \pi(v)$ in the interval $[0,K]$, if $d(v,\gamma[0,K]) \ge a$, and if $d(w,v) \le b \cdot d(v,\gamma[0,K])$, then the set $\gamma[\pi(v),\pi(w)]$ has diameter $\le c$.
\end{description}

\begin{lemma}\label{LemmaDesymmetrization}
For all constants $a,b,c > 0$ there exist constants $A,B > 0$ such that the \emph{desymmetrized strong contraction} axiom with constants $a$, $b$, and $c$ implies the \emph{strong contraction} axiom with constants $A$, $B$, and $C=c$.
\end{lemma}

\begin{proof} Set $A=4a$ and $B=\min\{1/4,3b/4\}$. We need only show that if $\pi(w) > \pi(v)$ in $[0,K]$, if $d(v,\gamma[0,K]) \ge A$ and if $d(w,v) \le B \cdot d(v,\gamma[0,K])$, then $d(w,\gamma[0,K]) \ge a$ and $d(v,w) \le b \cdot d(w,\gamma[0,K])$. We have
\begin{align*}
d(w,\gamma[0,K]) &\ge d(v,\gamma[0,K]) - d(w,v) \\
                             &\ge d(v,\gamma[0,K]) - \frac{1}{4} \cdot d(v,\gamma[0,K]) \\
                             &\ge \frac{3}{4} \cdot d(v,\gamma[0,K]) \ge 3a \ge a \\
\intertext{and}
d(v,w)                   &\le \frac{3}{4} \cdot b \cdot d(v,\gamma[0,K]) \\
                             &\le \frac{3}{4} \cdot b \cdot \frac{4}{3} d(w,\gamma[0,K]) = b \cdot d(w,\gamma[0,K])
\end{align*}
\end{proof}

We now define the path family $\{\gamma_i\}_{i \in \I}$ in $\FS'(F)$ that we use to prove the Main Theorem. Having associated to each fold sequence a fold path, which clearly depends only on the equivalence class of that fold sequence, the index set is defined to be the set of equivalence classes of fold sequences. 

To prove the Main Theorem, by combining the Masur--Minsky theorem, almost transitivity of fold paths, and Lemma~\ref{LemmaDesymmetrization}, it therefore suffices to prove:

\begin{proposition}\label{PropFoldContractions}
Associated to each fold sequence $S_0 \mapsto\cdots\mapsto S_K$ in $\FS'(F)$ there is a projection map $\pi \from \FS'(F) \to [0,K]$, depending only on the equivalence class of the fold sequence, such that the \emph{Coarse retraction}, the \emph{Coarse Lipschitz}, and the \emph{Desymmetrized strong contraction} axioms all hold, with constants $a,b,c$ depending only on $\rank(F)$.
\end{proposition}

The next step in the proof of the Main Theorem will be taken with the formulation of Proposition~\ref{PropProjToFoldPath}, where the projection maps are defined.

\subparagraph{Remark.} Theorem 2.3 of \cite{MasurMinsky:complex1} contains an additional conclusion, which in our context says that fold paths may be reparameterized to become uniform quasigeodesics in $\FS'(F_n)$, although the reparameterization does not fall out explicitly from their proof. Our method of proof will actually yield an explicit quasigeodesic reparameterization of fold paths, in terms of the ``free splitting units'' introduced in Section~\ref{SectionFSU}. See Proposition~\ref{PropFoldPathQuasis} for the statement and proof regarding this reparameterization.

\section{Combing}
\label{SectionCombing}
In this section we define a combing method for fold sequences. Roughly speaking, given a fold sequence $S_0 \mapsto \cdots \mapsto S_K$ and a free splitting $T'$ which differs from $S_K$ by a single edge in $\FS'(F)$, we want a construction which combs backwards to produce a fold sequence $T_0 \mapsto\cdots\mapsto T_K=T'$ in which each $T_k$ differs from the corresponding $S_k$ by at most a single edge in $\FS'(K)$. We would like to give this construction in two cases, depending on whether the oriented edge from $S_K$ to $T'$ is a collapse $S_K \collapse T'$ or an expansion $S_K \expand T'$. In the case of a collapse $S_K \collapse T'$ there is indeed a process of ``combing by collapse'' which produces a fold sequence as stated; see Proposition~\ref{PropCBC}. In the case of an expansion $S_K \expand T'$, although there is a process of ``combing by expansion'', the sequence $T_0 \mapsto\cdots\mapsto T_K=T'$ produced need not be a fold sequence, instead it belongs to a broader class of map sequences that we refer to as ``foldable sequences''; see Proposition~\ref{PropCBE}. It is an important part of our theory that both combing processes are closed on the collection of foldable sequences; combing by collapse is closed as well on the smaller collection of fold sequences.

In Section~\ref{SectionCombingRectangles} we define the collection of foldable sequences on which combing will be defined, and we define \emph{combing rectangles} which are the commutative diagrams of foldable sequences and collapse maps that are used to describe combing; see Figure~\ref{FigureCombingRectangle}. We also prove Lemma~\ref{LemmaCombingProperties} which says that combing by collapse is closed on foldable sequences. 

The two main combing processes --- combing by collapse, and combing by expansion --- are described in Section~\ref{SectionCombingConstructions}. In Section~\ref{SectionCombRectOps} we will also give some methods for constructing new combing rectangles by composing or decomposing old ones.

\bigskip

Also in Section~\ref{SectionCombingRectangles}, combing rectangles will be used to define the projection map from $\FS'(F)$ to each fold path $S_0 \mapsto\cdots\mapsto S_K$, and we will state Proposition~\ref{PropProjToFoldPath} which says that these projection maps satisfy the requirements of the Masur-Minsky axioms. 

Combing rectangles will be important structures for the rest of the paper. Much of the geometric intuition behind our methods involves visualizing combing rectangles and other, more complicated diagrams of free splittings and maps as objects sitting in the complex $\FS'(F)$, and visualizing various methods for geometrically manipulating these objects. The technical details of the proof of the Main Theorem will involve a calculus of combing rectangles, which is based on the constructions of combing rectangles given in Sections~\ref{SectionCombingConstructions} and~\ref{SectionCombRectOps}.

\subsection{Combing rectangles and the projection onto fold paths}
\label{SectionCombingRectangles}

\paragraph{Foldable sequences.} Consider a sequence of free splittings and maps over $F$ of the form $S_0 \xrightarrow{f_1} S_1 \xrightarrow{f_2}\cdots\xrightarrow{f_K} S_K$, and recall the notation $f^i_j = f_{i+1}\composed\cdots\composed f_j \from S_i \to S_j$ for each $0 \le i < j \le K$. This sequence is said to be a \emph{foldable sequence}\index{foldable sequence} over $F$ if for each $i=0,\ldots,K$ the map $f^i_K \from S_i \to S_K$ is a foldable map. It follows that each of the maps $f^i_j \from S_i \to S_j$ is a foldable map, $0 \le i < j \le K$, because for each vertex $v \in S_i$, the $f^i_j$-gate decomposition of $D_v S_i$ is a refinement of the $f^i_K$-gate decomposition.

\paragraph{Combing rectangles.} A \emph{combing rectangle}\index{combing rectangle} over $F$ is a commutative diagram of maps over $F$ having the form depicted in Figure~\ref{FigureCombingRectangle}, such that:
\begin{enumerate}
\item The top horizontal row is a foldable sequence.
\item Each vertical arrow $\pi_i \from S_i \to T_i$ is a collapse map with collapsed subgraph $\sigma_i \subset S_i$ indicated in the notation.
\item For all $i=1,\ldots,K$ we have $\sigma_{i-1} = f^\inv_i(\sigma_i)$. Equivalently, for all $0 \le i < j \le K$ we have $\sigma_i = (f^i_j)^\inv(\sigma_j)$.
\end{enumerate}
\begin{figure}[h]
$$\xymatrix{
S_0 \ar[r]^{f_1} \ar[d]_{[\sigma_0]}^{\pi_0} 
 & \cdots \ar[r]^{f_{i-1}} 
 & S_{i-1} \ar[d]_{[\sigma_{i-1}]}^{\pi_{i-1}} \ar[r]^{f_i} 
 & S_i \ar[d]_{[\sigma_i]}^{\pi_i} \ar[r]^{f_{i+1}}
 & \cdots \ar[r]^{f_K}  
 & S_K \ar[d]_{[\sigma_K]}^{\pi_K}  \\
T_0 \ar[r]^{g_1}                                                
 & \cdots \ar[r]^{g_{i-1}} & T_{i-1}                                   \ar[r]^{g_i} 
 & T_i                                 \ar[r]^{g_{i+1}}
 & \cdots \ar[r]^{g_K} 
 & T_K
}
$$
\caption{A combing rectangle. Horizontal sequences are foldable, the top by definition and the bottom by Lemma \ref{LemmaCombingProperties}. Vertical arrows are collapses and $\sigma_{i-1}=f_i^\inv(\sigma_i)$.}
\label{FigureCombingRectangle}
\end{figure}

As mentioned earlier, combing is not closed on the set of fold sequences. We will eventually prove that combing is closed on the set of all foldable sequences; the following proves this in part, by showing closure under ``combing by collapse''.

\begin{lemma} \label{LemmaCombingProperties} For any combing rectangle notated as in Figure~\ref{FigureCombingRectangle}, the bottom row is a foldable sequence.
\end{lemma}

We put off the proof of Lemma~\ref{LemmaCombingProperties} until after the definition of the projection map.

\paragraph{Projecting onto fold paths.} Given a free splitting~$F \act T$, a fold sequence $S_0 \mapsto\cdots\mapsto S_K$, and an integer $k \in [0,K]$, a \emph{projection diagram from $T$ to $S_0 \mapsto\cdots\cdots S_K$ of depth $k$}\index{projection diagram} is a commutative diagram of free splittings and maps over $F$ of the form depicted in Figure~\ref{FigureProjDiagram}, such that each horizontal row is a foldable sequence, and the two rectangles shown are combing rectangles. 
\begin{figure}[h]
$$\xymatrix{
T_0 \ar[r] \ar[d] & \cdots \ar[r] & T_{k} \ar[r] \ar[d] &  T \\
S'_0 \ar[r]          & \cdots \ar[r] & S'_{k} \\
S_0 \ar[r] \ar[u] & \cdots \ar[r] & S_{k} \ar[r] \ar[u]  & \cdots \ar[r] & S_K \\
}$$
\caption{A projection diagram of depth $k$ from $T$ to $S_0\mapsto\cdots\mapsto S_K$.}
\label{FigureProjDiagram}
\end{figure}

The \emph{projection} $\pi(T) \in [0,\ldots,K]$ of $T$ to $S_0\mapsto\cdots\mapsto S_K$ is defined to be the maximum depth of any projection diagram from a free splitting conjugate to $T$ to a fold sequence equivalent to $S_0 \mapsto\cdots\mapsto S_K$, if such a diagram exists, and $\pi(T)=0$ otherwise. It is clear that this gives a well-defined function $\pi \from \FS'(F) \to [0,\ldots,K]$ that depends only on the equivalence class of the fold sequence $S_0 \mapsto\cdots\mapsto S_K$.

One way to understand this definition is to think of $\FS'(F)$ as being Gromov hyperbolic and to think of fold paths as being quasigeodesic, all of which are true a posteriori assuming that Proposition~3.2 is true. That being so, given a fold path $S_0 \mapsto\cdots\mapsto S_K$ and $T$ projecting to $\pi(T) \in [0,\ldots,K]$, by moving to some point $S'_0$ nearby $S_0$ we should obtain a fold path from $S'_0$ to $T$ having an initial segment that fellow travels the given fold path from $S_0$ to $S_{\pi(T)}$ and no farther. In defining the projection map as above, the intuition is that combing rectangless provide an adequate definition of fellow traveling. The technical details of the definition were crafted to what would work in our proofs, but also received some original motivation from results of \cite{MasurMinsky:complex1} which amount to a proof that for any finite type oriented surface $S$, splitting sequences of train tracks on $S$ define quasigeodesics in the curve complex of~$S$. In particular, Lemma~4.4 of that paper --- which can be regarded as a verification of the \emph{Coarse Lipschitz} axiom --- contains the statement ``$\beta \in PE(\sigma)$'', and if one works out the train track diagram for that statement one obtains a rather strong analogue of our projection diagram above.

The rest of the paper is devoted to the proof of the following, which immediately implies Proposition~\ref{PropFoldContractions} and therefore implies the Main Theorem:

\begin{proposition}\label{PropProjToFoldPath}
There exist $a,b,c>0$ depending only on $\rank(F)$ such that for any fold sequence $S_0 \mapsto\cdots\mapsto S_K$ in $\FS'(F)$, the projection map $\pi \from \FS'(F) \to [0,\ldots,K]$ defined above satisfies the \emph{Coarse retraction}, \emph{Coarse Lipschitz}, and \emph{Desymmetrized strong contraction} axioms with constants $a,b,c$.
\end{proposition}

The \emph{Coarse Retraction} axiom is proved in Proposition~\ref{PropCoarseRetract} and the other two axioms are proved in Section~\ref{SectionMainProof}.

\bigskip

We now turn to:

\begin{proof}[Proof of Lemma \ref{LemmaCombingProperties}.] 
Following the notation of Figure~\ref{FigureCombingRectangle}, we must show that each map $g^i_K \from T_i \to T_K$ is foldable. First note that $g^i_K$ is injective on each edgelet $e \subset T_i$, because $e = \pi_i(\ti e)$ for some edgelet $\ti e \subset S_i \setminus \sigma_i$, so $f^i_K(\ti e) \subset S_K \setminus \sigma_K$, so $\pi_K(f^i_K(\ti e)) = g^i_K(\pi_i(\ti e)) = g^i_K(e)$ is an edgelet of $T_K$.

Given a vertex $w \in T_i$, we must show that $g^i_K$ has $\ge 2$ gates at $w$, and that if $w$ is natural then $g^i_K$ has $\ge 3$ gates. Let $w' = g^i_K(w) \in T_K$. We have a subgraph $W' = \pi_K^\inv(w') \subset S_K$, and a subgraph $W = \pi_i^\inv(w) \subset S_i$ such that $f^i_K(W) \subset W'$. Note that each direction in $D_W S_i$ is based at a frontier vertex of $W$ and is represented by an edgelet of $S_i \setminus \sigma_i$, and similarly for $D_{W'} S_K$, and so these direction sets are in the domain of definition of the derivative maps $d \pi_i$, $d \pi_K$, respectively. We have a commutative diagram of derivatives
$$\xymatrix{
D_W S_i \ar[r]^{d f^i_K} \ar[d]_{d\pi_i} & D_{W'} S_K \ar[d]^{d\pi_K} \\
D_w  T_i \ar[r]_{d_w g^i_K}                       & D_{w'} T_K
}$$
in which the vertical maps are bijections and so $d \pi_i$ induces a bijection between gates of $d_w g^i_K$ and point pre-images of the map in the top row. The valence of $w$ therefore equals the cardinality of the set $D_W S_i$, and the number of gates of $g^i_K$ at $w$ equals the cardinality of the image of the map in the top row. If $w$ has valence $\ge 2$ (resp.\ $\ge 3$) then we must prove that the image of the map in the top row has cardinality $\ge 2$ (resp.~$\ge 3$).

Suppose that $w$ is a valence~$2$ vertex contained in the interior of a natural edge $\eta \subset T_i$. The subgraph $W$ is either a point or a segment contained in the interior of a natural edge $\ti\eta \subset S_i$ such that $\pi_i(\ti\eta)=\eta$. Let $e_1,e_2 \subset \eta$ be the two oriented edgelets incident to $w$, representing the two directions of the set $D_w T_i$. Let $\ti e_1, \ti e_2 \subset \ti\eta \setminus W$ be the two oriented edgelets incident to the endpoints of $W$ representing the two elements of the set $D_W S_i$, indexed so that $\pi_i(\ti e_j)=e_j$, $j=1,2$. Since $f^i_K$ is injective on $\ti\eta$ it follows that $f^i_K(\ti e_1)$, $f^i_K(\ti e_2)$ are distinct edgelets of~$S_K$. Noting that $g^i_K(e_j) = g^i_K(\pi_i(\ti e_j)) = \pi_K(f^i_K(\ti e_i))$ for $j=1,2$, it follows that these are two distinct edgelets of $T_K$, and so $g^i_K$ has 2 gates at~$w$.

Suppose now that $w$ is a vertex of valence~$\ge 3$, so the set $D_W S_i$ has cardinality~$\ge 3$. If $W$ is a point then it has valence~$\ge 3$ and, since $f^i_K$ is foldable, there are $\ge 3$ gates of $f^i_K$ in $D_W S_i$; it follows that there are $\ge 3$ gates of $g^i_K$ in $D_w S_i$. If $W$ has infinite diameter then $D_W S_i$ is infinite and so $w$ has infinite valence, implying that $g^i_K$ has infinitely many gates at $w$. If $W$ does not contain a natural vertex of $S_i$ then it is a segment in the interior of a natural edge of $S_i$ implying that $w$ has valence~$2$, a contradiction.

We have reduced to the case that the graph $W$ has finite diameter, is not a point, and contains a natural vertex of $S_i$. The graph $f^i_K(W)$ also has finite diameter and is not a point, and so has $P \ge 2$ vertices of valence~$1$ (the cardinality $P$ may be countably infinite). Let $X \subset W$ be a set consisting of one vertex of $W$ in the preimage of each valence~$1$ vertex of $f^i_K(W)$. By the First Derivative Test, each $x \in X$ is a frontier vertex of $W$. Choosing a direction $\delta_x \in D_W S_i$ based at each $x \in X$, it follows that the directions $d f^i_K(\delta_x)$ are based at $P$ distinct points of $S_K$ and are therefore $P$ distinct directions in the set $D_{W'} S_K$. If $P \ge 3$ then we are done.

We have reduced further to the subcase that $P=2$, and so $f^i_K(W)$ is a segment with endpoints $u_1,u_2$. We have $X = \{x_1,x_2\}$ with $f^i_K(x_j)=u_j$. Consider a natural vertex $v \in S_i$ such that $v \in W$, and its image $v' = f^i_K(v) \in f^i_K(W)$. Since $f^i_K$ is foldable, there are $\ge 3$ gates at $v$ with respect to $f^i_K$. If $v' = u_j$ then at least one of the gates at $v$ maps to a direction at $u_j$ which is distinct from the direction $d f^i_K (\delta_{x_j})$ and from the unique direction of the segment $f^i_K(W)$ at $u_j$, and so we have found a third direction in the set $D_{W'} S_K$. If $v'$ is an interior point of the segment $f^i_K(W)$ then at least one of the gates at $v$ maps to a direction at $v'$ distinct from the two directions of the segment $f^i_K(W)$ at $v'$ and again we have found a third direction in $D_{W'} S_K$.
\end{proof}

\subsection{Combing by collapse and combing by expansion}
\label{SectionCombingConstructions}
In approaching the proof of Proposition~\ref{PropProjToFoldPath}, one immediately confronts the need for constructions of combing rectangles, in order to construct the projection diagrams needed to compute projection maps. This section and the next contain the constructions of combing rectangles that we use for this purpose. 

Our first construction of combing rectangles shows how to comb a foldable sequence followed by a collapse map.

\begin{proposition}[Combing by collapse]\label{PropCBC}
For each foldable sequence $S_0 \xrightarrow{f_1} S_1 \xrightarrow{f_2} \cdots \xrightarrow{f_K} S_K$, and for each collapse $S_K \xrightarrow{[\sigma_K]} T'$ there exists a combing rectangle of the form shown in Figure~\ref{FigureCombingRectangle} such that $T_K=T$.
\end{proposition}

\begin{proof} Define an equivariant subgraph $\sigma_i \subset S_i$ using the definition of a combing rectangle: starting with $\sigma_K \subset S_K$, by induction define $\sigma_i = f^\inv_{i+1}(\sigma_{i+1})$. Since $\sigma_K \subset S_K$ is a proper equivariant subgraph it follows by induction that each $\sigma_i \subset S_i$ is a proper equivariant subgraph, and so free splittings $F \act T_i$ with collapse maps $S_i \xrightarrow{[\sigma_i]} T_i$ and induced maps $g_{i} \from T_{i-1} \to T_i$ are all defined, and the squares are all evidently commutative. 
\end{proof}

We remark that the cheapness of the above proof is slightly offset by the modest expense of proving that the $T_i$ sequence is foldable, which was done back in Lemma~\ref{LemmaCombingProperties}.

\bigskip

Next we explain how to comb a foldable sequence followed by an expansion. In sharp contrast to the case of combing by collapse, both the construction of the combing rectangle and the proof that the resulting map sequence is foldable are very intricate in the case of combing by expansion.

\begin{proposition}[Combing by expansion]\label{PropCBE}
For each foldable sequence $S_0 \xrightarrow{f_1} S_1 \xrightarrow{f_2} \cdots \xrightarrow{f_K} S_K$, each expansion $S_K \expandsto T'$, and each collapse map $\pi' \from T' \to S_K$, there exists a combing rectangle of the form
$$\xymatrix{
 S_0 \ar[r]^{f_1} 
 & \cdots \ar[r]^{f_{i-1}} 
 & S_{i-1} \ar[r]^{f_i} 
 & S_i  \ar[r]^{f_{i+1}}
 & \cdots \ar[r]^{f_K}  
 & S_K  \\
 T_0 \ar[r]^{g_1} \ar[u]^{[\sigma_0]}_{\pi_0}                                          
 & \cdots \ar[r]^{g_{i-1}} 
 & T_{i-1}  \ar[u]^{[\sigma_{i-1}]}_{\pi_{i-1}}  \ar[r]^{g_i} 
 & T_i  \ar[u]^{[\sigma_i]}_{\pi_i}   \ar[r]^{g_{i+1}}
 & \cdots \ar[r]^{g_K} 
 & T_K \ar[u]^{[\sigma_K]}_{\pi_K=\pi'} \ar@{=}[r] 
 & T'
}
$$
\end{proposition}

\textbf{Remark.} Implicit in the conclusion via the definition of combing rectangle is that the sequence $T_0\xrightarrow{g_1}\cdots\xrightarrow{g_K} T_K$ is foldable.

\begin{proof} We will construct this combing rectangle in two steps. In Step~1 we produce a commutative diagram of free splittings and maps of the form  
$$\xymatrix{
 S_0 \ar[r]^{f_1} 
 & \cdots \ar[r]^{f_{i-1}} 
 & S_{i-1} \ar[r]^{f_i} 
 & S_i  \ar[r]^{f_{i+1}}
 & \cdots \ar[r]^{f_K}  
 & S_K  \\
 U_0 \ar[r]^{h_1} \ar[u]_{[\sigma'_0]}^{\pi'_0}                                          
 & \cdots \ar[r]^{h_{i-1}} 
 & U_{i-1}  \ar[u]_{[\sigma'_{i-1}]}^{\pi'_{i-1}}  \ar[r]^{h_i} 
 & U_i  \ar[u]_{[\sigma'_i]}^{\pi'_i}   \ar[r]^{h_{i+1}}
 & \cdots \ar[r]^{h_K} 
 & U_K \ar[u]_{[\sigma'_K=\sigma']}^{\pi'_K} \ar@{=}[r] 
 & T'
}
$$
in which each $\pi'_i$ is a collapse and $h_i^\inv(\sigma'_i)=\sigma'_{i-1}$, but the $U$ row slightly fails to be foldable in that certain explicitly described natural vertices of $U_i$ are ``bad'' by fault of having only $2$~gates with respect to~$h^{i}_K \from U_i \to U_K$. One of these gates will always be a singleton, and so each ``bad natural vertex'' will be incident to a ``bad natural edge''. In Step~2 we will repair this problem by splitting each bad natural edge, to produce a commutative diagram of the form
$$\xymatrix{
 U_0 \ar[r]^{h_1} 
 & \cdots \ar[r]^{h_{i-1}} 
 & U_{i-1} \ar[r]^{h_i} 
 & U_i  \ar[r]^{h_{i+1}}
 & \cdots \ar[r]^{h_K}  
 & U_K  \ar@{=}[r] 
 & T' \\
 T_0 \ar[r]^{g_1} \ar[u]^{\mu_0}                                          
 & \cdots \ar[r]^{g_{i-1}} 
 & T_{i-1}  \ar[u]^{\mu_{i-1}}  \ar[r]^{g_i} 
 & T_i  \ar[u]^{\mu_i}   \ar[r]^{g_{i+1}}
 & \cdots \ar[r]^{g_K} 
 & T_K \ar@{=}[u]^{\mu_K} \ar@{=}[r] 
 & T'
}
$$
The $T$ row will be a foldable sequence. The $\mu_i$ maps are not collapses but instead are ``multifolds'' that invert the splitting process. The desired combing rectangle will be obtained by concatenating these two rectangles: the composition $\pi_i \from T_i \xrightarrow{\mu_i} U_i \xrightarrow{\pi'_i} S_i$ will indeed be a collapse map, which collapses the subgraph $\sigma_i = \mu_i^\inv(\sigma'_i) \subset T_i$.  

\subparagraph{Step 1.} The free splitting $F \act U_i$ is defined to be the minimal subtree of the pushout of $S_i$ and $T'$. Here are more details. As a set, the pushout of $S_i$ and $T'$ is 
$$\wedge(S_i,T') = \{(x,y) \in S_i \cross T' \suchthat f^i_K(x)=\pi'(y) \}
$$
The action $F \act \wedge(S_i,T')$ is obtained by restricting the diagonal action $F \act S_i \cross T'$. The restrictions of the two projection maps are denoted 
$$\pi'_i \from \wedge(S_i,T') \to S_i \quad\text{and}\quad h^i_{T'} \from \wedge(S_i,T') \to T'
$$
Both are clearly $F$-equivariant and we have $f^i_K \composed \pi'_i = \pi' \composed h^i_{T'} \from \wedge(S_i,T') \to S_K$. As~a graph, the vertices and edgelets of the pushout are as follows. A vertex is a pair $(v,w) \in \wedge(S_i,T')$ such that $v$ is a vertex of $S_i$ and $w$ is a vertex of $T'$. Edgelets are of two types. First, a \emph{collapsed edgelet} is one of the form $v \cross e'$ where $v \in S_i$ is a vertex and $e' \subset \sigma' \subset T'$ is an edgelet such that $\pi'(e') = f^i_K(v)$; the barycentric coordinates on $e'$ induce those on $v \cross e'$ via the projection $h^i_{T'}$. Second, to each edgelet $e \subset S_i$ there corresponds a unique edgelet $e' \subset T'$ with the property that $f^i_K(e)=\pi'(e')$ (uniqueness follows since $\pi'$ is a collapse map), and there corresponds in turn an \emph{uncollapsed} edgelet $\ti e = \wedge(e,e') = \{(x,y) \in \wedge(S_i,T') \suchthat x \in e, y \in e'\}$ of $\wedge(S_i,T')$ with barycentric coordinates induced via the map $f^i_k \composed \pi'_i = \pi' \composed h^i_{T'}$ which takes $\ti e$ bijectively to the edgelet $f^i_K(e)=\pi'(e')$ of~$S_K$. The action of $F$ on $\wedge(S_i,T')$ and the projection maps $\pi'_i$, $h^i_{T'}$ are each simplicial. The simplicial complex $\wedge(S_i,T')$ is 1-dimensional by construction. It is furthermore a tree, in that removal of the interior of each edgelet separates, because the simplicial map $\pi'_i \from \wedge(S_i,T') \to S_i$ is injective over the interior of each edgelet of $S_i$, and for each vertex $x \in S_i$ the subcomplex $(\pi'_i)^\inv(x)$ is a tree (mapped by a simplicial isomorphism to the tree $(\pi')^\inv(f^i_K(x)) \subset T'$). The action $F \act S_i$ has no point fixed by each element of $F$, and so neither does the action $F \act \wedge(S_i,T')$; it follows that the $F$-tree $\wedge(S_i,T')$ contains a unique minimal $F$-invariant subtree which, by definition, is $U_i$. For each edgelet $e \subset \wedge(S_i,T')$, its stabilizer is contained in $\Stab_{S_i}(\pi'_i(e))$ if $e$ is uncollapsed and in $\Stab_{T_i}(h^i_{T'}(e))$ if $e$ is collapsed, and in either case is trivial. This proves that $F \act U_i$ is a free splitting. 

Here are some structural facts about the tree $U_i$. For each edgelet $e \subset S_i$, the edgelet $\ti e \subset \wedge(S_i,T')$ is the unique one mapped to $e$ via $\pi'_i$, and since $F \act S_i$ is minimal the map $\pi'_i \from U_i \to S_i$ is surjective which forces $\ti e$ to be contained in $U_i$. This also shows that $\pi'_i$ is a collapse map. The union of the collapsed edgelets of the pushout $\wedge(S_i,T')$ forms a subgraph $\Sigma_i \subset \wedge(S_i,T')$ with one component $\Sigma_{i,v} = (\pi'_i)^\inv(v)$ for each vertex $v \in S_i$ such that $(\pi')^\inv(f^i_K(v))$ is a component of~$\sigma'$; the map $h^i_{T'}$ restricts to a simplicial isomorphism between these components. The subgraph of $\sigma'_i \subset U_i$ that is collapsed by $\pi'_i \from U_i \to S_i$ is the union of those components of $\Sigma_i \intersect U_i$ that contain at least one edge. Each of these components has the form $\sigma'_{i,v} = \Sigma_{i,v} \intersect U_i$ when this intersection contains at least one edge; by convention we set $\sigma'_{i,v} = \emptyset$ otherwise. See below for a more detailed description of various features of $\sigma'_{i,v}$.

There is an induced map $h_i \from \wedge(S_{i-1},T') \to \wedge(S_i,T')$ which is defined by the formula $h_i(x,y)=(f_i(x),y)$, which makes sense because for each $(x,y) \in \wedge(S_{i-1},T')$ we have $f^i_K(f_{i-1}(x)) = f^{i-1}_K(x)=\pi'(y)$. The commutativity equation $\pi'_i \composed h_i = f_i \composed \pi'_{i-1}$ is immediate. Since $U_i$ is the minimal subtree of $\wedge(S_i,T')$ it follows that $h_i(U_{i-1}) \supset U_i$, but we are not yet in a position to prove the opposite inclusion, not until we have established that the map $h^i_{T'} \from U_i \to T'$ has $\ge 2$ gates at each vertex.

\subparagraph{Preparation for Step 2.} Here are some structural facts about the components of $\sigma'_i$. Consider a vertex $v \in S_i$ for which $\sigma'_{i,v} \ne \emptyset$ and so is a component of $\sigma'_i$. Given an oriented edgelet $e \subset S_i$ we abuse notation by writing $e \in D_v S_i$ to mean that $v$ is the initial vertex of~$e$. There is a function $\xi_{i,v} \from D_v S_i \to U_i$ where for each $e \in D_v S_i$ we define $\xi_{i,v}(e) \in \sigma'_{i,v}$ to be the initial vertex of the corresponding oriented edgelet $\ti e \subset U_i$. Note that the set $\image(\xi_{i,v})$ is the topological frontier of the subtree $\sigma'_{i,v}$ in the tree $U_i$. By Lemma~\ref{LemmaCollapseProps}~\pref{ItemCollapseFrontierHull} it follows that $\sigma'_{i,v}$ is the convex hull of the set $\image(\xi_{i,v})$ in $U_i$. Notice also that the function $\xi_{i,v}$ is constant on each gate of $D_v S_i$ with respect to the map $f^i_K$, for if $e_1,e_2 \in D_v S_i$ are in the same gate then $f^i_K(e_1)=f^i_K(e_2)$ is a single edgelet in $S_K$ which lifts to a unique edgelet $e' \subset T'$ and we have 
$$h^i_K(\ti e_1) = \wedge(f^i_K(e_1),e') = \wedge(f^i_K(e_2),e') = h^i_K(\ti e_2)
$$
and so the initial endpoints of $\ti e_1$ and $\ti e_2$ have the same image under $h^i_K$. But these initial endpoints are in the graph $\sigma'_{i,v}$ on which $h^i_K$ is injective, so these initial endpoints are equal. Letting $\Gamma_v S_i$ denote the set of gates of $f^i_K$ in $D_v S_i$, the map $\xi_{i,v}$ induces a map which we also denote $\xi_{i,v} \from \Gamma_v S_i \to \sigma'_{i,v}$ whose image is also the frontier of $\sigma'_{i,v}$.

We now study the extent to which the maps $h^i_K \from U_i \to U_K$ are foldable. Note first that we may identify $T'$ with the pushout $\wedge(S_K,T')$ and so we may identify $U_K = T'$ and $\sigma'_K = \sigma'$ up to simplicial conjugacy and we may identify $h^i_K=h^i_{T'}$, in particular the gates of $h^i_K$ and of $h^i_{T'}$ are therefore identical. We will show that $h^i_{T'}$ has $\ge 2$ gates at each vertex of $U_i$, so a vertex is either \emph{good} meaning it has valence~$\ge 3$ and~$\ge 3$ gates or valence~$2$ and~$2$ gates, or \emph{bad} meaning it has valence~$\ge 3$ but only~$2$ gates. We shall do this through a case analysis, going through various cases of good vertices and narrowing down to one remaining case which is bad. This will yield an explicit description of the bad vertices which will be used in describing the free splitting $F \act T_i$.

Fix a vertex $u = (v,w) \in U_i$, so if $\sigma'_{i,v} \ne \emptyset$ then $u \in \sigma'_{i,v}$. Denote $x = f^i_K(v) = \pi'(w)$.

Consider first the case that $\sigma'_{i,v} = \emptyset$; we shall show that $u$ is good. We have a commutative diagram of derivative maps
$$\xymatrix{
D_v S_i \ar[r]^{d_v f^i_K} & D_x S_K \\
D_u U_i \ar[u]^{d_u \pi'_k} \ar[r]^{d_u h^i_{T'}} & D_w T' \ar[u]_{d_w \pi'}
}$$
where the left arrow is a bijection, i.e. the valences of $u$ and $v$ are equal. Also, the set $\image(d_u h^i_{T'})$ is in the domain of definition of the right arrow and the right arrow is an injection on its domain of definition. The number of gates at $u,v$ are therefore equal. Since $f^i_K$ is foldable it follows that $u$ is good.

Consider now the case that $\sigma'_{i,v} \ne \emptyset$. To simplify notation we denote $W=\sigma'_i$ and $W_v = \sigma'_{i,v}$. Each gate of $h^i_{T'}$ in $DU_i$ is contained either in $DW$ or its complement $D(U_i \setminus W) = D U_i - D W$, because $W=\sigma'_i = (h^i_{T'})^\inv(\sigma')$ implying that $h^i_{T'}$ never maps a direction of $W$ and a direction of $U_i \setminus W$ to the same direction of $T'$. Since $h^i_{T'}$ is injective on $W_v$, each direction in the set $D_u W_v$ constitutes an entire gate of $D_u U_i$. Gates at $u$ in the complement $D_u U_i - D_u W_v$ exist if and only if $u$ is a frontier vertex of $W_v$, if and only if $u$ is in the image of $\xi_{i,v} \from D_v S_i \to W_v$.

Consider the subcase that $v$ has valence~$2$ in $S_i$. The graph $W_v$ is then a segment contained in the interior of a natural edge of $U_i$. The vertex $u$ therefore has valence~$2$ in~$U_i$, with either $2$ directions in $W_v$ or one each in $W_v$ and in $U_i \setminus W_v$, and in either case these 2 directions are mapped by $h^i_{T'}$ to two different directions in $T'$ and so $u$ is good.

Consider the subcase that $v$ has valence~$\ge 3$ in $S_i$. If the valence of $u$ in $W_v$ plus the number of gates at $u$ in the complement of $W_v$ is $\ge 3$ then $h^i_{T'}$ has $\ge 3$ gates at $u$, so $u$ is good. If $u$ is an interior vertex of $W_v$ then $u$ has valence~$\ge 2$ in $W_v$ by minimality of $F \act U_i$; furthermore, the valences of $u$ in $W_v$ and in $U_i$ are equal and the number of gates of $h^i_{T'}$ at $u$ equals the valence, so $u$ is good. If $u$ is a frontier vertex of valence~$\ge 2$ in $W_v$ then $u$ has $\ge 1$ gate in the complement of $W_v$ and we considered this case already and showed that $u$ is good. If $u$ is a frontier vertex of valence~$1$ in $W_v$ and if $u$ has $\ge 2$ gates in the complement of $W_v$ then we have also considered this case already and showed that $u$ is good. If $u$ is a frontier vertex of valence~$1$ in $W_v$ and $u$ has exactly 1 direction in the complement of $W_v$ then $u$ has valence~$2$ in $U_i$ and $2$ gates, so $u$ is good.

The only case that remains, and the case that characterizes when $u$ is bad, is when $v$ has valence~$\ge 3$ in $S_i$, $u$ is a frontier vertex of $W_v$, $u$ has valence~$1$ in $W_v$, $u$ has exactly one gate in the complement of $W_v$, and that gate has cardinality~$\xi_u \ge 2$ called the \emph{external valence} of $u$. When in this case, let $\zeta_u$ be the unique natural edge of $U_i$ with endpoint $u$ and with direction at $u$ equal to the unique direction of $W_v$ at $u$. We call $\zeta_u$ the \emph{bad natural edge} incident to $u$. Let $z_u$ be the natural endpoint of $\zeta_u$ opposite~$u$.

We claim that for each bad natural vertex $u \in U_i$ we have $\zeta_u \subset W_v$; the only way this could fail is if $W_v$ is an edgelet path whose vertices apart from $u$ all have valence~$2$ in $U_i$, implying that $f^i_K$ has 2 gates at the natural vertex $v$, contradicting that $f^i_K$ is foldable. We claim also that $z_u$ is good; otherwise it would follow that $W_v = \zeta_u = \zeta_{z_u}$ which again would imply the contradiction that $f^i_K$ has 2 gates at~$v$.

The union of the bad natural edges of $U_i$ forms an equivariant natural subgraph denoted $Z_i = \union\zeta_u \subset U_i$. The natural edges of its complement $U_i \setminus Z_i$ are the \emph{good natural edges} of $U_i$, some of which may be contained in $W$, some in $U_i \setminus W$, and some in neither. The endpoints of a good natural edge need not be good. From the description of bad natural edges it follows that each component of $Z_i$ contains a unique good vertex $z$ and is the union of some number $m \ge 1$ of bad natural edges with endpoint $z$, forming a star graph with $m$ valence~$1$ vertices apart from $z$.

\subparagraph{Step 2.} Ignoring the simplicial structure for the moment, define the free splitting $F \act T_i$ to be the one obtained from $F \act U_i$ by collapsing the bad subgraph $Z_i \subset U_i$. Let $\rho_i \from U_i \xrightarrow{[Z_i]} T_i$ be the collapse map. Note that $\rho_i$ restricts to an equivariant bijection from the good natural vertices of $U_i$ to the natural vertices of $T_i$, because $Z_i$ is a natural subgraph each of whose components contains exactly one good natural vertex. Also, $\rho_i$ induces a bijection from the good natural edges of $U_i$ --- those in $U_i \setminus Z_i$ --- to the natural edges of~$T_i$: denote this correspondence by $\ti\eta \leftrightarrow \eta$ for each good natural edge $\eta \subset T_i$, and note that $\rho_i$ maps $\ti\eta$ homeomorphically to $\eta$.

Define the map $\mu_i \from T_i \to U_i$ as follows. The restriction of $\mu_i$ to the natural vertices of $T_i$ is the equivariant bijection onto the good natural vertices of $U_i$ that is obtained by inverting~$\rho_i$. The endpoints of each natural edge of $T_i$ map to distinct points of $U_i$, and so $\mu_i$ may be extended equivariantly and continuously to be an injection on each natural edge of $T_i$. 

Define the simplicial structure on $T_i$ to be the unique one with respect to which $\mu_i$ is a simplicial map: its vertices are the inverse image under $\mu_i$ of the vertices of $U_i$; each of its edgelets maps via $\mu_i$ by simplicial isomorphism to an edgelet of~$U_i$.

Define the subgraph $\sigma_i \subset T_i$ to be $\mu_i^\inv(\sigma'_i)$; we will see below that $\pi'_i \composed \mu_i \from T_i \to S_i$ is a collapse map which collapses the subgraph $\sigma_i$. 

Knowing that $\mu_i$ is injective on each natural edge of $T_i$, we describe the image of each natural edge as follows. The notation $u \mapsto z_u$, which so far defines an equivariant function from the bad natural vertices of $U_i$ to the good natural vertices of $U_i$, extends to all natural vertices of $U_i$ by defining $z_u=u$ when $u$ is good. For each natural vertex $u \in U_i$ we have $\mu_i(\rho_i(u)) = z_u$: if $u=z_u$ is good this is because $\mu_i$ and $\rho_i$ are inverse bijections between good natural vertices of~$U_i$ and all natural vertices of $T_i$; if $u$ is bad then $u$ and $z_u$ are contained in the same component of $Z_i$ so $\rho_i(u)=\rho_i(z_u)$ and hence $\mu_i(\rho_i(u))=\mu_i(\rho_i(z_u))=z_u$. Given a natural edge $\eta \subset T_i$ with corresponding good natural edge $\ti\eta \subset U_i$, letting $u_1,u_2 \in U_i$ be the endpoints of $\ti\eta$ and letting $z_i=z_{u_i} \in U_i$, it follows that $\mu_i(\eta) = \mu_i(\rho_i(\ti\eta))$ is the arc in $U_i$ connecting $z_1$ to $z_2$, which is just the union of $\ti\eta$ together with the bad natural edges incident to whichever of $u_1,u_2$ are bad. 

From this description of $\mu_i$ we derive a few more properties of $\mu_i$, giving details about its behavior over good and bad natural edges of $U_i$, and its behavior on natural edges and natural vertices of $T_i$. 
\begin{description}
\item[(a) $\mu_i$ over good natural edges of $U_i$:] the map $\mu_i$ is injective over the interior of each good natural edge $\ti\eta \subset U_i$, the closure of $\mu_i^\inv(\interior(\eta))$ is an edgelet path contained in $\eta$, and the restriction of $\mu_i$ to this edgelet path is a simplicial isomorphism onto $\ti\eta$.
\item[(b) $\mu_i$ over bad natural edges of $U_i$:] for each bad natural edge $\zeta_u \subset U_i$ oriented to have terminal point $u$ and initial point $z_u$, letting $\chi_u$ be the external valence of $u$, letting $\ti\eta_\ell \subset U_i$ ($\ell=1,\ldots,\chi_u$) be the oriented good natural edges with common initial point $u$, and letting $\eta_\ell = \rho_i(\ti\eta_\ell) \subset T_i$ be the corresponding oriented natural edges with common initial point $w=\rho_i(u)$, there exist initial segments $[w,w_\ell] \subset \eta_\ell$, $\ell = 1,\ldots,\chi_u$, such that $\mu_i$ maps each $[w,w_\ell]$ to $\zeta_u$ by a simplicial isomorphism and such that $\mu_i^\inv(\zeta_u) = \union_{\ell=1}^{\chi_u} [w,w_\ell] \subset \sigma_i$. Furthermore each $w_\ell$ is a valence~$1$ vertex of $\sigma_i$.
\end{description}
Intuitively (a) and (b) together say that $\mu_i$ is a ``partial multifold'', which for each of its gates identifies proper initial segments of the oriented natural edges representing that gate. Perhaps the only nonobvious part of (a) and (b) is the last sentence of (b). For each bad natural vertex $u \in U_i$,  from (a) and the previous sentences of (b) it follows that $\mu_i^\inv(u) = \{w_1,\ldots,w_{\chi_u}\}$, and that for each $\ell=1,\ldots,\chi_u$ the vertex $w_\ell$ is contained in the interior of the natural edge $\eta_\ell$, one direction being in the segment $[w_u,w_\ell] \subset \sigma_i$ and the other direction being in the closure of $\mu_i^\inv(\interior(\eta_\ell))$ which is in $T_i \setminus \sigma_i$, and so $w_\ell$ has valence~$1$ in $\sigma_i$.

\begin{description}
\item[(c) $\mu_i$ on natural edges of $T_i$:] The restriction of $\mu_i$ to each good natural edge of $T_i$ is injective. Furthermore, an embedded edgelet path $\alpha \subset U_i$ is the $\mu_i$-image of some good natural edge of $T_i$ if and only if the endpoints of $\alpha$ are good natural vertices of $U_i$, no interior point of $\alpha$ is a good natural vertex, and $h^i_K \restrict \alpha$ is injective. 
\end{description}
Only the ``if'' part of~(c) is not obvious. Let $\alpha\subset U_i$ be an embedded edgelet path whose only good natural vertices are its endpoints, and suppose that $h^i_K \restrict \alpha$ is injective. If $\alpha$ contains no bad natural vertex then $\alpha = \ti\eta$ is a good natural edge with associated natural edge $\eta \subset T_i$ and $\alpha=\mu_i(\eta)$. If $u \in \alpha$ is a bad natural vertex then $u \in \interior(\alpha)$, and since $h^i_K \restrict \alpha$ is injective it follows that one direction of $\alpha$ at $u$ is the direction of the bad natural arc~$\zeta_u$, whose opposite good natural endpoint $z_u$ must be an endpoint of $\alpha$; the edgelet path $\alpha$ is therefore the concatenation of some natural edge $\ti\eta \subset U_i \setminus Z_i$ with any bad natural edges incident to the endpoints of $\ti\eta$, and it follows that $\alpha = \mu_i(\eta)$.

\begin{description}
\item[(d) $d\mu_i$ at natural vertices of $T_i$:] \qquad For each natural vertex $v \in T_i$, the map \break $d_v \mu_i \from D_v T_i \to D_{\mu_i(v)} U_i$ is surjective.
\end{description}
To justify~(d), the vertex $\mu_i(v)$ is a good natural vertex of $U_i$. Consider a direction $d \in D_{\mu_i(v)} U_i$. If $d$ is the initial direction of some oriented good natural edge $\ti\eta \subset U_i$ corresponding to an oriented natural edge $\eta \subset T_i$, it follows that the initial vertex of $\eta$ equals $v$ and the initial direction of $\eta$ maps to $d$. If $d$ is the initial direction of some bad oriented natural edge $\zeta_u \in U_i$ with opposite bad natural vertex $u$, let $\ti\eta$ be any of the good natural edges incident to $u$ oriented with initial vertex $u$, and let $\eta \subset T_i$ be the corresponding oriented natural edge, and it follows that the initial vertex of $\eta$ again equals $v$ and that the initial direction of $\eta$ maps to $d$. 

We now prove that we have a collapse map $\pi_i = \pi'_i \composed \mu_i \from T_i \xrightarrow{\sigma_i = (\mu_i)^\inv(\sigma'_i)} S_i$. Clearly an edgelet of $T_i$ is in $\sigma_i$ if and only its image under $\mu_i$ is in $\sigma'_i$ if and only if its image under $\pi_i=\pi'_i\composed\mu_i$ is a point. Given an edgelet $e \subset S_i$, the collapse map $\pi'_i$ is injective over the interior of $e$, so there is a unique edgelet $e' \subset U_i$ mapped to $e$ by $\pi'_i$, and $e' \not\subset \sigma'_i$; it follows that $e' \not\subset Z_i$ and so by item (a) above the map $\mu_i$ is injective over the interior of~$e'$; therefore $\pi_i$ is injective over the interior of~$e$.

Putting off for the moment the issue of defining the maps $g_i \from T_{i-1} \to T_i$, we \emph{define} the maps $g^i_K \from T_i \to T_K$ as follows. First note that the map $\mu_K \from T_K \to U_K$ is evidently a simplicial isomorphism, and so we may identify $T_K$ with $U_K$ and with $T'$. We now define $g^i_K$ to be the composition $T_i \xrightarrow{\mu_i} U_i \xrightarrow{h^i_K} U_K \xrightarrow{(\mu_K)^\inv} T_K$. The map $g^i_K$ is foldable, equivalently $h^i_K \composed \mu_i \from T_i \to U_K$ is foldable, for the following reasons: by (c) the map $h^i_K$ is injective on natural edges of $T_i$; for each natural vertex $v \in T_i$, its image $\mu_i(v) \in U_i$ is a good natural vertex and so has $\ge 3$ gates with respect to $h^i_K$, and by (d) the derivative map $d_v \mu_i \from D_v T_i \to D_{\mu_i(v)} U_i$ is surjective, which implies that $h^i_K \composed \mu_i$ has $\ge 3$ gates at $v$.

All that remains is to define a map $g_i \from T_{i-1} \to T_i$ so that the commutativity equation $h_i \composed \mu_{i-1} = \mu_i \composed g_i$ holds, for by combining this with the equation $h^{i-1}_K = h_K \composed\cdots\composed h_i$ it immediately follows that $g^{i-1}_K = g_K \composed\cdots\composed g_{i}$ and so the map sequence $T_0 \xrightarrow{g_1}\cdots\xrightarrow{g_K} T_K$ is defined and is foldable.

Consider a natural vertex $v \in T_{i-1}$. Its image $\mu_{i-1}(v) \in U_{i-1}$ is a good natural vertex and so has $\ge 3$ gates with respect to $h^{i-1}_K$, implying that $h_i(\mu_{i-1}(v)) \in U_i$ has $\ge 3$ gates with respect to $h^i_K$ and so is a good natural vertex, and hence there is a unique natural vertex in $T_i$ that maps to $h_i(\mu_{i-1}(v))$ which we take to be $g_i(v)$. We have thus defined $g_i$ so as to satisfy the commutativity equation on each natural vertex $v \in T_{i-1}$. 

Consider a natural edge $\eta \subset T_{i-1}$ with natural endpoints $v_0 \ne v_1$. Its image $\mu_{i-1}(\eta) \subset U_{i-1}$ is the arc with good natural endpoints $\mu_{i-1}(v_0) \ne \mu_{i-1}(v_1)$. By (c) above the map $h^{i-1}_K = h^i_K \composed h_i$ is injective on the arc $\mu_{i-1}(\eta)$, implying that $h_i$ is injective on $\mu_{i-1}(\eta)$ and that $h^i_K$ is injective on the arc $h_i(\mu_{i-1}(\eta)) \subset U_i$, the latter of which has good natural endpoints $h_i(\mu_{i-1}(v_0)) \ne h_i(\mu_{i-1}(v_1))$. Subdividing the arc $h_i(\mu_{i-1}(\eta))$ at all interior good natural vertices of $U_i$ we write it as a concatenation:
$$h_i(\mu_{i-1}(\eta)) = \alpha_1 * \cdots * \alpha_M
$$
Each of the arcs $\alpha_m$, $m=1,\ldots,M$ has good natural endpoints, no good natural interior points, and the map $h^i_K$ is injective on $\alpha_m$, and so by~(c) there is a unique natural edge $\hat\alpha_m \subset T_i$ mapped by $\mu_i$ to $\alpha_m$ by a simplicial isomorphism. Since every good natural vertex in $U_i$ has a unique natural pre-image in $T_i$, it follows that we may concatenate to obtain an arc $\hat\alpha_1 * \cdots * \hat\alpha_m \subset T_i$, and furthermore the restriction $\mu_i \restrict \hat\alpha_1 * \cdots * \hat\alpha_m$ is a simplicial isomorphism onto $h_i(\mu_{i-1}(\eta))$. Inverting this restriction we may then define 
$$g_i \restrict \eta = (\mu_i \restrict \hat\alpha_1 * \cdots * \hat\alpha_m)^\inv \composed (h_i \composed \mu_{i-1}) \restrict \eta
$$
which is a simplicial isomorphism with image $\hat\alpha_1 * \cdots * \hat\alpha_m$. We have thus defined $g_i$ so as to satisfy the commutativity equation on each natural edge $\eta \subset T_{i-1}$.

This completes the proof of Proposition~\ref{PropCBE}.
\end{proof}

\subsection{Composition and decomposition of combing rectangles.}
\label{SectionCombRectOps}

\begin{lemma}[Composition of combing rectangles]
\label{LemmaCombingComp}
 \qquad \\
Given two combing rectangles of the form 
$$\xymatrix{
    S_0 \ar[r]^{f_1}                \ar[d]^{\pi_0} 
 & \cdots \ar[r]^{f_i}           
 & S_i \ar[r]^{f_{i+1}}           \ar[d]^{\pi_i}
 & \cdots \ar[r]^{f_K}  
 & S_K                                \ar[d]^{\pi_K} 
\\
    T_0 \ar[r]^{g_1}               \ar[d]^{\rho_0}                                   
 & \cdots \ar[r]^{g_i}          
 & T_i      \ar[r]^{g_{i+1}}     \ar[d]^{\rho_i}
 & \cdots \ar[r]^{g_K} 
 & T_K                                 \ar[d]^{\rho_K}
\\ 
    U_0 \ar[r]^{h_1}                                                  
 & \cdots \ar[r]^{h_i} 
 & U_i      \ar[r]^{h_{i+1}}
 & \cdots \ar[r]^{h_K} 
 & U_K
}$$
their composition, which is the commutative diagram
$$\xymatrix{
    S_0 \ar[r]^{f_1}                \ar[d]^{\rho_0 \composed \pi_0} 
 & \cdots \ar[r]^{f_i}             
 & S_i \ar[r]^{f_{i+1}}           \ar[d]^{\rho_i \composed \pi_i}
 & \cdots \ar[r]^{f_K}  
 & S_K                                \ar[d]^{\rho_K \composed \pi_K} 
\\
    U_0 \ar[r]^{h_1}                                                  
 & \cdots \ar[r]^{h_i} 
 & U_i      \ar[r]^{h_{i+1}}
 & \cdots \ar[r]^{h_K} 
 & U_K
}$$
is a combing rectangle. The collapsed subgraph of $\rho_i \composed \pi_i$ is the union of the collapsed subgraph of $\pi_i$ with the inverse image under $\pi_i$ of the collapsed subgraph of $\rho_i$.
\end{lemma}

\begin{proof}
For each edgelet $e \subset U_i$, the map $\rho_{i}$ is injective over the interior of $e$, and so there is a unique edgelet $e' \subset T_i$ such that $\rho_{i}(e')=e$. The map $\pi_i$ is injective over the interior of $e'$, and it follows that $\rho_i \composed \pi_i$ is injective over the interior of $e$. This proves that $\rho_i \composed \pi_i$ is a collapse map and that the second diagram in the statement above is a combing rectangle.

Given an edgelet of $S_i$, clearly its image under $\rho_i \composed \pi_i$ is a vertex of $U_i$ if and only if its image under $\pi_i$ is a vertex of $T_i$ or an edgelet of $T_i$ whose image under $\rho_i$ is a vertex of $U_i$.
\end{proof}

\begin{lemma}[Decomposition of combing rectangles]
\label{LemmaCombingDecomp}
Given a combing rectangle of the form 
$$\xymatrix{
    S_0 \ar[r]^{f_1}                \ar[d]^{\upsilon_0}_{[\sigma_0]}
 & \cdots \ar[r]^{f_i}            
 & S_i \ar[r]^{f_{i+1}}           \ar[d]^{\upsilon_i}_{[\sigma_i]}
 & \cdots \ar[r]^{f_K}  
 & S_K                                \ar[d]^{\upsilon_K}_{[\sigma_K]}
\\
    U_0 \ar[r]^{h_1}                                                  
 & \cdots \ar[r]^{h_i} 
 & U_i      \ar[r]^{h_{i+1}}
 & \cdots \ar[r]^{h_K} 
 & U_K
}$$
and given equivariant subgraphs $\sigma'_i \subset \sigma_i$ ($i=0,\ldots,K$) having the property that $f_{i}^\inv(\sigma'_{i}) = \sigma'_{i-1}$ for each $i=1,\ldots,K$, there exist two combing rectangles of the form 
$$\xymatrix{
    S_0 \ar[r]^{f_1}                \ar[d]^{\pi_0}_{[\sigma'_0]} 
 & \cdots \ar[r]^{f_i}           
 & S_i \ar[r]^{f_{i+1}}           \ar[d]^{\pi_i}_{[\sigma'_i]}
 & \cdots \ar[r]^{f_K}  
 & S_K                                \ar[d]^{\pi_K}_{[\sigma'_K]}
\\
    T_0 \ar[r]^{g_1}               \ar[d]^{\rho_0}                                   
 & \cdots \ar[r]^{g_i}          
 & T_i      \ar[r]^{g_{i+1}}     \ar[d]^{\rho_i}
 & \cdots \ar[r]^{g_K} 
 & T_K                                 \ar[d]^{\rho_K}
\\ 
    U_0 \ar[r]^{h_1}                                                  
 & \cdots \ar[r]^{h_i} 
 & U_i      \ar[r]^{h_{i+1}}
 & \cdots \ar[r]^{h_K} 
 & U_K
}$$
whose composition (as in Lemma~\ref{LemmaCombingComp}) is the given combing rectangle.
\end{lemma}

\begin{proof} Define the collapse map $\pi_i \from S_i \xrightarrow{[\sigma'_i]} T_i$ to be the quotient map obtained by collapsing each component of $\sigma'_i$ to a point. Since $f_i^\inv(\sigma'_i)=\sigma'_{i-1}$, there exists a map $g_i \from T_{i-1} \to T_i$ induced from $f_i \from S_{i-1} \to S_i$ under the quotient, which makes the top rectangle with the $S$ row and the $T$ row commutative, and this rectangle is therefore a combing rectangle. By Lemma~\ref{LemmaCombingProperties}, the $T$ sequence is foldable. Define a subgraph $\tau_i = \pi_i(\sigma_i) \subset T_i$. We have $g_i^\inv(\tau_i) = g_i^\inv(\pi_i(\sigma_i)) = \pi_{i-1}(f_i^\inv(\sigma_i)) = \pi_{i-1}(\sigma_{i-1}) = \tau_{i-1}$, where the second equation is verified by a diagram chase using that the map $\pi_{i-1}$ is surjective, and that $\pi_i$ is injective over the interior of each edgelet of $T_i$. Clearly the collapse map $\upsilon_i \from S_i \xrightarrow{[\sigma_i]} U_i$ factors as the composition of $\pi_i \from S_i \xrightarrow{[\sigma'_i]} T_i$ and a collapse map $\rho_i \from T_i \xrightarrow{[\tau_i]} U_i$, making the bottom diagram with the $T$ row and the $U$ row commutative, and this row is therefore a combing rectangle. 
\end{proof}

\section{Free splitting units}
\label{SectionFSU}

In this section we study how to break a fold sequence into natural units called \emph{free splitting units}. Our story of free splitting units grew in the telling. The original concept was motivated by units along train track splitting paths that are implicit in the ``nested train track argument'' of \cite{MasurMinsky:complex1} and refinements of that argument in \cite{MMS:Splitting}. The details of the definition were tailored to fit the proofs of our two major results: our Main Theorem on hyperbolicity of the free splitting complex, via the arguments of Sections~\ref{SectionPushingDownPeaks}; and Proposition~\ref{PropFoldPathQuasis} which says that free splitting units give a uniformly quasigeodesic parameterization of fold paths in $\FS'(F)$. 

The main results of this section are Proposition~\ref{PropCoarseRetract} which verifies the \emph{Coarse Retraction} axiom of Masur and Minsky, and Lemma~\ref{LemmaUnitsLipschitz} which says that free splitting units give a uniformly coarse Lipschitz parameterization of fold paths in $\FS'(F)$. Underlying Lemma~\ref{LemmaUnitsLipschitz} are Lemmas~\ref{LemmaBRNatural} and~\ref{LemmaBROneB} which give two methods of finding diameter bounds along foldable foldable sequences. 

The diameter bounds, which are stated and proved in Section~\ref{SectionDiamBounds}, arise from finding ``invariant natural structures'' along the foldable sequence. The first diameter bound, Lemma~\ref{LemmaBRNatural} occurs when each free splitting along the fold path decomposes equivariantly into a pair of natural subgraphs in a manner which is ``invariant'' with respect to the foldable maps (see Definition~\ref{DefBRDecompos}). The second diameter bound, Lemma~\ref{LemmaBROneB}, occurs when each free splitting has a particular orbit of natural edges which is ``almost invariant'' with respect to the foldable maps (see Definition~\ref{DefAIE}).

The combinatorial structures underlying the two diameter bounds are used to formulate the definition of free splitting units along a fold sequence (see Definitions~\ref{DefLessThanOneFSU} and~\ref{DefGeneralFSU}). The diameter bounds are not applied directly to the fold sequence itself, but instead to foldable sequences obtained by transforming the given fold sequence via an application of ``combing by collapse'' followed by an application of ``combing by expansion''. One can already see this kind of transformation in the ``nested train track argument'' of \cite{MasurMinsky:complex1}.

\subsection{Diameter bounds along foldable sequences} 
\label{SectionDiamBounds}

In this section we describe a pair of techniques for finding upper bounds on the diameter of foldable sequences. 

\paragraph{Diameter bounds from natural red-blue decompositions.} Consider a free splitting $F \act T$ and a nonempty, proper, $F$-invariant subgraph $\beta \subset T$ having no degenerate components. The conjugacy classes of nontrivial stabilizers of connected components of $\beta$ form a free factor system $\F(\beta)$, as one can see by forming the collapse map $T \xrightarrow{[\beta]} U$ and noting that $\F(\beta)$ is a subset of $\F(U)$. Passing further to the quotient graph of groups $X = U / F_n$, the image of $\beta$ under the composition $T \mapsto U \mapsto X$ is a subset $V_\beta$ of the vertex set of $X$. Let $C_1(\beta)$ be the number of $F$-orbits of components of $\beta$, equivalently the cardinality of $V_\beta$. Let $C_2(\beta)$ be the sum of the ranks of the components of $\F(\beta)$, equivalently the sum of the ranks of the subgroups labelling the vertices $V_\beta$ in the graph of groups $X$, and so we have $0 \le C_2(\beta) \le \rank(F)$. Defining the \emph{complexity} of~$\beta$ to be $C(\beta) \equiv C_1(\beta) + (\rank(F) - C_2(\beta))$, we have $C(\beta) \ge 1$. If furthermore $\beta$ is a natural subgraph of $S$ then $C_1(\beta) \le 3 \rank(F) - 3$, because the number of component orbits of $\beta$ is at most the number of natural edge orbits in $\beta$, and $3 \rank(F)-3$ is an upper bound for the number of natural edge orbits of any free splitting of $F$. Altogether this shows that the complexity of any nonempty, proper, natural, $F$-invariant subgraph $\beta \subset T$ satisfies
$$1 \le C(\beta) \le 4 \rank(F)-3
$$

\begin{definition}[Invariant blue-red decompositions]\label{DefBRDecompos}
An \emph{invariant blue--red decomposition} for a foldable sequence $T_0 \xrightarrow{g_1} T_1 \xrightarrow{g_2}\cdots\xrightarrow{g_k} T_K$, also called an \emph{invariant decomposition} for short, is a decomposition $\beta_k \union \rho_k = T_k$ for each $k=0,\ldots,K$ such that for $0 \le i \le j \le K$ we have $(g^i_j)^\inv(\beta_j) = \beta_i$ and $(g^i_j)^\inv(\rho_j) = \rho_i$ (where in expressions like $(g^i_j)^\inv(\beta_j)$ we abuse notation by deleting degenerate components). Notice that any choice of final decomposition $\beta_K \union \rho_K = T_K$ determines a unique invariant decomposition by the equations $\beta_i = (g^i_K)^\inv(\beta_K)$ and $\rho_i = (g^i_K)^\inv(\rho_K)$.  An invariant decomposition is \emph{natural} if either of the following two equivalent properties holds: $\beta_0,\rho_0$ are natural subgraphs of~$T_0$ if and only if $\beta_k,\rho_k$ are natural subgraphs of $T_k$ for all $k=0,\ldots,K$. The ``only if'' direction follows by observing that the image of each natural vertex under a foldable map is a natural vertex, and so the image of a natural subgraph is a natural subgraph. 
\end{definition}

Because an invariant decomposition is determined by the final decomposition, a general invariant decomposition carries little information about the foldable sequence. The typical behavior is that the edgelets within a natural edge $e \subset T_i$ will alternate many times between red and blue, that is, the number of components of $e \intersect \beta_i$ and $e \intersect \rho_i$ will be very large. Exploiting the difference between this typical behavior and the contrasting special behavior of a natural invariant decomposition is at the heart of the proof of the Main Theorem, specifically in the proof of Proposition~\ref{PropPushdownInToto} Step~2.

Here is our first diameter bound:

\begin{lemma}\label{LemmaBRNatural}
Given a foldable sequence $T_0 \xrightarrow{g_1} T_1 \xrightarrow{g_2}\cdots\xrightarrow{g_k} T_K$ with an invariant natural decomposition $\beta_k \union \rho_k = T_k$, the following hold:
\begin{enumerate}
\item \label{ItemBRDecrease}
The complexity $C(\beta_k)$ is nonincreasing as a function of $k=0,\ldots,K$.
\item \label{ItemBRSubintervals}
The interval $0 \le k \le K$ subdivides into $\le 4 \rank(F)-3$ subintervals on each of which $C(\beta_k)$ is constant.
\item \label{ItemBRBound}
If $C(\beta_k)$ is constant on the subinterval $a \le k \le b$, where $0 \le a \le b \le K$, then 
$$\diam\{T_a,\ldots,T_b\} \le 4
$$
\end{enumerate}
\end{lemma}

\textbf{Remark.} When $T_0 \xrightarrow{g_1} T_1 \xrightarrow{g_2}\cdots\xrightarrow{g_k} T_K$ is a \emph{fold} sequence, one obtains a diameter bound for the entire sequence as follows. Subdivide the interval $0,\ldots,K$ into $\le 4 \rank(F)-3$ subintervals on which $C(\beta_k)$ is constant. On each subinterval one has a diameter bound of $4$. At each of the $\le 4 \rank(F)-4$ fold maps where one subinterval transitions to another, one has an additional distance bound of $2$ coming from Lemma~\ref{LemmaFoldDistance}. Putting these together, 
$$\diam\{T_0,\ldots,T_K\} \le 4 (4 \rank(F)-3) + 2 (4 \rank(F)-4) = 24 \rank(F) - 20
$$
However, the manner in which we actually apply Lemma~\ref{LemmaBRNatural} to fold sequences is via concepts of free splitting units in the next section; see Lemma~\ref{LemmaUnitsLipschitz}.

Before turning to the proof proper of Lemma~\ref{LemmaBRNatural}, we first state a sublemma about the behavior of complexity of invariant subforests under fold maps.

\begin{sublemma}\label{SublemmaBlueFoldBounds}
If $f \from S \to T$ is a fold map of free splittings, if $\beta_T \subset T$ is a nonempty, proper, $F$-invariant subgraph, and if $\beta_S=f^\inv(\beta_T)$ (as usual ignoring degenerate components), then $C_1(\beta_S) \ge C_1(\beta_T)$, and $C_2(\beta_S) \le C_2(\beta_T)$, and so $C(\beta_S) \ge C(\beta_T)$. Furthermore, equality holds if and only if $f$ restricts to a bijection of component sets of $\beta_S$ and $\beta_T$.
\end{sublemma}

We delay the proof of this sublemma and meanwhile turn to:

\begin{proof}[Proof of Lemma \ref{LemmaBRNatural}] Item~\pref{ItemBRDecrease} follows from Sublemma~\ref{SublemmaBlueFoldBounds} by factoring each foldable map $g_k \from T_{k-1} \to T_k$ into folds. Item~\pref{ItemBRSubintervals} follows from \pref{ItemBRDecrease} and the fact that $1 \le C(\beta_K) \le C(\beta_0) \le 4 \rank(F)-3$.

To prove \pref{ItemBRBound}, fixing $i,j$ with $a \le i < j \le b$, it suffices to prove that $d(T_i,T_j) \le 4$. By assumption of \pref{ItemBRBound}, $C(\beta_k)$ is constant for $i \le k \le j$. For each $i < k \le j$, factoring $g_k \from T_{k-1} \to T_k$ into folds, applying \pref{ItemBRDecrease} to get constant complexity on the fold factorization, and applying Sublemma~\ref{SublemmaBlueFoldBounds} to each of those folds, it follows that $g_k$ induces a bijection between the component sets of $\beta_{k-1}$ and $\beta_k$. By composing, it follows that $g^i_j = g_j \composed \cdots \composed g_{i+1}$ induces a bijection between the component sets of $\beta_i$ and~$\beta_j$. 

Now we may factor $g^i_j$ into a fold sequence of the form 
$$T_i = U_0 \xrightarrow{h_1} \cdots \xrightarrow{h_P} U_P \xrightarrow{h_{P+1}}\cdots \xrightarrow{h_Q} U_Q = T_j
$$
by prioritizing folds of blue edge pairs over folds of red edge pairs up until $U_P$ when there are no more blue edge pairs to fold, with the result that if $0 < p \le P$ then an edge pair of $U_{p-1}$ folded by $f_p$ is blue, whereas if $P < q \le Q$ then an edge pair of $S_{q-1}$ folded by $h_q$ is red. To see that prioritizing in this manner is possible, if $g^i_j$ does not already restrict to a simplicial isomorphism from $\beta_i$ to $\beta_j$ then, using that $g^i_j$ induces a bijection of components of $\beta_i$ and $\beta_j$, together with the \emph{Local to global principle} (see the proof of Lemma~\ref{LemmaFoldSequenceConstruction} and the following \emph{Remark}), it follows that some pair of oriented natural edges $\eta_1,\eta_2 \subset \beta_i$ with the same initial vertex have images in $\beta_j$ with the same initial direction. We may therefore define the first fold $h_1$ to be a maximal fold factor of $g^i_j$ obtained by folding $\eta_1,\eta_2$, producing a factorization $T_i = U_0 \xrightarrow{h_1} U_1 \mapsto T_j$. Pushing the natural blue-red decomposition of $U_0$ forward (or equivalently pulling that of $T_j$ back), we obtain a natural blue-red decomposition of $U_1$, and the map $U_1 \mapsto T_j$ still induces a bijection of component sets of blue graphs. We may then continue by induction, stopping when the map $U_P \mapsto T_j$ restricts to a simplicial isomorphism of blue graphs. If the map $U_P \mapsto T_j$ is not already a simplicial isomorphism then one continues the fold factorization arbitrarily, with the result that all folds from $U_P$ to $T_j$ are red.

For $0 \le p \le P$, by collapsing all blue edges of $U_p$, we obtain a free splitting $X_p$ with a collapse map $U_p \mapsto X_p$. Also, for $P \le q \le Q$, by collapsing red edges of $U_q$ we obtain a free splitting $Y_q$ with a collapse map $U_q \to Y_q$. 

We claim that up to equivalence $X_p$ is independent of $p=0,\ldots,P$ and $Y_q$ is independent of $q=P,\ldots,Q$. From this claim it follows that $T_i,T_j$ are connected in $\FS'(F_n)$ by a path of length~$\le 4$ as follows: 
$$[T_i] = [U_0] \collapse [X_0] = [X_P] \expand [U_P] \collapse [Y_P] = [Y_Q] \expand [U_Q] = [T_j]
$$
which completes the proof.

We prove for each $p=1,\ldots,P$ that $X_{p-1},X_p$ are equivalent, and for $q=P+1,\ldots,Q$ that $Y_{q-1},Y_q$ are equivalent; the two cases are similar and we do just the first. Let $e_1,e_2$ be maximal oriented segments with the same initial vertex that are identified by the fold $U_{p-1} \mapsto U_p$. Recall that the fold map $U_{p-1} \mapsto U_p$ can be factored as $U_{p-1} \xrightarrow{q'} U' \xrightarrow{q''} U_p$ where $q'$ identifies proper initial segments of $e_1,e_2$ and $q''$ folds the remaining unidentified segments. Since $e_1,e_2$ are blue, by pushing forward the blue-red decomposition of $U_{p-1}$, or pulling back that of $U_p$, we obtain a blue-red decomposition of $U'$. Furthermore, there is a collapse map $U' \mapsto U_{p-1}$ which collapses the blue segment resulting from partially identifying $e_1,e_2$, and a collapse map $U' \mapsto U_p$ which collapses the remaining unidentified segments, also blue. By composition we obtain collapse maps $U' \mapsto U_{p-1} \to X_{p-1}$ and $U' \mapsto U_p \mapsto X_p$ each of which collapses the entire blue subgraph of $U'$. It follows that $X_{p-1}$ and $X_p$ are equivalent.
\end{proof}

\begin{proof}[Proof of Sublemma \ref{SublemmaBlueFoldBounds}] Let $e_1,e_2 \subset S$ be oriented natural edges with the same initial vertex that are folded by the map $f$. Let $\eta_1 \subset e_1$, $\eta_2 \subset e_2$ be maximal initial subsegments that are identified by $f$. Let $v_1 \in \eta_1$, $v_2 \in \eta_2$ be the terminal endpoints. Note that either $\eta_1 \union \eta_2 \subset \beta_S$ or $\eta_1 \union \eta_2 \subset S \setminus \beta_S$. If $\eta_1,\eta_2 \subset \beta_S$, or if $\eta_1,\eta_2 \subset S \setminus \beta_S$ and either $v_1 \not\in \beta_S$ or $v_2 \not\in \beta_S$, then all inequalities are equations and $f$ is a bijection of component sets. 

We are reduced to the case that $\eta_1 \union \eta_2 \subset S \setminus \beta_S$ and $v_1,v_2 \in \beta_S$, and so $f$ is not a bijection of component sets because the two components $\beta_{S,1}, \beta_{S,2}$ of $\beta_S$ containing $v_1,v_2$ are mapped to the one component of $\beta_{T,0}$ of $\beta_T$ that contains $f(v_1)=f(v_2)$. We must prove that the inequalities $C_1(\beta_S) \ge C_1(\beta_T)$ and $C_2(\beta_S) \le C_2(\beta_T)$ both hold and that at least one of them is strict.  

Let the fold map $f \from S \to T$ be factored as $S \mapsto U \mapsto T$ where $S \mapsto U$ folds short initial segments of $\eta_1,\eta_2$, and $U \mapsto T$ folds the remaining segments, as in the proof of Lemma~\ref{LemmaFoldDistance}. Let $u_1,u_2 \in U$ be the images of $v_1,v_2$. In order to compare the complexities of $\beta_S \subset S$ and $\beta_T \subset T$ we shall move them both into $U$ where we can make the comparison directly.

Letting $\beta_{U} \subset U$ be the image of $\beta_S$, equivalently the preimage of $\beta_T$, the fold map $S \mapsto U$ clearly induces an equivariant bijection from the component set of $\beta_S$ to that of $\beta_{U}$, and so the values of $C_1$, $C_2$, and $C$ on $\beta_S$, $\beta_{U}$ are all equal. Letting $\beta^+_{U} = \beta_{U} \union F \cdot [u_1,u_2]$, the fold map $U \mapsto T$ induces an equivariant bijection from the component set of $\beta^+_{U}$ to that of $\beta_T$, and so the values of $C_1$, $C_2$, and $C$ on $\beta^+_{U}$, $\beta_T$ are equal. So now we must prove the inequalities $C_1(\beta_{U}) \ge C_1(\beta^+_{U})$ and $C_2(\beta_{U}) \le C_2(\beta^+_{U})$ and that at least one of them is strict. 

Let $\beta_{U,1}$, $\beta_{U,2}$ be the images of $\beta_{S,1}$, $\beta_{S,2}$, respectively, under the fold map $S \mapsto U$. In the quotient graph of groups $U/F$, notice that $\beta^+_U / F$ is the union of $\beta_U / F$ with the segment obtained by projecting $[u_1,u_2]$, that segment is disjoint from $\beta_U / F$ except at its endpoints, it has one endpoint on $\beta_{U,1}/F$, and the other endpoint at $\beta_{U,2} / F$, and the stabilizer of the interior vertex of that segment is trivial. It follows that if $C_1(\beta_U) > C_1(\beta^+(U))$, that is, if $\beta_{U,1}$, $\beta_{U,2}$ are in different component orbits, then $C_1(\beta_U) = C_1(\beta^+_U)+1$ and $C_2(\beta_U) = C_2(\beta^+_U)$. On the other hand if $C_1(\beta_U) = C_1(\beta^+_U)$, that is if $\beta_{U,1}$ and $\beta_{U,2}$ are in the same component orbit, then $C_1(\beta_U) = C_1(\beta^+_U)$ and $C_2(\beta_U)+1 = C_2(\beta^+_U)$.
\end{proof}

\paragraph{Diameter bounds from almost invariant edges.} Consider a foldable map $f \from S \to T$ and a natural edge $e_T \subset T$. By ignoring unnatural vertices in $e_T$ and their pre-images in $S$ we may speak about $e_T$-edgelets in $S$; these are the closures of the components of $f^\inv(\interior(e_T))$, each of which is a subsegment of a natural edge of $S$. If $S$ contains a unique $e_T$-edgelet and if $e_S \subset S$ is the natural edge containing that edgelet then we say that the pair $e_S,e_T$ is an \emph{almost invariant edge} of the foldable map $f$. 

\begin{definition}[Almost invariant edge] \label{DefAIE}
An \emph{almost invariant edge} for a foldable sequence $T_0 \xrightarrow{f_1} T_1 \xrightarrow{f_2}\cdots\xrightarrow{f_k} T_K$ is a sequence of natural edges $e_k \subset T_k$, $k=0,\ldots,K$, such that for $0 \le i < j \le K$ the edges $e_i \subset T_i$ and $e_j \subset T_j$ are an almost invariant edge for the foldable map $f^i_j \from T_i \to T_j$. Note that an almost invariant edge exists for the whole foldable sequence if and only if one exists for the map $f^0_K \from T_0 \to T_K$. To see why, observe that for any natural edge $e_K \subset T_K$, letting $m_k$ be the number of $e_K$ edgelets in $T_k$, the sequence $m_k$ is nonincreasing as a function of $k \in 0,\ldots,K$. If there is a natural edge $e_0 \subset T_0$ so that $e_0,e_K$ is an almost invariant edge for the map $f^0_K$ then $m_0=1$, and so $m_k$ has constant value equal to~$1$. Letting $e_k \subset T_k$ be the unique natural edge containing an $e_K$ edgelet in $T_k$, it follows that $(e_k)_{0\le k \le K}$ is an almost invariant edge for the whole foldable sequence. This argument also shows that each almost invariant edge for a foldable sequence $T_0 \mapsto\cdots\mapsto T_K$ is determined by its last term $e_K \subset T_K$.
\end{definition}

Here is our second diameter bound:

\begin{lemma}\label{LemmaBROneB}
Given a foldable sequence $T_0 \mapsto \cdots \mapsto T_K$, the following are equivalent:
\begin{enumerate}
\item \label{ItemAIEexistsSome}
The foldable map $T_0 \mapsto T_K$ has an almost invariant edge.
\item \label{ItemAIEexistsAll}
The foldable sequence $T_0 \mapsto \cdots \mapsto T_K$ has an almost invariant edge.
\item \label{ItemOneEdgeExistsAll}
There exists a one-edge free splitting $R$ such that $d(T_k,R) \le 1$ for all $k=0,\ldots,K$.
\item \label{ItemOneEdgeExistsSome}
There exists a one-edge free splitting $R$ such that $d(T_0,R) \le 1$ and $d(T_K,R) \le 1$.
\end{enumerate}
Furthermore if these hold then $\diam\{T_0,\ldots,T_K\} \le 2$.
\end{lemma}

\begin{proof} The bound in the last sentence clearly follows from~\pref{ItemOneEdgeExistsAll}. We have seen that \pref{ItemAIEexistsSome}$\implies$\pref{ItemAIEexistsAll}, and clearly~\pref{ItemOneEdgeExistsAll}$\implies$~\pref{ItemOneEdgeExistsSome}.

\medskip

We next prove \pref{ItemAIEexistsAll}$\implies$\pref{ItemOneEdgeExistsAll}. Let $(e_k)_{k=0,\ldots,K}$ be an almost invariant edge. Let $\sigma_k \subset T_k$ be the complement of the orbit of the natural edge $e_k$. Define a collapse map $T_k \xrightarrow{[\sigma_k]} R_k$, so $R_k$ is a one-edge free splitting. It suffices to prove for each $k=1,\ldots,K$ that $[R_{k-1}]=[R_k]$. Letting $e'_{k-1} \subset e_{k-1}$ be the unique $e_k$-edgelet in $T_{k-1}$, letting $\sigma'_{k-1} \subset T_{k-1}$ be the complement of the orbit of $e'_{k-1}$, and defining a collapse map $T_{k-1} \xrightarrow{[\sigma'_{k-1}]} R'_{k-1}$, clearly the map $T_{k-1} \mapsto T_k$ induces an equivariant homeomorphism $R'_{k-1} \to R_k$, and so $[R'_{k-1}]=[R_k]$. Also, since $\sigma_{k-1}$ is the maximal natural subgraph of $\sigma'_{k-1}$, the identity map on $T_{k-1}$ induces a collapse map $R_{k-1} \to R'_{k-1}$ which is a bijection on natural vertices and which, on each natural edge of $R_{k-1}$, collapses an initial and/or terminal segment and is otherwise injective. It follows that the collapse map $R_{k-1} \mapsto R'_{k-1}$ is equivariantly homotopic to a conjugacy, and so $[R_{k-1}]=[R'_{k-1}]=[R_k]$.

\medskip

It remains to prove \pref{ItemOneEdgeExistsSome}$\implies$\pref{ItemAIEexistsSome}. After rewording, this says that if $f \from S \to T$ is a foldable map of free splittings, and if there exists a one-edge free splitting $R$ such that $d(R,S), d(R,T) \le 1$, then $f \from S \to T$ has an almost invariant edge. Fix an oriented natural edge $e_R \subset R$ with initial and terminal vertices $r_\pm$, and oriented natural edges $e_S \subset S$, $e_T \subset T$ with initial and terminal vertices $s_\pm$, $t_\pm$ respectively, so that there are collapse maps $S,T \mapsto R$ which collapse the complement of the orbits of $e_S,e_T$ and which take $e_S,e_T$ homeomorphically to $e_R$. We shall prove that $e_S,e_T$ is an almost invariant edge for $f \from S \to T$. 

There is a component decomposition $R \setminus e_R = R_- \disjunion R_+$ where $R_\pm$ contains the vertex $r_\pm$ and there are corresponding component decompositions $S \setminus e_S = S_- \disjunion S_+$, $T \setminus e_T = T_- \disjunion T_+$ so that $S_\pm$, $T_\pm$ are the inverse images of $R_\pm$, respectively, under the collapse maps $S,T \mapsto R$ (in general the ``$\pm$'' notation means ``$+$ or $-$, respectively''; for instance ``$S_\pm$ is the inverse image of $R_\pm$'' means ``$S_+$, $S_-$ is the inverse image of $R_+$, $R_-$, respectively''). Note that $R_\pm$, $S_\pm$, $T_\pm$ are natural subgraphs of $R,S,T$, respectively. Also, $r_\pm$ is the unique point on the topological frontier of $R_\pm$ in $R$, and similarly for $S_\pm$, $T_\pm$. Also, each vertex in each of these subgraphs has valence~$\ge 2$ within the subgraph: in, say, $R_-$ this is obvious for all interior vertices, and the frontier vertex $r_-$ is a natural vertex in $R$ having only one $R$-direction not in $R_-$, namely the direction of $e_R$. 

It suffices to prove that $f(S_\pm) \subset T_\pm$, which immediately implies that $e_S,e_T$ is an almost invariant edge for $f \from S \to T$. Assuming that either $f(S_-) \not\subset T_-$ or $f(S_+) \not\subset T_+$, we shall produce a contradiction. The arguments are similar in either case, so we shall assume that $f(S_-) \not\subset T_-$.

Given a free splitting $F \act U$ and a nontrivial $\gamma \in F$ let $\alpha_U(\gamma)$ denote either the axis of $\gamma$ in $U$ or the unique vertex of $\gamma$ fixed by $U$. Let $F_\pm$ denote the set of nontrivial $\gamma \in F$ such that $\alpha_R(\gamma) \subset R_\pm$. 
Note that for each natural edge $e \subset R_\pm$ there exists $\gamma \in F_\pm$ whose axis under the action $F \act R$ contains $e$. It follows that 
$$R_\pm = \bigcup_{\gamma \in F_\pm} \alpha_R(\gamma)
$$
Note also that 
\begin{enumerate}
\item\label{ItemFpm} 
$\displaystyle S_\pm = \bigcup_{\gamma \in F_\pm} \alpha_S(\gamma) \quad\text{and}\quad T_\pm = \bigcup_{\gamma \in F_\pm} \alpha_T(\gamma)$
\end{enumerate}
To prove this for $S_-$, say, note first that the collapse map $S \mapsto R$ takes $S_\pm$ to~$R_\pm$ and its restriction to $\alpha_S(\gamma)$ has image $\alpha_R(\gamma)$ for each $\gamma \in F$. If $\alpha_S(\gamma) \subset S_-$ then $\alpha_R(\gamma) \subset R_-$ and hence $\gamma \in F_-$, and since the axes contained in $S_-$ cover $S_-$ we get one inclusion $S_- \subset \cup_{\gamma \in F_-} \alpha_S(\gamma)$. For the other inclusion, if $\alpha_S(\gamma) \not\subset S_-$ then either $\alpha_S(\gamma)$ crosses $e_S$ and so $\alpha_R(\gamma)$ crosses $e_r$, or $\alpha_S(\gamma) \subset S_+$ and so $\alpha_R(\gamma) \subset R_+$, and in either case $\gamma \not\in F_-$.

Next we show:
\begin{enumeratecontinue}
\item\label{ItemCrossingBound}
There exists a finite number $A \ge 0$ such that $T_- \subset f(S_-) \subset N_A(T_-)$
\end{enumeratecontinue}
Applying the inclusion $f(\alpha_S(\gamma)) \supset \alpha_T(\gamma)$ to all $\gamma \in F_-$ and using \pref{ItemFpm} we obtain one inclusion $T_- \subset f(S_-)$. The opposite inclusion follows by applying the \emph{bounded cancellation lemma} to the map $f \from S \to T$. The version of the lemma that we need comes from \BookZero, Lemma~3.1, and although the hypothesis there requires that $F \act S$ be a properly discontinuous action (called there a ``free simplicial tree''), the first paragraph of that proof works exactly as stated for a map like $f$ that factors as a fold sequence. The conclusion of that first paragraph is that there exists $A$, a \emph{bounded cancellation constant} for $f$, such that for any vertices $x,y \in S$, in the tree $T$ the set $f[x,y]$ is contained in the $A$ neighborhood of the segment $[f(x),f(y)]$. Applying this to our situation, we conclude that for any $\gamma \in F$ we have $f(\alpha_S(\gamma)) \subset N_A(\alpha_T(\gamma))$. Applying this to all $\gamma \in F_-$ and using~\pref{ItemFpm}, it follows that $f(S_-) \subset N_A(T_-)$, completing the proof of~\pref{ItemCrossingBound}.

We show that the only way for $f(S_-)$ to cross $e_T$ is to do so rather sharply:
\begin{enumeratecontinue}
\item\label{ItemThornProp} If $f(S_-) \not\subset T_-$ then $f(S_-) = T_- \union [t_-,f(s_-)]$. Recalling that $t_-$ is the unique frontier point of $T_-$, it follows that $T_- \intersect [t_-,f(s_-)]=\{t_-\}$.
\end{enumeratecontinue}
To see why, by \pref{ItemCrossingBound} the tree $f(S_-) \setminus T_-$ has finite diameter, by assumption of \pref{ItemThornProp} it is nondegenerate, and so it has at least two vertices of valence~$1$, at least one being distinct from $t_-$. The graph $f(S_-)$ therefore has at least one vertex of valence~$1$. But~$s_-$ is the unique frontier vertex of $S_-$ so by the First Derivative Test the point $f(s_-)$ is the unique valence~$1$ vertex of $f(S_-)$.
Combining this with $T_- \subset f(S_-)$, \pref{ItemThornProp} follows immediately. 

But from \pref{ItemThornProp} we deduce that $f \from S \to T$ has at most~$2$ gates at the natural vertex $s_-$, because all of the directions at $s_-$ distinct from the direction of $e_S$ are mapped by $f$ to a single direction at $f(s_-)$, namely, the direction of the segment $[f(s_-),t_-]$. This contradicts that a foldable map has at least~$3$ gates at every natural vertex.
\end{proof}

\subsection{Definitions and properties of free splitting units}
\label{SectionFSUDefsPropsApps}

Given a fold sequence $S_0 \xrightarrow{f_1} S_1 \xrightarrow{f_2} \cdots \xrightarrow{f_K} S_K$, we shall first define what it means for $S_i,S_j$ to ``differ by $<1$ free splitting unit'' for $i,j \in 0,\ldots,K$, and we prove an appropriate stability result for this definition. With this in hand, for any $i,j \in 0,\ldots,K$ we then define the number of free splitting units between $S_i$ and $S_j$. Lemma~\ref{LemmaFirstFSUBound} proves that the free splitting parameterization along the fold sequence is a Lipschitz parameterization with respect to distance in $\FS'(F)$. 

\begin{definition}[$<1$ free splitting unit]  \label{DefLessThanOneFSU}
Given a fold sequence $S_0 \xrightarrow{f_1} \cdots \xrightarrow{f_K} S_K$ and $0 \le i < j \le K$, we say that \emph{$S_i,S_j$ differ by $<1$ free splitting unit} if there exists a commutative diagram of the form
$$\xymatrix{
T_i \ar[r] \ar[d]_{[\tau_i]}  
                 & T_{i+1} \ar[r] \ar[d]_{[\tau_{i+1}]}    
                                & \cdots \ar[r]           & T_{j-1} \ar[r] \ar[d]^{[\tau_{j-1}]}   
                                                                           & T_j  \ar[d]^{[\tau_j]}\\
S'_i \ar[r]                         
                 & S'_{i+1} \ar[r]                            
                                & \cdots \ar[r]           & S'_{j-1} \ar[r]   
                                                                            & S'_j & \\
S_i \ar[r]_{f_{i+1}}\ar[u]^{[\sigma_i]}  
                 & S_{i+1} \ar[r]_{f_{i+2}} \ar[u]^{[\sigma_{i+1}]} 
                                 & \cdots \ar[r]_{f_{j-1}}  & S_{j-1} \ar[r]_{f_{j}} \ar[u]_{[\sigma_{j-1}]}
                                                                            & S_j \ar[u]_{[\sigma_j]} \\
}$$
whose top and bottom rectangles are combing rectangles, so that foldable sequence $T_i \mapsto\cdots\mapsto T_j$ on the top row has either an invariant natural blue-red decomposition of constant complexity or an almost invariant edge (by combining Lemmas~\ref{LemmaBRNatural} and~\ref{LemmaBROneB}, this holds if and only if the foldable map $T_i \mapsto T_j$ has either an invariant natural blue-red decomposition of constant complexity or an almost invariant edge). To complete the definition, we symmetrize the concept by requiring that $S_j,S_i$ differ by $<1$ free splitting unit if and only if $S_i,S_j$ differ by $<1$ free splitting unit.
\end{definition}

\medskip

The following is an immediate consequence of the definition, by restricting to the appropriate subdiagram of the above commutative diagram:

\begin{lemma}[Stability of free splitting units] Given a fold sequence $S_0 \mapsto\cdots\mapsto S_K$ and $0 \le i \le i' \le j' \le j \le K$, if $S_i,S_j$ differ by $<1$ free splitting unit then $S_{i'}$, $S_{j'}$ differ by $<1$ free splitting unit.
\qed\end{lemma}

Using these concepts we get a diameter bound as follows:

\begin{lemma}\label{LemmaFirstFSUBound}
Given a fold sequence $S_0 \mapsto\cdots\mapsto S_K$ and $0 \le i \le j \le K$, if $S_i,S_j$ differ by $<1$ free splitting unit then $\diam\{S_i,\ldots,S_j\} \le 8$.
\end{lemma}

\begin{proof} Consider the commutative diagram in the definition of $<1$ free splitting unit. Combining Lemmas~\ref{LemmaBRNatural} and~\ref{LemmaBROneB}, it follows that $\diam\{T_i,\ldots,T_j\} \le 4$. Since $d(S_k,T_k) \le 2$ for each $k$, we have $\diam\{S_i,\ldots,S_j\} \le 8$.
\end{proof}

\paragraph{The coarse retract axiom.} As an application of the concepts of free splitting units, particularly Lemma~\ref{LemmaBROneB}, we now prove that our definition for projecting $\FS'(S)$ onto fold paths satisfies the first of the three Masur-Minsky axioms:

\begin{proposition}
\label{PropCoarseRetract}
For any fold sequence $S_0 \mapsto \cdots \mapsto S_K$, the associated projection map $\pi \from \FS'(F) \to [0,\ldots,K]$ satisfies the \emph{Coarse Retraction} axiom with the constant $c = 6$: for any $i=0,\ldots,K$ we have $i \le \pi(S_i)$ and the diameter of the set $\{S_i,\ldots,S_{\pi(S_i)}\}$ is $\le 6$. Furthermore, there is $<1$ free splitting unit between $S_i$ and $S_{\pi(S_i)}$. 
\end{proposition}

\begin{proof} We start by noticing that a projection diagram from $S_i$ to $S_0 \mapsto\cdots\mapsto S_K$ of depth $i$ certainly exists, where all vertical arrows are conjugacies and all collapse graphs are empty; see Figure~\ref{FigureCoarseRetractCombRect1}. 
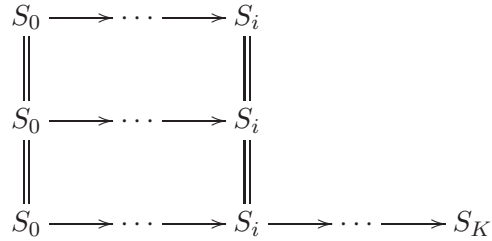
\begin{figure}[h]
\centerline{\xymatrix{
S_0 \ar[r] \ar@{=}[d] & \cdots \ar[r] & S_i \ar@{=}[d] \\
S_0 \ar[r]                  & \cdots \ar[r] & S_i                   \\
S_0 \ar[r] \ar@{=}[u] & \cdots \ar[r] & S_i \ar[r] \ar@{=}[u] & \cdots \ar[r] & S_K \\
}}
\caption{A projection diagram from $S_i$ to $S_0 \mapsto\cdots\mapsto S_K$ of depth $i$.}
\label{FigureCoarseRetractCombRect1}
\end{figure}%
By definition, $\pi(S_i)$ is the largest integer in the set $[0,\ldots,K]$ such that (after rechoosing the free splitting $F \act S_i$ in its conjugacy class, and after rechoosing the fold sequence $S_0 \mapsto\cdots\mapsto S_K$ in its conjugacy class) a projection diagram from $S_i$ to $S_0 \mapsto\cdots\mapsto S_K$ of depth $\pi(S_i)$ exists. This largest integer therefore satisfies $i \le \pi(S_i)$ and yields a projection diagram as in Figure~\ref{FigureCoarseRetractCombRect2}.
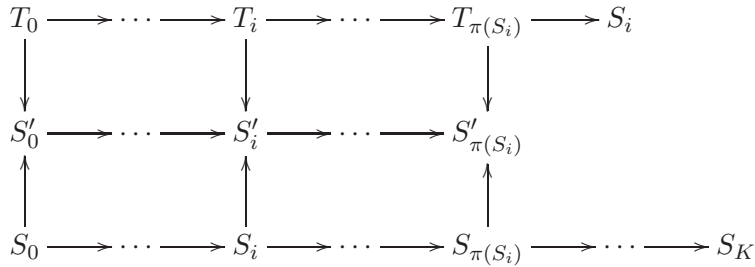
\begin{figure}\centerline{\xymatrix{
T_0 \ar[r] \ar[d] & \cdots \ar[r] & T_i \ar[r] \ar[d] & \cdots \ar[r] & T_{\pi(S_i)} \ar[r] \ar[d] & S_i \\
S'_0 \ar[r]          & \cdots \ar[r] & S'_i \ar[r]        & \cdots \ar[r] & S'_{\pi(S_i)} \\
S_0 \ar[r] \ar[u] & \cdots \ar[r] & S_i \ar[r] \ar[u] & \cdots \ar[r] & S_{\pi(S_i)} \ar[r] \ar[u]  & \cdots \ar[r] & S_K \\
}}
\caption{A maximal depth projection diagram from $S_i$ to $S_0 \mapsto\cdots\mapsto S_K$.}
\label{FigureCoarseRetractCombRect2}
\end{figure}
Let $e' \subset S'_i$ be any natural edge, and let $R$ be the one-edge free splitting obtained from $S'_i$ by collapsing the complement of the orbit of $e'$. Then we have collapse maps $T_i \mapsto S'_i \mapsto R$ and $S_i \mapsto S'_i \mapsto R$, proving that $d(T_i,R) \le 1$ and $d(S_i,R) \le 1$. Applying Lemma~\ref{LemmaBROneB}, the foldable sequence on the top row from $T_i$ to $S_i$ has an almost invariant edge, and by restriction there is an almost invariant edge from $T_i$ to $T_{\pi(S_i)}$. Also by Lemma~\ref{LemmaBROneB}, the set $\{T_i,\ldots,T_{\pi(S_i)}\}$ has diameter $\le 2$, and since $d(S_k,T_k) \le 2$ for each $k$ it follows that $\diam\{S_i,\ldots,S_{\pi(S_i)}\} \le 6$. And by Definition~\ref{DefLessThanOneFSU}, it follows that there is $<1$ free splitting unit between $S_i$ and $S_{\pi(S_i)}$.
\end{proof}

\begin{definition}[General count of free splitting units]
\label{DefGeneralFSU}
\index{free splitting unit}%
\hfill\break
Given a fold sequence $S_0 \mapsto \cdots \mapsto S_K$, for $0 \le i,j \le K$ we say that $S_i,S_j$ \emph{differ by $\ge 1$ free splitting unit} if they do not differ by $<1$ free splitting unit. Then, for $0 \le I \le J \le K$, the \emph{number of free splitting units between $S_I$ and $S_J$} is defined to be the maximum integer $\Upsilon \ge 0$ for which there exists a sequence of integers $I \le i(0) < \cdots < i(\Upsilon) \le J$ of length $\Upsilon+1$, parameterized by integers $0 \le u \le \Upsilon$, such that if $1 \le u \le \Upsilon$ then $S_{i(u-1)}, S_{i(u)}$ differ by $\ge 1$ free splitting unit. Notice that our definitions are consistent in that $\Upsilon=0$ if and only if, following the earlier definition, there is $<1$ free splitting unit between $S_I$ and $S_J$. Also, we symmetrize the definition by saying that the number of free splitting units between $S_J$ and $S_I$ equals the number between $S_I$ and $S_J$.
\end{definition}

\subparagraph{Remark.} In counting the number of free splitting units between $S_i$ and $S_j$, although this number depends on the fold sequence $S_i \mapsto\cdots\mapsto S_j$ that connects $S_i$ to $S_j$, that fold sequence will always be clear by context and we suppress this dependence in our terminology. Notice that this number does \emph{not} depend on any other details of an ambient fold sequence of which $S_i \mapsto\cdots\mapsto S_j$ might be a subinterval. In particular, the number of free splitting units between $S_i$ and $S_j$ is unaffected if the ambient fold sequence is truncated by deleting an initial segment before $S_i$ and/or a terminal segment after $S_j$.

\medskip

Notice that with the notation as above, if $0 \le u \le v \le \Upsilon$ then the number of free splitting units between $S_{i(u)}$ and $S_{i(v)}$ equals $v-u$. To see why, first note that this number is $\ge v-u$ by construction. If it were $\ge v-u+1$ then one could alter the sequence $i(0) < \cdots < i(\Upsilon)$ by removing the entries $i(u), \ldots, i(v)$ and inserting an increasing sequence of $\ge v-u+2$ entries in the interval $[i(u),i(v)]$ which amongst themselves have $\ge 1$ free splitting unit between any consecutive two. By \emph{Stability of Free Splitting Units} the new entries would have $\ge 1$ free splitting units with the remaining entries outside of the interval $[i(u),i(v)]$. The new sequence would therefore still have $\ge 1$ free splitting units between consecutive terms, but would have length $\ge \Upsilon+2$, contradicting the maximality of $\Upsilon$. 

One can count free splitting units between $S_I$ and $S_J$ in several ways. For example, define the \emph{front greedy subsequence}\index{front greedy subsequence} from $I$ to $J$ to be the sequence $I=j(0) < j(1) < \cdots < j(\Upsilon') \le J$ obtained by induction as follows: assuming $j(u)$ is defined, and assuming $S_{j(u)}$ and $S_J$ differ by $\ge 1$ free splitting unit, let $j(u+1)$ be the least integer $>j(u)$ such that $S_{j(u)}$ and $S_{j(u+1)}$ differ by $\ge 1$ free splitting unit; the sequence stops when $S_{j(\Upsilon')}, S_J$ differ by $<1$ free splitting unit. We claim that $\Upsilon'$, the length of the front greedy subsequence, is equal to the number of free splitting units between $S_I$ and $S_J$. When $S_I,S_J$ differ by $<1$ free splitting unit the claim is immediate. In the case where $S_I,S_J$ differ by $\ge 1$ free splitting unit, clearly $\Upsilon' \ge 1$; then, noting by stability that $S_{j(u)}$, $S_{j(v)}$ differ by $\ge 1$ free splitting unit for $1 \le u < v \le \Upsilon'$, and using maximality of $\Upsilon$, it follows that $\Upsilon \ge \Upsilon'$. For the opposite inequality we argue by contradiction assuming that $\Upsilon \ge \Upsilon'+1$. Consider any subsequence $I \le i(0) < i(1) < \cdots < i(\Upsilon) \le J$ such that $S_{i(u-1)}, S_{i(u)}$ differ by $\ge 1$ free splitting unit for each $u=1,\ldots,\Upsilon$. By maximality of $\Upsilon$ it follows that between each of the pairs $S_I, S_{i(0)}$ and $S_{i(\Upsilon)}, S_J$ there is $<1$ free splitting unit. By stability it follows that between $S_I$ and $S_{i(1)}$ there is $\ge 1$ free splitting unit. By definition of $j(1)$ we have $j(1) \le i(1)$. By stability it follows that $S_{j(1)}$ and $S_{i(2)}$ differ by $\ge 1$ free splitting unit from which it follows that $j(2) \le i(2)$. Continuing by induction we see that $j(u) \le i(u)$ for $u=1,\ldots,\Upsilon'$. But since $j(\Upsilon') \le i(\Upsilon') < i(\Upsilon'+1) \le i(\Upsilon) \le J$ and since $S_{i(\Upsilon')}, S_{i(\Upsilon'+1)}$ differ by $\ge 1$ free splitting unit, it follows by stability that $S_{j(\Upsilon')}, S_J$ differ by $\ge 1$ free splitting unit, which contradicts the definition of $\Upsilon'$.

In a similar fashion one proves that the number of free splitting units is equal to the length of the \emph{back greedy subsequence}\index{back greedy subsequence} $I \le \ell(\Upsilon'') < \ell(\Upsilon''-1) < \cdots < \ell(1) < \ell(0) = J$, defined as follows: assuming by induction that $\ell(u)$ is defined and that $S_I$ and $S_{\ell(u)}$ differ by $\ge 1$ free splitting unit, $\ell(u+1)$ is the greatest integer $<\ell(u)$ such that $S_{\ell(u+1)}$ and $S_{\ell(u)}$ differ by $\ge 1$ free splitting unit; the sequence stops when $S_I$, $S_{\ell(\Upsilon'')}$ differ by $<1$ free splitting unit.

The following result says that a fold path which is parameterized by free splitting units is a coarse Lipschitz path in $\FS(F)$:

\begin{lemma}\label{LemmaUnitsLipschitz} For any fold path $S_0 \mapsto\cdots\mapsto S_K$ and any $0 \le I \le J \le K$, if the number of free splitting units between $S_I$ and $S_J$ equals $\Upsilon$ then the diameter in $\FS'(F)$ of the set $\{S_I,\ldots,S_J\}$ is $\le 10 \Upsilon + 8$.
\end{lemma}

\begin{proof} If $\Upsilon=0$, that is if $S_I,S_J$ differ by $<1$ free splitting unit, then by Lemma~\ref{LemmaFirstFSUBound} we have $\diam\{S_I,\ldots,S_J\} \le 8$. 

If $\Upsilon \ge 1$, from $S_I$ to $S_J$ let $I=i(0) <  \cdots < i(\Upsilon) \le J$ be the front greedy sequence. For $u=1,\ldots,\Upsilon$, the free splittings $S_{i(u-1)}$ and $S_{i(u)-1}$ differ by $<1$ free splitting unit, and so $\diam\{S_{i(u-1)},\ldots,S_{i(u)-1}\}\le 8$. By Lemma~\ref{LemmaFoldDistance} we have $d(S_{i(u)-1},S_{i(u)}) \le 2$ and so $\diam\{S_{i(u-1)}, \ldots, S_{i(u)}\} \le 10$. It follows in turn that $\diam\{S_I=S_{i(0)},\ldots,S_{i(\Upsilon)}\} \le 10 \Upsilon$. Since $S_{i(\Upsilon)},S_J$ differ by $<1$ free splitting unit we have $\diam\{S_{i(\Upsilon)},\ldots,S_J\} \le 8$, and putting it all together, $\diam \{S_I,\ldots,S_J\} \le 10\Upsilon + 8$.
\end{proof}

We also need the following lemma which gives a coarse triangle inequality for free splitting units within a fold path:

\begin{lemma}\label{LemmaFSUTriangleInequality}
Given a fold path $S_0 \mapsto\cdots\mapsto S_K$ and $i,j,k \in \{0,\ldots,K\}$, if $\Upsilon_1$ is the number of free splitting units between $S_i$ and $S_j$ and $\Upsilon_2$ is the number between $S_j$ and $S_k$ then the number $\Upsilon$ between $S_i$ and $S_k$ satisfies $\Upsilon \le \Upsilon_1 + \Upsilon_2 + 1$.
\end{lemma}

\begin{proof} In the case where $j$ is between $i$ and $k$, using \emph{symmetry of free splitting units} we may assume that $i \le j \le k$. Let $i=i(0)<\cdots<i(\Upsilon)\le k$ be the front greedy sequence from $S_i$ to $S_k$. Clearly the front greedy sequence from $S_i$ to $S_j$ is an initial segment, implying that $i(\Upsilon_1) \le j$ and $i(\Upsilon_1+1) > j$, and so we have a subsequence $S_{i(\Upsilon_1+1)},\ldots,S_{i(\Upsilon)}$ of $S_j,\ldots,S_k$ with the property that between any two adjacent elements of this subsequence there is $\ge 1$ free splitting unit. 
By Definition~\ref{DefGeneralFSU} and the hypothesis on $\Upsilon_2$, the length of this subsequence is therefore $\le \Upsilon_2+1$, giving us $\Upsilon - \Upsilon_1 \le \Upsilon_2+1$. 

In the case where $j > \max\{i,k\}$, again using symmetry we may assume $i \le k < j$. Let $i=i(0) < \cdots < i(\Upsilon_1) \le j$ be the front greedy subsequence between $S_i$ and $S_j$. Again the front greedy subsequence between $S_i$ and $S_k$ is an initial subsegment and so $\Upsilon \le \Upsilon_1 \le \Upsilon_1 + \Upsilon_2 + 1$.

In the case where $j < \min\{i,k\}$, using symmetry we assume $j < k \le i$, and we proceed similarly using the back greedy subsequence between $S_j$ and $S_i$.
\end{proof}

\section{Proof of the Main Theorem}
\label{SectionMainProof}
We begin with a quick sketch of the proof.

Consider a free splitting $T$, a fold sequence $S_0 \mapsto\cdots\mapsto S_K$, and a maximal depth projection diagram which defines the projection ${k_T} \in \{0,\ldots,K\}$ from $T$ to this fold sequence. The form of this projection diagram can be viewed in Section~\ref{SectionCombingRectangles}, Figure~\ref{FigureProjDiagram}, the top row of which is a foldable sequence $T_0 \mapsto\cdots\mapsto T_{k_T} \mapsto T$. We then apply Lemma~\ref{LemmaFoldSequenceConstruction} to factor the final foldable map $T_{k_T} \mapsto T$ as a fold sequence of the form $T_{k_T} \mapsto\cdots\mapsto T_L=T$, which we then paste into the foldable sequence on the top row of the projection diagram to get an ``augmented'' projection diagram. Figure~\ref{FigureMaxProjDiagram} shows the original, unaugmented projection diagram and the augmented version in the same picture. Note that the top row of the augmented projection diagram is the foldable sequence $T_0 \mapsto\cdots\mapsto T_{k_T} \mapsto\cdots\mapsto T_L=T$. See Section~\ref{SectionStatmentFSUContraction} for more details on augmented projection diagrams.

Consider also a geodesic in the \nb{1}skeleton of $\FS'(F)$ starting with $T$ and ending with some free splitting~$R$. This geodesic is a zig-zag path; suppose for concreteness that it starts with a collapse and ends with an expand, $T = T_L^0 \collapses T_L^1 \expands T_L^2 \collapses T_L^3 \expands T_L^4 \collapses \cdots \expands T_L^{D}=R$, and so $D = d(T,R)=d(T^0_L,T^D_L)$ is even. By combing the foldable sequence $T_0 \mapsto\cdots\mapsto T_{k_T} \mapsto\cdots\mapsto T_L=T$ across each collapse and expansion in this zig-zag path one at a time, we obtain ``The Big Diagram, Step 0'' depicted in Section~\ref{SectionProofFSUContraction}, Figure~\ref{FigureBigDiagram0}, which is built out of the projection diagram and an $L \cross D$ rectangle composed of $D$ combing rectangles. Note that the interior even terms along the zig-zag path, the free splittings $T_L^2, T_L^4, \ldots, T_L^{D-2}$, are ``peaks'' of the zig-zag. The big $L \cross D$ rectangle has the form of a corrugated aluminum roof in which the interior even horizontal rows are peaks of the corrugations. 

Our technique can be described as ``pushing down the peaks''. In brief, we prove that if one backs up from $T_L$ to some earlier term in the fold path $T_{k_T}\mapsto\cdots\mapsto T_L$, moving back a certain fixed number of free splitting units, then the big diagram can be simplified by pushing the first corrugation peak down, reducing the number of corrugation peaks by~$1$, as shown in ``The Big Diagram, Step 1''. These ``back up --- pushdown'' arguments are found in Section~\ref{SectionPushingDownPeaks}. Therefore, if the number of free splitting units between $T_{k_T}$ and $T_L$ is greater than a certain multiple of the number of peaks in the zig-zag path from $T_L$ to $T^D_L$ then the number of corrugation peaks in the Big Diagram can be reduced to zero. With one final ``back up --- push down'' step that uses up some of the original projection diagram for $T_L$, one obtains a projection diagram from $R$ to $S_0 \mapsto\cdots\mapsto S_K$, from which one concludes that the projection of $R$ to $S_0\mapsto\cdots\mapsto S_K$ is not much further back (measured in free splitting units) than $S_{k_T}$ which is the projection of $T$. 

The exact statement proved by these arguments is contained in Proposition~\ref{PropFSUContraction} which can be regarded as a reformulation of the \emph{Coarse Lipchitz} and \emph{Desymmetrized strong contraction} axioms in terms of free splitting units, and which quickly implies those axioms and the main theorem as shown in Section~\ref{SectionStatmentFSUContraction}. The proof of Proposition~\ref{PropFSUContraction} itself is carried out in Sections~\ref{SectionPushingDownPeaks} and~\ref{SectionProofFSUContraction}.

\subsection{Desymmetrized strong contraction reformulated and applied}
\label{SectionStatmentFSUContraction}

In Proposition~\ref{PropFSUContraction} we reformulate the \emph{Coarse Lipschitz} and \emph{Desymmetrized strong contraction} axioms as a joint statement expressed in terms of free splitting units. The proposition will be proved in later subsections of Section~\ref{SectionMainProof}.

After stating the proposition, we use it to finish off the proof of the main theorem. We also use it to prove Proposition~\ref{PropFoldPathQuasis} which describes precisely how to reparameterize fold paths in terms of free splitting units so as to obtain uniform quasigeodesics in~$\FS'(F)$.  

To set up Proposition~\ref{PropFSUContraction}, consider any fold path $S_0 \mapsto \cdots \mapsto S_K$, any free splitting $F \act T$ and any projection diagram of maximal depth $\pi(T)=k_T \in [0,\ldots,K]$ as depicted in Figure~\ref{FigureMaxProjDiagram}. Applying Proposition~\ref{LemmaFoldSequenceConstruction}, we may factor the foldable map $f \from T_{k_T} \to T$ as a fold sequence, and then replace $f$ with this factorization in the top line of the projection diagram, to obtain a sequence of maps
$$T_0 \xrightarrow{f_1} \cdots\xrightarrow{f_{k_T}} T_{k_T} \xrightarrow{f_{k_T+1}} \cdots \xrightarrow{f_L} T_L=T
$$
This sequence of maps is still foldable --- if $0 \le k \le k_T$ then $f^k_L$ is foldable by virtue of being a map in the original foldable sequence on the top line of the unaugmented projection diagram; and if $k_T < k \le L$ then $f^k_L$ is foldable by virtue of being a map in the newly inserted fold sequence (note that if one replaces any but the last map in a foldable sequence with a fold factorization, this trick does not work --- the resulting sequence need not be foldable). We therefore obtain the \emph{augmented projection diagram}\index{projection diagram!augmented} from $T$ to $S_0\mapsto\cdots\mapsto S_K$ of maximal depth, as depicted also in the Figure~\ref{FigureMaxProjDiagram}.  
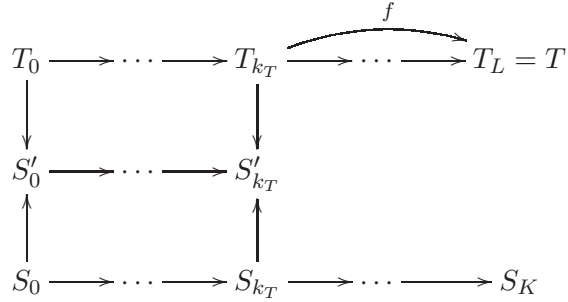
\begin{figure}[h]
$$\xymatrix{
T_0 \ar[r] \ar[d] & \cdots \ar[r] & T_{k_T} \ar[r] \ar[d] \ar@/^1pc/[rr]^f  & \cdots \ar[r] & T_L = T \\
S'_0 \ar[r]          & \cdots \ar[r] & S'_{k_T} \\
S_0 \ar[r] \ar[u] & \cdots \ar[r] & S_{k_T} \ar[r] \ar[u]  & \cdots \ar[r] & S_K \\
}$$
\caption{An \emph{augmented} projection diagram from $T$ to $S_0\mapsto\cdots\mapsto S_K$ of maximal depth~$k_T$ (with the straight arrows from $T_{k_T}$ to $T$) is obtained from a maximal depth projection diagram (with the curved arrow from $T_{k_T}$ to $T$ labelled $f$) by inserting a fold sequence factorization of the foldable map $f \from T_{k_T} \to T$. After this insertion the whole sequence $T_0 \mapsto\cdots\mapsto T_{k_T} \mapsto\cdots\mapsto T_L=T$ in the top row is still a foldable sequence.}
\label{FigureMaxProjDiagram}
\end{figure}

\begin{proposition}[Strong contraction in terms of free splitting units] 
\label{PropFSUContraction} 
\quad \\
Letting $b_1 = 4 \rank(F)-3$, the following holds. Consider a fold path $S_0 \mapsto \cdots \mapsto S_K$, a free splitting $F \act T$ with projection $\pi(T)=k_T \in [0,\ldots,K]$, and an augmented projection diagram of maximal depth $k_T$ as notated in Figure~\ref{FigureMaxProjDiagram}. Let $\Upsilon$ be the number of free splitting units between $T_{k_T}$ and $T_L=T$. If $F \act R$ is a free splitting such that $d(T,R) \le \max\left\{2 \lfloor \Upsilon / b_1 \rfloor,1\right\}$, and if the number of free splitting units between $S_0$ and $S_{k_T}$ is $\ge b_1$, then there exists $l \in [0,\pi(R)]$ such that the number of free splitting units between $S_l$ and $S_{k_T}$ is~$\le b_1$.
\end{proposition}

\subparagraph{Remark.} To put it more plainly, Proposition~\ref{PropFSUContraction} says that the projection of $R$ to the fold path $S_0 \mapsto\cdots\mapsto S_K$ is no farther to the left of the projection of $T$ than a bounded number of free splitting units, as long as $d(T,R)$ is at most some bounded proportion of the number~$\Upsilon$. One can think of the number $\Upsilon$ as being a stand-in for the distance from $T$ to the fold path $S_0 \mapsto\cdots\mapsto S_K$ (a posterior one sees that $\Upsilon$ is indeed quasicomparable to that distance). Notice that the proposition does not apply if no projection diagram exists for $T$, nor if the number of free splitting units between $S_0$ and $S_{k_T}$ is too small; in either of these cases the projection of $T$ is close to $S_0$ in $\FS'(F)$. These special situations are handled in Case~1 of the proof of the Main Theorem. 

Note that Proposition~\ref{PropFSUContraction} is trivially true when $\pi(R) \ge k_T$, by taking $l=k_T$. The real meat of the proposition is when $\pi(R) < k_T$.

\bigskip

Proposition~\ref{PropFSUContraction} is proved in Sections~\ref{SectionPushingDownPeaks} and~\ref{SectionProofFSUContraction}. For the rest of Section~\ref{SectionStatmentFSUContraction} we shall apply Proposition~\ref{PropFSUContraction} to prove first the Main Theorem and then Proposition~\ref{PropFoldPathQuasis} regarding quasigeodesics in $\FS'(F)$.

\begin{proof}[Proof of the Main Theorem] \qquad As we showed earlier, Proposition~\ref{PropProjToFoldPath} implies Proposition~\ref{PropFoldContractions} which implies the Main Theorem. To prove Proposition~\ref{PropProjToFoldPath} we must prove that the projections to fold paths in $\FS'(F)$ satisfy the \emph{Coarse retraction}, \emph{Coarse Lipschitz}, and \emph{Desymmetrized strong contraction} axioms given in Section~\ref{SectionMasurMinsky}, with uniform constants depending only on $\rank(F)$. In Proposition~\ref{PropCoarseRetract} we already did this for the \emph{Coarse retraction} axiom. We turn to the other two axioms. 

Fix the fold path $S_0 \mapsto\cdots\mapsto S_K$ and free splittings $F\act T,R$ with projections $\pi(T),\pi(R) \in [0,\ldots,K]$. For verifying both the \emph{Coarse Lipschitz} and \emph{Desymmetrized strong contraction} axioms we may assume that $\pi(R) \le \pi(T)$. We seek to bound the diameter in $\FS'(F)$ of the set $\{S_{\pi(R)},\ldots,S_{\pi(T)}\}$. If $\pi(T)=0$ then $\pi(R)=0$ and we are done. Otherwise, after rechoosing $T$ in its conjugacy class and rechoosing $S_0 \mapsto\cdots\mapsto S_K$ in its equivalence class, we may choose an augmented maximal depth projection diagram for $T$ and $S_0 \mapsto\cdots\mapsto S_K$ as notated in Figure~\ref{FigureMaxProjDiagram}. Let $\Upsilon$ be the number of free splitting units between $T_{k_L}$ and $T_L = T$.

Throughout the proof we denote the constants from Lemma~\ref{LemmaUnitsLipschitz} as
$$L=10, \quad C=8
$$
It follows that along any fold path, for any two terms of that path between which the number of free splitting units is at most
$$b_1 = 4 \rank(F)-3
$$ 
the diameter in $\FS'(F)$ of the segment between those two terms is at most
$$c = L b_1 + C = 40 \rank(F) - 22
$$
This is the value of $c$ that will be used in verifying the two axioms.

\subparagraph{Case 1:} Suppose that the number of free splitting between $S_0$ and $S_{\pi(T)}$ is $<b_1$. Applying the inequality $0 \le \pi(R) \le \pi(T)$ together with \emph{Stability of free splitting units}, it follows that the number of free splitting units between $S_{\pi(R)}$ and $S_{\pi(T)}$ is~$< b_1$. By Lemma~\ref{LemmaUnitsLipschitz} the diameter of the set $\{S_{\pi(R)},\ldots,S_{\pi(T)}\}$ is~$\le c$, which is the common conclusion of the \emph{Coarse Lipschitz} and \emph{Desymmetrized strong contraction} axioms. In this case, those axioms are verified using any values of $a,b$.

\subparagraph{Case 2:} Suppose that the number of free splitting units between $S_0$ and $S_{\pi(T)}$ is $\ge b_1 > 0$. 

We claim that the following statement holds:
\begin{itemize}
\item[$(*)$] If $d(T,R) \le \max\left\{2 \lfloor \Upsilon / b_1 \rfloor,1\right\}$ then the number of free splitting units between $S_{\pi(R)}$ and $S_{\pi(T)}$ is $\le b_1$, and so the diameter in $\FS'(F)$ of the set $\{S_{\pi(R)},\ldots,S_{\pi(T)}\}$ is~$\le c$.
\end{itemize}
To prove $(*)$, assume that  $d(T,R) \le \max\left\{2 \lfloor \Upsilon / b_1 \rfloor,1\right\}$. Using the hypothesis of Case~2 we may apply Proposition~\ref{PropFSUContraction}, concluding that for some $l \in [0,\pi(R)]$ the number of free splitting units between $S_l$ and $S_{k_T}$ is $\le b_1$. Using \emph{Stability of free splitting units} it follows that the number of free splitting units between $S_{\pi(R)}$ and $S_{k_T}$ is $\le b_1$. Applying Lemma~\ref{LemmaUnitsLipschitz} we have $\diam\{S_{\pi(R)},\ldots,S_{\pi(T)}\} \le c$.

Since $(*)$ applies whenever $d(T,R) \le 1$, the \emph{Coarse Lipschitz} axiom follows immediately.

To prove \emph{Desymmetrized strong contraction} we shall produce constants $a,b > 0$ so that if $a \le d(T,\{S_0,\ldots,S_K\})$ and $d(T,R) \le b \,\cdot\, d(T,\{S_0,\ldots, S_K\})$ then $d(T,R) \le 2 \lfloor \Upsilon / b_1 \rfloor$, for then $(*)$ applies and so $\diam\{S_{\pi(R)},\ldots,S_{\pi(T)}\} \le c$.

Consider first the case that $\Upsilon < 2b_1$. By Lemma~\ref{LemmaUnitsLipschitz} we have $d(T_{k_T},T) < 2 b_1 L   + C$ and so $d(T,S_0\mapsto\cdots\mapsto S_K) < 2 b_1 L + C +2$. By taking $a= 2 b_1 L + C + 2 = 80 \rank(F) - 52$ we may dispense with this case.

Consider next the case that $\Upsilon \ge 2 b_1$. It follows that $\Upsilon \ge 1$. We have $\Upsilon/b_1 \le 2 (\Upsilon / b_1 - 1)$ from which it follows that 
$$\Upsilon / b_1 \le 2 \lfloor \Upsilon / b_1 \rfloor
$$
The number of free splitting units between $T_{k_T}$ and $T_L=T$ equals $\Upsilon$ and so by Lemma~\ref{LemmaUnitsLipschitz} we have $d(T,T_{k_T}) \le L \Upsilon + C$. It follows that $d(T,S_{k_T}) \le L \Upsilon+C+2$, which implies that $d(T,S_0\mapsto\cdots\mapsto S_K) \le L \Upsilon+C+2$. Let
\begin{align*}
b &= \frac{1}{80 \rank(F) - 60} = \frac{1}{b_1(L + C + 2)} \\
   &\le \frac{1}{b_1(L + \frac{C+2}{\Upsilon})} = \frac{\Upsilon}{b_1(L \Upsilon+C+2)}
\end{align*}
where the inequality follows from $\Upsilon \ge 1$. We then have
$$b(L \Upsilon + C + 2) \le \Upsilon / b_1 
$$
It follows that if $d(T,R) \le b \cdot d(T,S_0\mapsto\cdots\mapsto S_K)$ then $d(T,R) \le \Upsilon  / b_1 \le 2 \lfloor \Upsilon / b_1 \rfloor$ and we are done, subject to proving Proposition~\ref{PropFSUContraction}.
\end{proof}

\paragraph{Quasigeodesic reparameterization of fold paths.} We can also use these arguments to show how fold paths can be reparameterized, using free splitting units, to give a system of uniform quasigeodesics in $\FS'(F)$. Recall that each fold sequence $S_0 \mapsto\cdots\mapsto S_M$ can be interpolated by a continuous edge path in $\FS'(F)$: for each fold $S_{m-1} \mapsto S_m$, the vertices $S_{m-1},S_m$ are connected in $\FS'(F)$ by an edge path of length $2$, $1$, or $0$, by Lemma~\ref{LemmaFoldDistance}. Let $\Upsilon$ be the number of free splitting units from $S_0$ to~$S_M$. Choose any sequence $0 \le m_0 < m_1 < \cdots < m_\Upsilon \le M$ such that for $u=1,\ldots,\Upsilon$ there is $\ge 1$ free splitting unit between $S_{m_{u-1}}$ and~$S_{m_u}$. Notice that by \emph{Stability of Free Splitting Units}, the number of free splitting units between $S_0$ and $S_{m_1}$, and between $S_{m_{\Upsilon-1}}$ and $S_M$ is $\ge 1$, and so we may rechoose the first and last terms of the sequence so that $0=m_0 < m_1 < \cdots < m_\Upsilon=M$. Choose a continuous parameterization of the interpolating edge path of the form $\gamma \from [0,\Upsilon] \to \FS'(F)$ such that $S_{m_u} = \gamma(u)$. We call this a \emph{free splitting parameterization} of the fold sequence $S_0 \mapsto\cdots\mapsto S_M$. 

We use Proposition~\ref{PropFSUContraction}, in particular some details of the preceding proof, in order to prove the following result:

\begin{proposition}\label{PropFoldPathQuasis}
There exist constants $k,c$ depending only on $\rank(F)$ such that any free splitting parameterization $\gamma \from [0,\Upsilon] \to \FS'(F)$ of any fold path $S_0 \mapsto\cdots\mapsto S_M$ is a $k,c$ quasigeodesic in $\FS'(F)$, that is,
$$ \frac{1}{k} \abs{s-t} - c \le d(\gamma(s),\gamma(t)) \le k \abs{s-t}+c \quad\text{for all $s,t \in [0,\Upsilon]$.}
$$
\end{proposition}

\begin{proof} We continue with the constants $L=10$, $C=8$, $b_1=4 \rank(F)-3$ from the previous proof. 

As shown back in the definition of free splitting units, for each integer $u=1,\ldots,\Upsilon$ there is exactly~$1$ free splitting unit between $S_{m_{u-1}}$ and $S_{m_u}$. Applying Lemma~\ref{LemmaUnitsLipschitz} it follows that for each $u=1,\ldots,\Upsilon$ the set $\{S_{m_{u-1}},\ldots,S_{m_u}\}$ has diameter $\le L+C$. Combining this with the fact that the edge path interpolating each fold has length~$\le 2$ it follows that 
$$(**) \qquad\qquad \diam(\gamma[u-1,u]) \le L+C+1 \quad\text{for each $u = 1,\ldots,\Upsilon$}
$$ 
Given $s,t \in [0,\Upsilon]$, if there is no integer in the interval $[s,t]$ then $d(\gamma(s),\gamma(t)) \le L+C+1$. Otherwise we take $u,v \in [s,t]$ to be the smallest integer $\ge s$ and the largest integer $\le t$, respectively, and we have
\begin{align*}
d(\gamma(s),\gamma(t)) &\le d(\gamma(u),\gamma(v)) + d(\gamma(s),\gamma(u)) + d(\gamma(t),\gamma(v)) \\ & \le (L+C+1) \abs{v-u} + 2(L+C+1) \\ &\le k \abs{s-t} + c
\end{align*}
using any $k \ge L+C+1=19$ and any $c \ge 2(L+C+1)=38$ (and we note that this inequality also holds in the previous case where there is no integer in $[s,t]$). This proves the second inequality of the proposition.

To prove the first inequality, we first prove it for integer values $u \le v \in [0,\ldots,\Upsilon]$. Fix a geodesic edge path $\rho$ of length $D=d(\gamma(u),\gamma(v))$ connecting $\gamma(u)$ to $\gamma(v)$ in $\FS'(F)$. Project $\rho$ to the fold path $S_0\mapsto\cdots\mapsto S_M$. By the statement $(*)$ above, within this fold path there are $\le b_1$ free splitting units between the projections of any two consecutive vertices of $\rho$. By applying Lemma~\ref{LemmaFSUTriangleInequality}, the coarse triangle inequality for free splitting units, it follows that there are $\le D (b_1+1)$ free splitting units between $S_{\pi(\gamma(u))}$ and $S_{\pi(\gamma(v))}$, the projections of $\gamma(u)$ and~$\gamma(v)$, respectively. By Proposition~\ref{PropCoarseRetract}, where the \emph{Coarse retract} axiom was proved, the number of free splitting units between $S_{m_u}=\gamma(u)$ and $S_{\gamma(u)}$, and between $S_{m_v} = \gamma(v)$ and $S_{\gamma(v)}$, are both~$<1$. By applying Lemma~\ref{LemmaFSUTriangleInequality} again, the number of free splitting units between $S_{m_u}$ and $S_{m_v}$ is $\le D(b_1+1)+2$, that is, $\abs{u-v} \le D(b_1+1)+2$.

For arbitrary $s<t \in [0,\ldots,\Upsilon]$, letting $u \in [0,\Upsilon]$ be the largest integer $\le s$ and $v \in [0,\Upsilon]$ be the smallest integer $\ge t$, we have $\gamma(s) \in \gamma[u,u+1]$ and $\gamma(t) \in \gamma[v-1,v]$. By $(**)$ we therefore have $d(S_{\gamma(s)},S_{\gamma(u)})$, $d(S_{\gamma(t)},S_{\gamma(v)}) \le L+C+1=19$. It follows that:
\begin{align*}
\abs{s-t} &\le \abs{u-v} \\
              &\le (b_1+1) d(S_{\gamma(v)}, S_{\gamma(v)}) + 2 \\
\frac{1}{b_1+1} \abs{s-t} - \frac{2}{b_1+1} & \le d(S_{\gamma(v)}, S_{\gamma(v)})  \\     
              &\le d(S_{\gamma(v)}, S_{\gamma(v)}) + (19 - d(S_{\gamma(u)},S_{\gamma(s)})) \\
              &\qquad\qquad  + (19 - d(S_{\gamma(v)},S_{\gamma(t)})) \\
\frac{1}{b_1+1} \abs{s-t} - \left( \frac{2}{b_1+1} + 38 \right) &\le d(S_{\gamma(s)},S_{\gamma(t)})
\end{align*}
This proves that the first inequality is true for any $\displaystyle k \ge b_1+1=4 \rank(F)-2$ and any \break $\displaystyle c \ge \frac{2}{b_1+1} + 38 = \frac{1}{2\rank(F)-1} + 38$.

Proposition~\ref{PropFoldPathQuasis} is therefore proved for $k = \max\{19,4\rank(F)-2\}$ and $c=39$.
\end{proof}

\subsection{Pushing down peaks}
\label{SectionPushingDownPeaks}
Recall that every geodesic in $\FS'(F)$ is a zig-zag edge path. On a zig-zag subpath of the form $T^{i-1} \expandsto T^i \collapsesto T^{i+1}$, where $T^i$ is the domain of two incident collapse maps $T^i \mapsto T^{i-1}$ and $T^i \mapsto T^{i+1}$, we say that $T^i$ is a \emph{peak}. If on the other hand $T^{i-1} \collapsesto T^i \expandsto T^{i+1}$ then $T^i$ is a \emph{valley}.

We start with a simplistic technique that can be used to shortcut a zig-zag path, and we work up to a technique, described in Proposition~\ref{PropPushdownInToto}, that will be central to the proof of the Main Theorem. In each case the intuition is to ``push down the peak'', thereby reducing length.

\paragraph{The peak of a W diagram.} A \emph{W diagram}\index{W diagram} or a \emph{W zig-zag} is a length~$4$ zig-zag path with a peak in the middle, sometimes depicted as in Figure~\ref{FigureWDiagram}. 
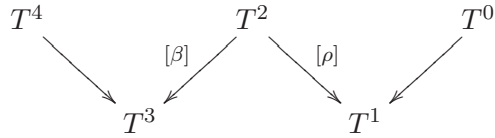
\begin{figure}[h]
$$\xymatrix{
T^4 \ar[dr] && T^2 \ar[dl]_{[\beta]} \ar[dr]^{[\rho]} && T^0 \ar[dl] \\
        & T^3               &&       T^1           
}$$
\caption{A W diagram}
\label{FigureWDiagram}
\end{figure}
We think of $\beta,\rho$ as the ``blue'' and ``red'' subgraphs of $T^2$. In this generality, an edgelet of $T^2$ may be in either, or both, or neither of $\beta,\rho$. The subgraphs $\beta,\rho$ therefore do \emph{not} necessarily form a blue--red decomposition of $T^2$ as in Definition~\ref{DefBRDecompos}, which requires that $\beta,\rho$ have no edgelets in common and and their union is all of $T^2$; furthermore, even if $\beta,\rho$ did form a blue--red decomposition, they need not be a \emph{natural} one, which requires in addition that they both be natural subgraphs of $T^2$. Soon, though, we shall narrow down to a key special case where $\beta,\rho$ is indeed a natural blue--red decomposition.

Pushing down the peak is easy when $\beta\union\rho$ is a proper subgraph of $T^2$, for in that case the given W diagram extends to a commutative diagram of collapse maps as shown in the diagram in Figure~\ref{FigureSimplistic}.
\begin{figure}
$$\xymatrix{
T^4 \ar[dr] && T^2 \ar[dl]_{[\beta]} \ar[dr]^{[\rho]} \ar[dd]|{[\beta \union \rho]} && T^0 \ar[dl] \\
        & T^3 \ar[dr]_{[\rho\setminus(\beta\intersect\rho)]}               &&       T^1 \ar[dl]^{[\beta\setminus(\beta\intersect\rho)]}    \\
        && T^h       
}$$
\caption{A simplistic pushdown works if $\beta\union\rho \subset T^2$ is a proper subgraph.}
\label{FigureSimplistic}
\end{figure}
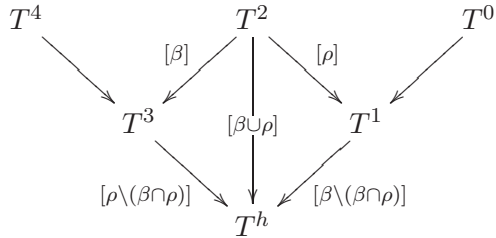
In that diagram, collapse of $\beta\union\rho \subset T^2$ produces $T^h$. The collapse map $T^2 \xrightarrow{[\rho]} T^1$ takes the edgelets of the subgraph $\beta\setminus(\beta\intersect\rho)\subset T^2$ bijectively to the edgelets of a subgraph of $T^1$ which by convention is also denoted $\beta\setminus(\beta\intersect\rho)$; collapse of this subgraph also produces~$T^h$. Similarly, collapse of $\rho\setminus(\beta\intersect\rho) \subset T^3$ produces $T^h$. Compositions of collapse maps being collapse maps, we obtain a length~2 zig-zag path $T^0 \rightarrow T^h \leftarrow T^4$ that cuts short the original length~4 zig-zag path --- we have successfully ``pushed down the peak''. 

The same argument works on a length~$3$ zig-zag path --- which can be visualized by cutting off one of the terminal edges of a W zig-zag --- with the result that if the union of the two collapse graphs at the peak of the zig-zag forms a proper subgraph then there is a length~2 path with the same endpoints. We summarize as follows:

\begin{lemma}\label{LemmaRedBlueUnion}
Given a W zig-zag as notated in Figure~\ref{FigureWDiagram} or a length~3 zig-zag obtained from Figure~\ref{FigureWDiagram} by cutting off one of the terminal edges, if the path is geodesic then $T^2 = \beta \union \rho$. 
\qed\end{lemma}

\paragraph{Normalizing a W diagram.} We shall also need to push down the peak of certain W diagrams in the situation where $T^2 = \beta \union \rho$. In this situation it is convenient to first alter the W diagram to ensure that $\beta \intersect \rho$ contains no edgelet of $T^2$, equivalently $\beta,\rho$ is a blue--red decomposition of~$T^2$ as in Definition~\ref{DefBRDecompos}. If $\beta\intersect\rho$ does contain an edgelet of $T^2$ then, since $\beta,\rho$ are proper subgraphs, the given W diagram is contained in a commutative diagram of collapse maps as shown in the diagram in Figure~\ref{FigureNormalization}, called a \emph{normalization diagram}.\index{normalization diagram}
\begin{figure}[h]
$$\xymatrix{
T^4 \ar[dddrr] &&&& T^2 \ar[dddll]_{[\beta]}  \ar[dd]|{[\beta\intersect\rho]}  \ar[dddrr]^{[\rho]} &&&& T^0 \ar[dddll] \\
        &&&&   &&&&  \\
        &&               && T'{}^2 \ar[dll]|{[\beta\setminus(\beta\intersect\rho)]} \ar[drr]|{[\rho\setminus(\beta\intersect\rho]} &&       \\
        && T^{3} &&&& T^{1}
}$$
\caption{A normalization diagram. The W zig-zag on the top of the diagram has the property that $T^2=\beta\union\rho$. The W zig-zag on the bottom of the diagram is normalized.}
\label{FigureNormalization}
\end{figure}
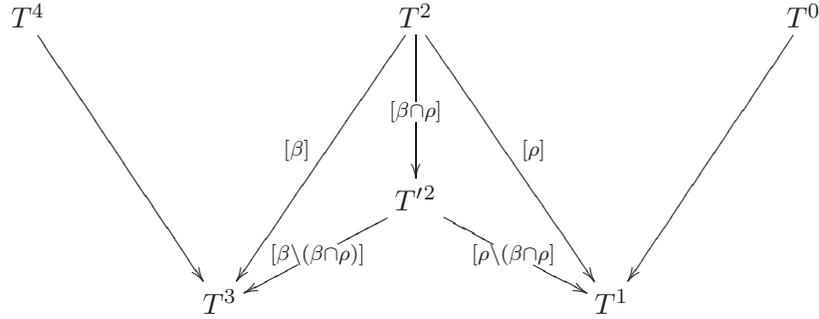
In this diagram, subgraphs of $T'{}^2$ are labelled by the same convention as described above. 
Since $T^2 = \beta \union \rho$ it follows that the two subgraphs $\beta \setminus (\beta\intersect \rho)$ and $\rho \setminus (\beta \intersect \rho)$ of $T'{}^2$ partition the edgelets of $T'{}^2$.

Motivated by this observation, we say that a zig-zag path in $\FS'(F)$ is \emph{normalized} if at every free splitting $F \act T$ along the path that forms a peak, the two subgraphs of $T$ whose collapses define the vertices of the path incident to $T$ form a blue--red decomposition of $T$. 
The argument we have just given shows that every geodesic zig-zag path in $\FS'(F)$ may be replaced by a normalized zig-zag path of the same length and with the same set of valleys.

\paragraph{Pushdown subgraphs and baseball diagrams.} We now turn to a more sophisticated technique for pushing down the peak of a W diagram. Consider a W diagram as notated in Figure~\ref{FigureWDiagram} and suppose that $\beta \union \rho = T^2$ is a blue--red decomposition. Consider also a subgraph $\kappa \subset T^2$ that satisfies the following:
\begin{description}
\item[$\kappa$ is a pushdown subgraph:] $\kappa$ is a proper, equivariant subgraph, and each natural edge of $T^2$ not contained in $\kappa$ contains at least one red and one blue edgelet of $T^2$ that are not contained in $\kappa$.
\end{description}
No requirement is imposed that a pushdown subgraph be a natural subgraph; the proof of Proposition~\ref{PropPushdownInToto} produces pushdown subgraphs which are not natural. Note that a pushdown subgraph can \emph{only} exist if $\beta \union \rho = T^2$ is not a natural blue--red decomposition.

Given a normalized W diagram and a pushdown subgraph $\kappa \subset T^2$, we may extend the W diagram to a larger commutative diagram of collapse maps called a \emph{baseball diagram},\index{baseball diagram} as shown in Figure~\ref{FigureBaseball}. Certain superscripts in this diagram represent various positions on a baseball diamond: $T^1$, $T^2$, $T^3$ represent $1^{\text{st}}$, $2^{\text{nd}}$ and $3^{\text{rd}}$ bases, $T^p$ the pitcher's mound, $T^{h1}$ and $T^{h3}$ the points halfway from home plate to $1^{\text{st}}$ and $3^{\text{rd}}$ bases. 
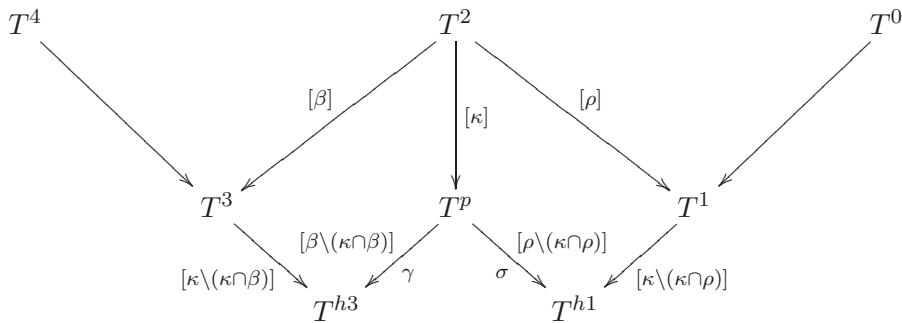
\begin{figure}
$$\xymatrix{
T^4 \ar[ddrr] &&&& T^2 \ar[ddll]_{[\beta]}  \ar[dd]^{[\kappa]}  \ar[ddrr]^{[\rho]} &&&& T^0 \ar[ddll] \\
        &&&&   &&&&  \\
        && T^3 \ar[dr] _{[\kappa \setminus (\kappa\intersect\beta)]}              && T^p \ar[dl]_{[\beta\setminus(\kappa\intersect\beta)]}^{\gamma} \ar[dr]^{[\rho\setminus(\kappa\intersect\rho)]}_{\sigma} &&       T^1  \ar[dl]^{[\kappa\setminus(\kappa\intersect\rho)]}  \\
        &&& T^{h3} && T^{h1}
}$$
\caption{A baseball diagram}
\label{FigureBaseball}
\end{figure}
Collapsed subgraphs of the trees $T^1,T^p,T^3$ in this diagram are named following a convention similar to that used earlier. Because $\kappa$ is a pushdown subgraph, neither of the two subgraphs $\rho \setminus (\kappa \intersect \rho)$, $\beta \setminus (\kappa \intersect \beta) \subset T^p$ contains a natural edge of $T^p$. It follows that neither of the two collapse maps $\sigma \from T^p \to T^{h1}$, $\gamma \from T^p \to T^{h3}$ collapses an entire natural edge of $T^p$. Each of the maps $\sigma,\gamma$ therefore induces by restriction a bijection of natural vertex sets, takes each natural edge onto a natural edge inducing a bijection of natural edge sets, and is homotopic to a conjugacy relative to natural vertex sets. By restricting to natural vertex sets we therefore obtain a well-defined bijection $\gamma \composed \sigma^\inv$ from the natural vertex set of $T^{h1}$ to the natural vertex set of $T^{h3}$ which extends to a conjugacy $\xi \from T^{h1} \mapsto T^{h3}$. Since collapses are transitive, we have again successfully ``pushed down the peak'', without even bothering to involve home plate as in the earlier scenario:
$$
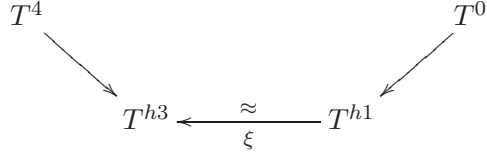
\xymatrix{
T^4 \ar[dr] &&&& T^0 \ar[dl] \\
    & T^{h3} && T^{h1} \ar[ll]^{\xi}_{\approx}
}$$
We record this as:

\begin{lemma}[Pushing down peaks]\label{LemmaPushingDown}
Given a normalized W diagram notated as in Figure~\ref{FigureWDiagram}, and given a pushdown subgraph $\kappa \subset T^2$, there exists a baseball diagram notated as in Figure~\ref{FigureBaseball}, in which each map $\gamma \from T^p \to T^{h3}$ and $\sigma \from T^p \to T^{h1}$ induces by restriction a bijection of natural vertex sets and a bijection of natural edge sets, and is homotopic rel natural vertices to a conjugacy. By composition we therefore obtain a bijection $\gamma\sigma^\inv$ from the natural vertex set of $T^{h1}$ to the natural vertex set of $T^{h3}$ that extends to a conjugacy $\xi \from T^{h1} \to T^{h3}$.
\qed\end{lemma}

We emphasize that the conjugacy in the conclusion of this lemma need not be a \emph{map}, i.e.\ it need not be simplicial. Nonsimplicial conjugacies resulting from Lemma~\ref{LemmaPushingDown} will proliferate into the proof of Proposition~\ref{PropFSUContraction} given in Section~\ref{SectionProofFSUContraction}, and that proof will have a certain step dedicated to patching up this problem.

\paragraph{Pushing down corrugation peaks.} One key strategy occuring in the proof of Proposition~\ref{PropFSUContraction} is to set up applications of Lemma~\ref{LemmaPushingDown} by finding pushdown subgraphs in peaks of normalized W diagrams. Of course this is impossible if the W diagram is geodesic. Nevertheless in Proposition~\ref{PropPushdownInToto} we will show that when combing a fold path across an arbitrary W diagram, even one which is geodesic, one can always locate enough pushdown subgraphs to carry out the pushdown process in a useful fashion, as long as the fold path is sufficiently long when measured in free splitting units.

Consider a fold sequence $T^0_0 \mapsto \cdots \mapsto T^0_J$. Consider also a zig-zag path $T^0_J \xrightarrow{} T^1_J \xleftarrow{[\rho_J]} T^2_J \xrightarrow{[\beta_J]} T^3_J \xleftarrow{} T^4_J$ in $\FS'(F)$, which may be regarded as a W diagram. We do not assume that this W diagram is a geodesic, nor even that it is normalized, but we do assume that $T^2_J = \beta_J \union \rho_J$. Consider finally a stack of four combing rectangles combined into one commutative diagram as shown in Figure~\ref{FigurePullBack1}, where the given fold sequence occurs as the $T^0$ row along the bottom of the diagram, and the W zig-zag occurs as the $T_J$ column along the right side (in such diagrams, in general we refer to rows by dropping subscripts, and to columns by dropping superscripts).
\begin{figure}[h]
$$\xymatrix{
T^4_0 \ar[r] \ar[d] 
               & \cdots \ar[r] & T^4_{I} \ar[r] \ar[d]                  & \cdots \ar[r]  
               & T^4_J \ar[d] \\
T^3_0 \ar[r]                                                                     
               & \cdots \ar[r] & T^3_{I}  \ar[r]                             & \cdots \ar[r]
               & T^3_J           \\
T^2_0  \ar[r] \ar[u]^{[\beta_0]} \ar[d]_{[\rho_0]}  
               & \cdots \ar[r] & T^2_{I}  \ar[r] \ar[u]^{[\beta_{I}]} \ar[d]_{[\rho_{I}]}   & \cdots \ar[r]
               & T^2_J  \ar[u]^{[\beta_J]} \ar[d]_{[\rho_J]} \\
T^1_0  \ar[r]                                  
               & \cdots \ar[r] & T^1_I  \ar[r]  & \cdots \ar[r]
               & T^1_J \\
T^0_0 \ar[r] \ar[u]
               & \cdots \ar[r] & T^0_I \ar[r] \ar[u]  & \cdots \ar[r]
               & T^0_J \ar[u]
}$$
\caption{A diagram of four combing rectangles over $F$. The $T^0$ row along the bottom is assumed to be a fold sequence. In the $T_J$ column we assume that $T^2_J = \rho_J \union \beta_J$.}
\label{FigurePullBack1}
\end{figure}
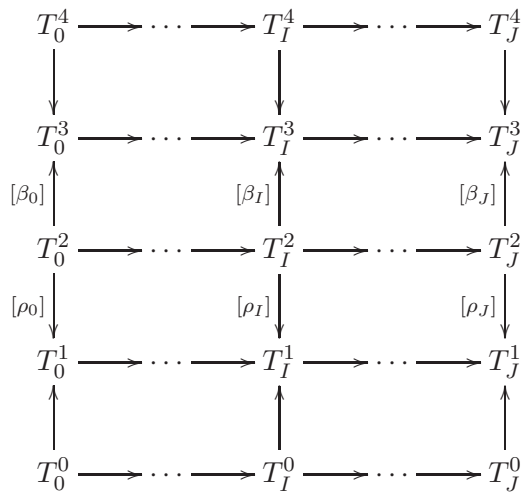
Such a diagram can be constructed, for example, by starting with the bottom row and right side, and applying Propositions~\ref{PropCBC}, then~\ref{PropCBE}, then~\ref{PropCBC}, then~\ref{PropCBE}, in that order, to comb the given fold sequence along each of the four edges of the given zig-zag path. We will also encounter such diagrams constructed by other combing processes involving concatenation and deconcatenation of combing rectangles.

We can visualize Figure~\ref{FigurePullBack1} as a piece of corrugated metal. The $T^2$ row forms a peak of the corrugation which we wish to push down all at once, by parallel applications of Lemma~\ref{LemmaPushingDown}. Of course this is impossible in general, for instance when the $T_J$ column is a geodesic path in $\FS'(F)$. 

We now describe a process which allows us to push down the corrugation peak along the $T^2$ row, at the expense of throwing away the portion of the diagram to the right of the $T_I$ column that is depicted in Figure~\ref{FigurePullBack1}. The next lemma says that this is always possible \emph{as long as} the bottom row has sufficiently many free splitting units between $T^0_I$ and $T^0_J$. As a consequence, the $T_j$ columns for $0 \le j \le I$ are \emph{not} geodesic paths in $\FS'(F)$ because $d(T^0_j,T^4_j) \le 2$, even when the $T_J$ on the far right is geodesic. We thus obtain a key indicator of ``hyperbolic'' behavior: local curve shortening. 

The following proposition introduces the constant $4 \rank(F)-3$ which is needed for the proof of Proposition~\ref{PropFSUContraction}.

\break

\begin{proposition}\label{PropPushdownInToto}
For any commutative diagram as in Figure~\ref{FigurePullBack1}, if the number of free splitting units between $T^0_I$ and $T^0_J$ is $\ge 4 \rank(F) - 3$ then there is a commutative diagram
$$\xymatrix{
T^4_0 \ar[r] \ar[d] 
               & \cdots \ar[r] 
               &T^4_I \ar[d] \\
T^{h3}_0 \ar[r]  \ar@{=}[d]^{\xi_0}                                                                 
               & \cdots \ar[r] 
               &T^{h3}_I \ar@{=}[d]^{\xi_I}    \\
T^{h1}_0  \ar[r]                               
               & \cdots \ar[r] 
               &T^{h1}_I \\
T^0_0 \ar[r] \ar[u]
               & \cdots \ar[r]
               &T^0_I \ar[u]
}$$
such that the following hold: the top and bottom horizontal rows are the same foldable sequences as the top and bottom rows of Figure~\ref{FigurePullBack1} between the $T_0$ and $T_I$ columns; the $T^{h1}$ and $T^{h3}$ rows are foldable sequences; for each $j=0,\ldots,J$ the function $\xi_j$ is a (nonsimplicial) conjugacy between $T^{h1}_j$ and~$T^{h3}_j$; and the top and bottom horizontal rectangles are combing rectangles obtained from the top and bottom combing rectangles of Figure~\ref{FigurePullBack1} between the $T_0$ and $T_I$ columns by application of \emph{Composition of combing rectangles} \ref{LemmaCombingComp}. 
\end{proposition}

\begin{proof} There are three steps to the proof: normalization; pullback; and pushdown.

\smallskip

\textbf{Step 1: Normalization.} Knowing that $T^2_J = \beta_J \union \rho_J$, and knowing for each $j=0,\ldots,J$ that $\beta_j$, $\rho_j$ are the union of the edgelets mapped to $\beta_J$, $\rho_J$, respectively, under the foldable map $T^2_j \mapsto T^2_J$, it follows that $T^2_j = \beta_j \union \rho_j$. If the $T_J$ column is already normalized, that is if $\beta_J \union \rho_J= T_J$ is a blue--red decomposition, then the same is true of $\beta_j \union \rho_j = T_j$, and so each $T_j$ column is normalized and we pass directly to Step~2.

Otherwise, let us assume that $\beta_J$, $\rho_J$ have some edgelets in common. The union of these edgelets is a subgraph with nondegenerate components which by abuse of notation we denote $\beta_J \intersect \rho_J \subset T^2_J$. It follows that for each $j=0,\ldots,J$ the graphs $\beta_j,\rho_j$ have some edgelets in common, these being the edgelets that are mapped to $\beta_J \intersect \rho_J$ by the foldable map $T^2_j \mapsto T^2_J$; their union forms a subgraph $\beta_j \intersect \rho_j \subset T^2_j$. We may now carry out the normalization process depicted in Figure~\ref{FigureNormalization}, in parallel as $j$ varies from $0$ to $J$. The resulting normalization diagrams, commutative diagrams of collapse maps, are shown in Figure~\ref{FigureNormalizationWZigZag}.
\begin{figure}[h]
$$\xymatrix{
T^4_j \ar[dddrr] &&&& T^2_j \ar[dddll]_{[\beta_j]}  \ar[dd]|{[\beta_j\intersect\rho_j]}  \ar[dddrr]^{[\rho_j]} &&&& T^0_j \ar[dddll] \\
        &&&&   &&&&  \\
        &&               && T'{}^2_j \ar[dll]|{\quad[\beta_j\setminus(\beta_j\intersect\rho_j)]} \ar[drr]|{[\rho_j\setminus(\beta_j\intersect\rho_j]\quad} &&       \\
        && T^{3}_j &&&& T^{1}_j
}$$
\caption{Parallel normalization diagrams associated to the W zig-zags from $T^0_j$ to $T^4_j$ in Figure~\ref{FigurePullBack1}.}
\label{FigureNormalizationWZigZag}
\end{figure}
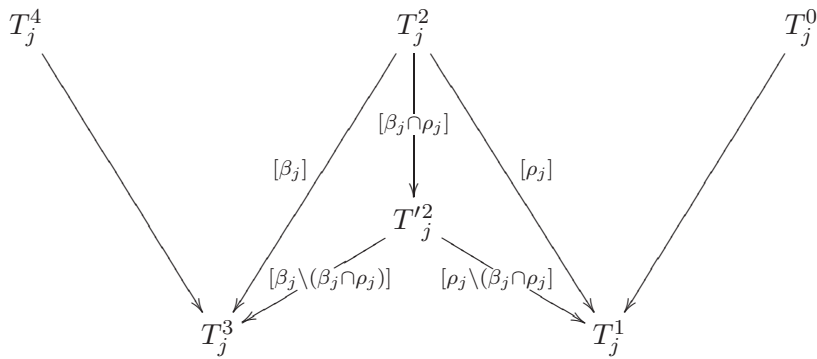

\break

We claim that for each of the seven arrows in Figure~\ref{FigureNormalizationWZigZag}, as $j$ varies from $0$ to $J$ we obtain a combing rectangle. One can visualize this statement as a description of a 3-dimensional commutative diagram where the normalization diagrams are lined up in parallel vertical planes, connected up by six foldable sequences (one for each of the six positions in the normalization diagram) and seven combing rectangles (one for each of the seven arrows). The claim is true by hypothesis for the four arrows on the top of the diagram. To obtain the combing rectangle with vertical arrows from $T^2_j$ to $T'{}^2_j$, since $\beta_j\intersect\rho_j$ is the inverse image of $\beta_J\intersect\rho_J$ under the foldable map $T^2_j \mapsto T^2_J$, by Proposition~\ref{PropCBC} the collapse maps $T^2_j \xrightarrow{[\beta_j\intersect\rho_j]} T'{}^2_j$ fit together in a combing rectangle as follows:
$$\xymatrix{
T^2_0 \ar[r] \ar[d]^{[\beta_0\intersect\rho_0]} 
				& \cdots \ar[r] & T{}^2_{I}  \ar[r]\ar[d]^{[\beta_I\intersect\rho_I]} & \cdots \ar[r]
				& T^2_J \ar[d]^{[\beta_J\intersect\rho_J]} \\
T'{}^2_0 \ar[r]           
				& \cdots \ar[r] & T'{}^2_{I}  \ar[r] & \cdots \ar[r]
				& T'{}^2_J \\
}$$
The two combing rectangles with vertical arrows from $T'{}^2_j$ to $T^1_j$ and from $T'{}^2_j$ to $T^3_j$, respectively, are obtained by two applications of Lemma~\ref{LemmaCombingDecomp} \emph{Decomposition of combing rectangles}, the first application using the $T^2_j$ to $T^1_j$ and the $T^2_j$ to $T'{}^2_j$ combing rectangles, and the second using the $T^2_j$ to $T^3_j$ and the $T^2_j$ to $T'{}^2_j$ combing rectangles. This proves the claim.

The outcome of the claim is a commutative diagram of the form shown in Figure~\ref{FigureNormalOutcome}, in which the top and bottom rectangles are the same combing rectangles as in Figure~\ref{FigurePullBack1}. By construction (see Figure~\ref{FigureNormalization}), the zig zag path on the right side of Figure~\ref{FigureNormalOutcome} is normalized, completing Step~1.

\begin{figure}
$$\xymatrix{
T^4_0 \ar[r] \ar[d] 
               & \cdots \ar[r] & T^4_{I} \ar[r] \ar[d]                  & \cdots \ar[r]  
               & T^4_J \ar[d] \\
T^3_0 \ar[r]                                                                     
               & \cdots \ar[r] & T^3_{I}  \ar[r]                             & \cdots \ar[r]
               & T^3_J           \\
T'{}^2_0  \ar[r] \ar[u] \ar[d] 
               & \cdots \ar[r] & T'{}^2_{I}  \ar[r] \ar[u] \ar[d]   & \cdots \ar[r]
               & T'{}^2_J  \ar[u] \ar[d] \\
T^1_0  \ar[r]                                  
               & \cdots \ar[r] & T^1_I  \ar[r]  & \cdots \ar[r]
               & T^1_J \\
T^0_0 \ar[r] \ar[u]
               & \cdots \ar[r] & T^0_I \ar[r] \ar[u]  & \cdots \ar[r]
               & T^0_J \ar[u]
}$$
\caption{The outcome of normalizing Figure~\ref{FigurePullBack1}, using the parallel normalization diagrams of Figure~\ref{FigureNormalizationWZigZag}.}
\label{FigureNormalOutcome}
\end{figure}
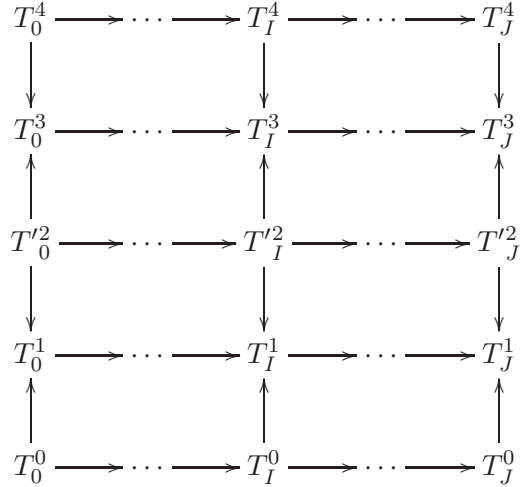

\textbf{Step 2: Pullback.} This is the central argument where the concepts of free splitting units are used to their maximal effect. 

Having carried out Step 1, we may now go back to Figure~\ref{FigurePullBack1} and assume that each $T_j$ column is a normalized W zig-zag. In other words, for each $j$ we have a blue--red decomposition $\beta^2_j \union \rho^2_j = T^2_j$. 

Let $\Upsilon$ be the number of free splitting units along the bottom row of the diagram between $T^0_I$ and $T^0_J$, and choose a sequence $I \le i(0) < \cdots < i(\Upsilon) \le J$ so that for each $u=1,\ldots,\Upsilon$ there is $\ge 1$ free splitting unit between $T^0_{i(u-1)}$ and $T^0_{i(u)}$. By hypothesis we have $\Upsilon \ge 4 \rank(F)-3$.

We prove that the blue--red decomposition $\beta_I \union \rho_I = T^2_I$ is not natural. Arguing by contradiction, suppose that $\beta_I \union \rho_I = T^2_I$ is natural. By Definition~\ref{DefBRDecompos}, it follows that $\beta_i \union \rho_i = T^2_i$ is natural for $I \le i \le J$. By Lemma~\ref{LemmaBRNatural}, the interval $I \le i \le J$ breaks into no more than $4 \rank(F)-3$ subintervals on each of which the complexity of $\beta_i$ is constant. By Definition~\ref{DefLessThanOneFSU}, on each of these subintervals there is $<1$ free splitting unit, and so each of these subintervals contains at most one entry from the sequence $i(0) < \cdots < i(\Upsilon)$. It follows that $\Upsilon \le 4 \rank(F) - 4$, contradicting the hypothesis.

\subparagraph{Remark.} The previous version of this paper contained an invalid argument, starting from the statement that $\beta_i, \rho_i \union T^2_i$ is natural for $I \le i \le J$. The erroneous statement, which incorrectly exploited $\beta_i,\rho_i$, said that if one expands $T^2_i$ by blowing up each vertex $v \in \beta_i \intersect \rho_i$, pulling the blue and red edges at $v$ apart to form two vertices connected by a gray edge, then the resulting tree with $F$-action is a free splitting. The error is that the inserted gray edges might have nontrivial stabilizers. Correcting this error led to a revamping of the theory of free splitting units presented in Section~\ref{SectionFSU}. In particular, the concept of an ``invariant, natural, blue--red decomposition'' in Definition~\ref{DefBRDecompos}, and the diameter bounds of Lemma~\ref{LemmaBRNatural}, are new to this version of the paper and were concocted to correctly exploit the subgraphs $\beta_i,\rho_i \subset T^2_i$.

\smallskip

\textbf{Step 3: Pushdown.} Having carried out Steps~1 and~2, we assume now that we have a commutative diagram as shown in Figure~\ref{FigurePullBack2}, in which each column is normalized and the blue--red decomposition $\beta_I \union \rho_I = T^2_I$ is not natural. It follows that $T^2_I$ has a natural edge $e$ which contains both red and blue edgelets. Using this, we shall produce the commutative diagram needed for the conclusion of Proposition~\ref{PropPushdownInToto}. The argument will be a somewhat more intricate version of the parallel normalization process used in Step~1, using parallel baseball diagrams instead.
\begin{figure}[h]
$$\xymatrix{
T^4_0 \ar[r] \ar[d] 
               & \cdots \ar[r]               
               & T^4_I \ar[d] \\
T^3_0 \ar[r]                                                                     
               & \cdots \ar[r]               
               & T^3_I           \\
T^2_0  \ar[r] \ar[u]^{[\beta_0]} \ar[d]_{[\rho_0]}  
               & \cdots \ar[r]
               & T^2_I  \ar[u]^{[\beta_I]} \ar[d]_{[\rho_I]} \\
T^1_0  \ar[r]                                  
               & \cdots \ar[r] 
               & T^1_I \\
T^0_0 \ar[r] \ar[u]
               & \cdots \ar[r]
               & T^0_I \ar[u]
}$$
\caption{Each of the four horizontal rectangles is a combing rectangle. We assume that every column is a normalized W zig-zag and that the tree $T^2_I$ has an edge $e$ containing both red and blue edgelets.}
\label{FigurePullBack2}
\end{figure}

Define a proper $F$-equivariant natural subgraph $\kappa_I = T^2_I$ to be the complement of the orbit of $e$, and so every natural edge of $T^2_I$ not in $\kappa_I$ contains both a red and a blue edgelet. By decreasing induction on $j \in \{0,\ldots,I-1\}$ define an $F$-equivariant subgraph $\kappa_j \subset T^2_j$ to be the inverse image of $\kappa_{j+1}$ under the foldable map $T^2_j \mapsto T^2_{j+1}$ (ignoring degenerate components as usual); equivalently $\kappa_j$ is the inverse image of $\kappa_I$ under $T^2_j \mapsto T^2_I$. It follows that the subgraphs $\kappa_j \subset T^2_j$ are proper for all $j=0,\ldots,I$.

We claim that for $j=0,\ldots,I$ the graph $\kappa_j$ is a pushdown subgraph of $T^2_j$. To prove this, given a natural edge $\eta_j \subset T^2_j$ such that $\eta_j \not\subset \kappa_j$, we must find a red and a blue edgelet in $\eta_j$ neither of which is in $\kappa_j$. Foldable maps take natural vertices to natural vertices and natural edges to nondegenerate natural edge paths, so the image of $\eta_j$ under the foldable map $T^2_j \mapsto T^2_I$ is a nondegenerate natural edge path denoted $\eta_I \subset T^2_I$. Since $\eta_j \not\subset \kappa_j$, it follows that $\eta_I \not\subset \kappa_I$, and so $\eta_I$ contains a natural edge not in $\kappa_I$ which therefore has both a red and a blue edgelet. Since natural edges not in $\kappa_I$ have interior disjoint from $\kappa_I$ it follows that $\eta_I$ contains a red and a blue edgelet neither of which is in $\kappa_I$. By pulling back under the foldable map $T^2_j \mapsto T^2_I$ we obtain a red and a blue edgelet in $\eta_j$ neither of which is in $\kappa_j$.

We now apply Lemma~\ref{LemmaPushingDown} in parallel to each column $j$ of Figure~\ref{FigurePullBack2} for $j=0,\ldots,I$. The resulting baseball diagrams, commutative diagrams of collapse maps, are shown in Figure~\ref{FigureParameterizedBaseball} (compare Figure~\ref{FigureBaseball}). Lemma~\ref{LemmaPushingDown} also produces conjugacies $T^p_j \mapsto T^{h3}_j$ and $T^p_j \mapsto T^{h1}_j$ and hence conjagacies $T^{h1}_j \to T^{h3}_j$. What we are still missing, however, are the conclusions of Proposition~\ref{PropPushdownInToto} concerned with combing rectangles and commutativity. 
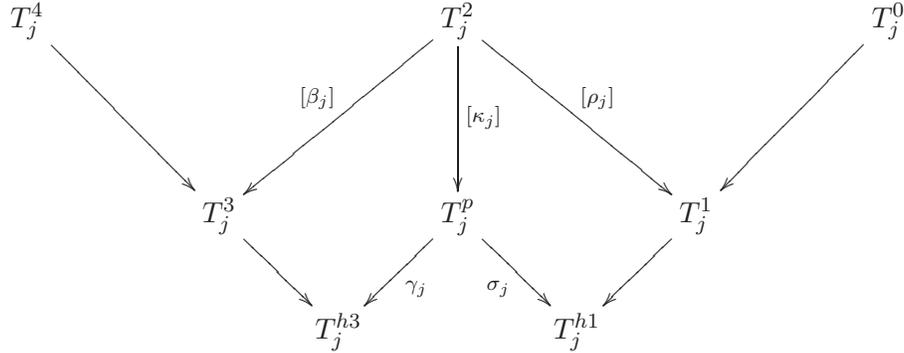
\begin{figure}[h]
$$\xymatrix{
T^4_j \ar[ddrr] &&&& T^2_j \ar[ddll]_{[\beta_j]}  \ar[dd]^{[\kappa_j]}  \ar[ddrr]^{[\rho_j]} &&&& T^0_j \ar[ddll] \\
        &&&&   &&&&  \\
        && T^3_j \ar[dr]           && T^p_j \ar[dl]^{\gamma_j} \ar[dr]_{\sigma_j} &&       T^1_j  \ar[dl]  \\
        &&& T^{h3}_j && T^{h1}_j
}$$
\caption{The baseball diagram associated to the W-diagram from $T^0_j$ to $T^4_j$.}
\label{FigureParameterizedBaseball}
\end{figure}

We claim that for each of the nine arrows in Figure~\ref{FigureParameterizedBaseball}, as $j$ varies from $0$ to $I$ we obtain a combing rectangle. As in Step~1, one visualizes this as a 3-dimensional commutative diagram by lining up the baseball diagrams in parallel vertical planes, connected up by eight foldable sequences (one for each of the eight positions in the baseball diagram) and nine combing rectangles (one for each of the nine arrows). The claim is true by hypothesis for the four arrows on the top of the diagram.

For the arrow from 2nd base to the pitcher's mound, since $\kappa_j$ is the inverse image of $\kappa_J$ under the foldable map $T^2_j \mapsto T^2_I$, by Proposition~\ref{PropCBC} the collapse maps $T^2_j \xrightarrow{[\kappa_j]} T^p_j$ fit together in a combing rectangle
$$\xymatrix{
T^2_0 \ar[r] \ar[d]_{[\kappa_0]} & \cdots\ar@{}[d]|{\text{I}} \ar[r] & T^2_I \ar[d]^{[\kappa_I]} \\
T^p_0 \ar[r]           & \cdots \ar[r] & T^p_I \\
}$$
Notice that for each $j=0,\ldots,J$, the subgraph $\kappa_j \union \rho_j$ is proper, because any natural edge not in $\kappa_j$ contains a blue edgelet not in $\kappa_j$, which is also not in $\kappa_j \union \rho_j$. Similarly the subgraph $\kappa_j \union \beta_j$ is proper. By Proposition~\ref{PropCBC}, since $\kappa_j \union \rho_j$ is the inverse image of $\kappa_{j+1} \union \rho_{j+1}$, and since $\kappa_j \union \beta_j$ is the inverse image of $\kappa_{j+1} \union \beta_{j+1}$, we obtain combing rectangles
$$\xymatrix{
T^2_0 \ar[r] \ar[d]_{[\kappa_0 \union \beta_0]} & \cdots\ar@{}[d]|{\text{II}}  \ar[r] & T^2_I \ar[d]^{[\kappa_I \union \beta_I ]}
& & &
T^2_0 \ar[r] \ar[d]_{[\kappa_0 \union \rho_0]} & \cdots\ar@{}[d]|{\text{III}}  \ar[r] & T^2_I \ar[d]^{[\kappa_I \union \rho_I ]} \\
T^{h3}_0 \ar[r]                           & \cdots \ar[r] & T^{h3}_0
& & &
T^{h1}_0 \ar[r]                           & \cdots \ar[r] & T^{h1}_0
}$$
Rectangles II and III do not correspond to any of the nine arrows in the baseball diagram, but to invisible arrows going from 2nd base to the point halfway between 1st base and home plate and from 2nd base to the point halfway between 3rd base and home plate. 

For the arrows going from the pitcher's mound to the points halfway between 1st and home and halfway between 3rd and home, apply Lemma~\ref{LemmaCombingDecomp} \emph{Decomposition of combing rectangles}, first to combing rectangles II and I and then to combing rectangles III and~I, to obtain combing rectangles 
$$\xymatrix{
T^p_0 \ar[r] \ar[d]_{[\beta_0 \setminus (\kappa_0 \intersect \beta_0)]}^{\gamma_0}
                  & \cdots\ar@{}[d]|{\text{IV}} \ar[r] & T^p_I \ar[d]^{[\beta_I \setminus (\kappa_I \intersect \beta_I)]}_{\gamma_I}
& & & &
T^p_0 \ar[r] \ar[d]_{[\rho_0 \setminus (\kappa_0 \intersect \rho_0)]}^{\sigma_0}           
                  & \cdots\ar@{}[d]|{\text{V}} \ar[r] & T^p_I \ar[d]^{[\rho_I \setminus (\kappa_I \intersect \rho_I)]}_{\sigma_I}
\\
T^{h3}_0 \ar[r]                           & \cdots \ar[r] & T^{h3}_0
& & & &
T^{h1}_0 \ar[r]                           & \cdots \ar[r] & T^{h1}_0
}$$ 
where we follow the same notation convention for subgraphs of $T^p_0$ as used in the original baseball diagram Figure~\ref{FigureBaseball}. 

For the arrows going from 1st base and 3rd base to the points halfway home, applying Lemma~\ref{LemmaCombingDecomp} \emph{Decomposition of combing rectangles} to combing rectangle II and the 2nd base to 3rd base combing rectangle, and then to combing rectangle III and the 2nd base to 1st base combing rectangle, we obtain combing rectangles
$$\xymatrix{
T^3_0 \ar[r] \ar[d]_{[\kappa_0 \setminus (\kappa_0 \intersect \beta_0)]}^{\gamma_0}
                  & \cdots\ar@{}[d]|{\text{VI}} \ar[r] & T^3_I \ar[d]^{[\kappa_I \setminus (\kappa_I \intersect \beta_I)]}
& & & &
T^1_0 \ar[r] \ar[d]_{[\kappa_0 \setminus (\kappa_0 \intersect \rho_0)]}^{\sigma_0}           
                  & \cdots\ar@{}[d]|{\text{VII}} \ar[r] & T^1_I \ar[d]^{[\kappa_I \setminus (\kappa_I \intersect \rho_I)]}
\\
T^{h3}_0 \ar[r]                           & \cdots \ar[r] & T^{h3}_0
& & & &
T^{h1}_0 \ar[r]                           & \cdots \ar[r] & T^{h1}_0
}$$

Applying Lemma~\ref{LemmaCombingComp} \emph{Composition of combing rectangles}, by composing the two combing rectangles corresponding to the arrows along the 1st base foul line in Figure~\ref{FigureParameterizedBaseball} we obtain the combing rectangle from the $T^0$ row to the $T^{h1}$ row needed for the conclusion of Proposition~\ref{PropPushdownInToto}. Similarly, by composing the two combing rectangles corresponding to the arrows along the 3rd base foul line we obtain the combing rectangle from the $T^4$ row to the $T^{h3}$~row.

To complete Step~3 and the proof of the proposition, it remains to construct the commutative diagram of conjugacy maps $\xi_j \from T^{h1}_j \to T^{h3}_j$ in the conclusion of the lemma. For this purpose it suffices to replace combing rectangles IV and V by commutative diagrams of conjugacies of the form 
$$\xymatrix{
T^p_0 \ar[r] \ar[d]^{\bar\gamma_0}
                  & \cdots\ar@{}[d]|{\overline{\text{IV}}} \ar[r] & T^p_I \ar[d]_{\bar\gamma_I}
& & & &
T^p_0 \ar[r] \ar[d]^{\bar\sigma_0}           
                  & \cdots\ar@{}[d]|{\overline{\text{V}}} \ar[r] & T^p_I \ar[d]_{\bar\sigma_I}
\\
T^{h3}_0 \ar[r]                           & \cdots \ar[r] & T^{h3}_I
& & & &
T^{h1}_0 \ar[r]                           & \cdots \ar[r] & T^{h1}_I
}$$ 
for then defining $\xi_j = \bar\gamma_j \composed \bar\sigma_j^\inv \from T^{h1}_j \to T^{h3}_j$ we will be done. While Lemma~\ref{LemmaPushingDown} produces conjugacies $T^{h1}_j \to T^{h3}_j$ for each $j=0,\ldots,J$, if that lemma is used crudely there is no guarantee that these conjugacies will form commutative diagrams as needed. With a little care in how Lemma~\ref{LemmaPushingDown} is applied we can get the needed guarantee. We construct diagram $\overline{\text{IV}}$ in detail, the construction of $\overline{\text{V}}$ being similar. The construction is by induction, starting from the $T_I$ column on the far right and moving leftward. 

First apply Lemma~\ref{LemmaPushingDown} to produce a conjugacy $\bar\gamma_I \from T^p_I \to T^{h3}_I$ so that the restrictions of $\gamma_I$ and $\bar\gamma_I$ to natural vertex sets are the same. Proceeding by decreasing induction on~$j$, suppose that for some $j$ we have produced all the conjugacies from column $T_j$ to $T_I$ in diagram $\overline{\text{IV}}$ making that portion of the diagram commute, and so that the restrictions to natural vertex sets of the conjugacies in diagrams IV and $\overline{\text{IV}}$ are the same from column $T_j$ to column $T_I$. We must choose the conjugacy $\bar \gamma_{j-1} \from T^p_{j-1} \to T^{h3}_{j-1}$ so as to fill in a commutative diagram of $F$-equivariant functions
$$\xymatrix{
T^p_{j-1} \ar@{.>}[d]_{\bar\gamma_{j-1}} \ar[r]^{f_{j}} & T^{p}_{j} \ar[d]^{\bar\gamma_{j}} \\
T^{h3}_{j-1} \ar[r]^{g_j} & T^{h3}_{j}
}$$
where $f_j$, $g_j$ are the foldable maps in Rectangle IV, and where the restrictions of $\bar\gamma_{j-1}$ and $\gamma_{j-1}$ to natural vertex sets are the same. This tells us how to define $\bar\gamma_{j-1}$ on natural vertex sets. Consider a natural edge $\eta \subset T^p_{j-1}$. By Lemma~\ref{LemmaPushingDown} its image $\gamma_{j-1}(\eta) \subset T^{h3}_{j-1}$ is a natural edge whose endpoints are the $\bar\gamma_{j-1}$ images of the endpoints of $\eta$. The foldable map $f_j \from T^p_{j-1} \mapsto T^{p}_j$ is injective on $\eta$, the conjugacy $\bar\gamma_{j}$ is injective on $f_j(\eta)$, and we have the following equation of subsets:
$$g_j(\gamma_{j-1}(\eta)) = \gamma_{j}(f_j(\eta)) = \bar\gamma_{j}(f_j(\eta))
$$
The foldable map $g_j$ is injective on the natural edge $\gamma_{j-1}(\eta)$, and therefore has a homeomorphic inverse $g_j^\inv \from \bar\gamma_{j}(f_j(\eta)) \to \gamma_{j-1}(\eta)$, and so we can define
$$\bar\gamma_{j-1} \restrict \eta = (g_j^\inv \composed \bar\gamma_j \composed f_j) \restrict \eta
$$
This completes Step~3 and the proof of Proposition~\ref{PropPushdownInToto}.
\end{proof}

\subsection{Proof of Proposition \ref{PropFSUContraction}}
\label{SectionProofFSUContraction}

\subparagraph{Prologue.} Consider a fold sequence $S_0 \mapsto\cdots\mapsto S_K$ over $F$, a free splitting $F \act T$, and an augmented projection diagram of maximal depth $k_T=\pi(T)$ as notated in Figure~\ref{FigureMaxProjDiagram} of Section~\ref{SectionStatmentFSUContraction}, whose top row has the fold sequence $T_{k_T} \mapsto\cdots\mapsto T_L=T$ as a terminal segment. Let $\Upsilon$ be the number of free splitting units between $T_{k_T}$ and $T_L=T$. Using the constant $b_1 = 4 \rank(F)-3$ from Proposition~\ref{PropPushdownInToto}, we list every $b_1^{\text{th}}$ term of the back greedy subsequence of this fold sequence~as 
$$k_T \le L_\Omega < L_{\Omega-1} < \cdots < L_1 < L_0=L
$$
where $\Omega = \lfloor \Upsilon / b_1 \rfloor$. Thus $L_\omega$ is the greatest integer $< L_{\omega-1}$ such that there are exactly $b_1$ free splitting units between $T_{L_\omega}$ and $T_{L_{\omega-1}}$, for each $\omega=1,\ldots,\Omega$. Emphasizing only those $T$'s with subscripts from the list $L_\Omega,\ldots,L_0$, and assigning them a superscript~$0$ for later purposes, we may write the augmented projection diagram in the form
$$\xymatrix{
T^0_0 \ar[r] \ar[d] & \cdots \ar[r] & T^0_{k_T} \ar[d] \ar[r] & \cdots \ar[r] & T^0_{L_\Omega} \ar[r] & T^0_{L_{\Omega-1}} \ar[r] & \cdots \ar[r] & T^0_{L_1} \ar[r] & T^0_{L_0} =T 
\\
S'_0 \ar[r]          & \cdots \ar[r] & S'_{k_T} \\
S_0 \ar[r] \ar[u] & \cdots \ar[r] & S_{k_T} \ar[r] \ar[u]  & \cdots \ar[r] & S_K \\
}$$
where the foldable map $T^0_{k_T} \to T^0_{L_\Omega}$ may just be the identity map.

\smallskip

Consider also a vertex $R \in \FS'(F)$ and a geodesic path from $T$ to $R$ in~$\FS'(F)$. We shall assume here that $d(T,R) \ge 3$; the case that $d(T,R) \le 2$ will be considered in the epilogue. If the path from $T$ to $R$ starts with an expansion of $T$, prefix the path with an improper collapse. The result in a zig-zag path of the form 
$$T=T^0_{L_0} \rightarrow T^1_{L_0} \leftarrow T^2_{L_0} \rightarrow T^3_{L_0} \cdots T^D_{L_0} = R
$$
where $D = d(T,R)$ or $d(T,R)+1$ and~$D \ge 3$. The peaks along this zig-zag are the even terms strictly between $0$ and $D$, the first such peak being $T^2_{L_0}$. For each peak along this path, applying Lemma~\ref{LemmaRedBlueUnion} together with the assumption that $d(T,R) \ge 3$ it follows that the peak is the union of its two collapse graphs. The number of peaks along this zig-zag path equals $\lfloor \frac{D-1}{2} \rfloor$ which equals $\frac{D-2}{2}$ if $D$ is even and $\frac{D-1}{2}$ if $D$ is odd. 

By combing the foldable sequence $T^0_0 \mapsto\cdots\mapsto T^0_{L_0}$ across each collapse or expansion of the zig-zag path $T^0_{L_0} \rightarrow T^1_{L_0} \leftarrow \cdots T^D_{L_0} = R$,
alternately applying \emph{Combing by Collapse} \ref{PropCBC} and \emph{Combing by Expansion} \ref{PropCBE}, and by stacking the resulting combing rectangles atop the augmented projection diagram, we obtain The Big Diagram, Step 0, shown in Figure~\ref{FigureBigDiagram0}.
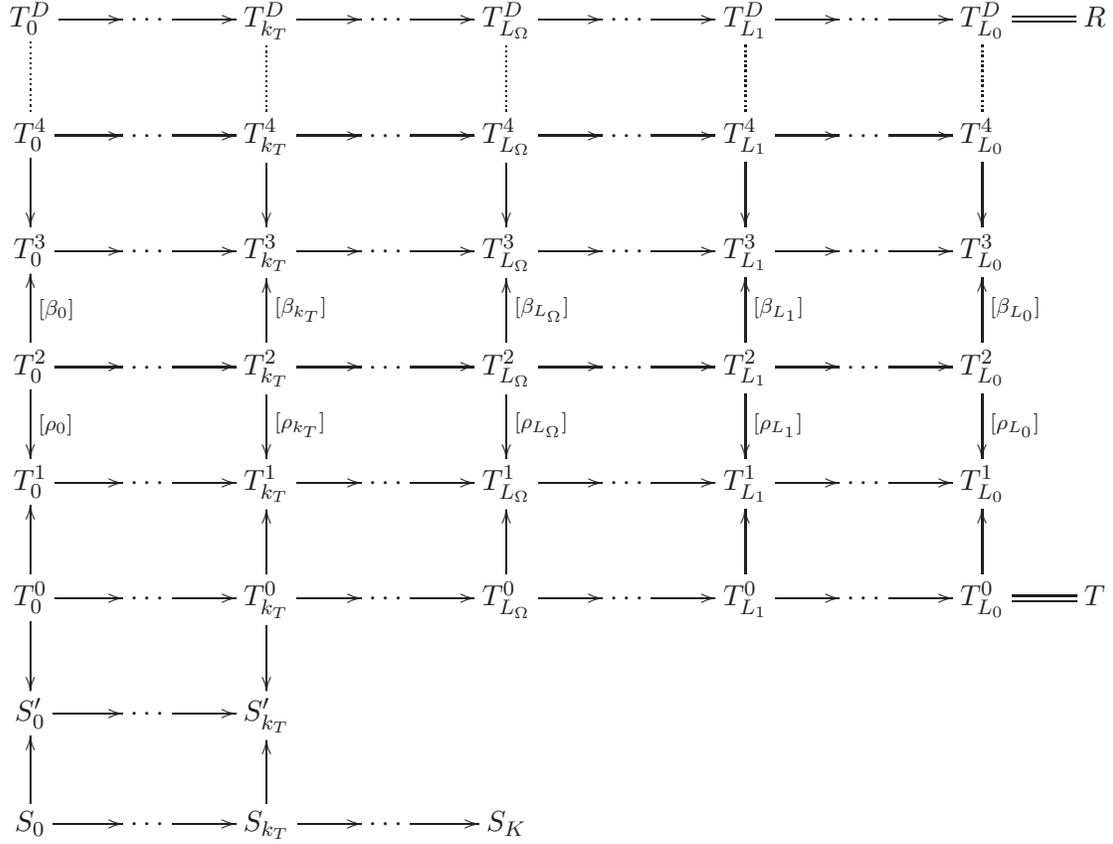
\begin{figure}
$$\xymatrix{
T^D_0 \ar@{.}[d]\ar[r]&\cdots \ar[r]&T^D_{k_T} \ar@{.}[d] \ar[r] &  \cdots \ar[r]& T^D_{L_\Omega} \ar[r]\ar@{.}[d] & \cdots \ar[r] & T^D_{L_1} \ar[r]\ar@{.}[d] &  \cdots \ar[r] & T^D_{L_0}\ar@{.}[d] \ar@{=}[r] & R \\
T^4_0 \ar[r]\ar[d]&\cdots \ar[r]&T^4_{k_T}  \ar[d] \ar[r] &  \cdots \ar[r]& T^4_{L_\Omega} \ar[r]\ar[d] & \cdots \ar[r] & T^4_{L_1} \ar[r]\ar[d] &  \cdots \ar[r] & T^4_{L_0}\ar[d] \\ 
T^3_0 \ar[r]&\cdots \ar[r]&T^3_{k_T} \ar[r] &  \cdots \ar[r]& T^3_{L_\Omega} \ar[r]         & \cdots \ar[r] & T^3_{L_1} \ar[r]         &  \cdots \ar[r] & T^3_{L_0} \\ 
T^2_0 \ar[r]\ar[d]^{[\rho_0]}\ar[u]_{[\beta_0]}&\cdots \ar[r]&T^2_{k_T} \ar[u]_{[\beta_{k_T}]} \ar[d]^{[\rho_{k_T}]} \ar[r] &  \cdots \ar[r]& T^2_{L_\Omega} \ar[r]\ar[d]^{[\rho_{L_\Omega}]}\ar[u]_{[\beta_{L_\Omega}]} & \cdots \ar[r] & T^2_{L_1} \ar[r]\ar[d]^{[\rho_{L_1}]}\ar[u]_{[\beta_{L_1}]} &  \cdots \ar[r] & T^2_{L_0}\ar[d]^{[\rho_{L_0}]}\ar[u]_{[\beta_{L_0}]} \\ 
T^1_0 \ar[r]&\cdots \ar[r]&T^1_{k_T} \ar[r] &  \cdots \ar[r]& T^1_{L_\Omega} \ar[r]         & \cdots \ar[r] & T^1_{L_1} \ar[r]         &  \cdots \ar[r] & T^1_{L_0} \\ 
T^0_0 \ar[r] \ar[d] \ar[u]& \cdots \ar[r] & T^0_{k_T} \ar[d]  \ar[u] \ar[r] &  \cdots \ar[r] & T^0_{L_\Omega} \ar[r]\ar[u] & \cdots \ar[r] & T^0_{L_1} \ar[r]\ar[u] &  \cdots \ar[r] & T^0_{L_0} \ar[u]  \ar@{=}[r] & T
\\
S'_0 \ar[r]          & \cdots \ar[r] & S'_{k_T} \\
S_0 \ar[r] \ar[u] & \cdots \ar[r] & S_{k_T} \ar[r] \ar[u]  & \cdots \ar[r] & S_K \\
}$$
\caption{The Big Diagram, Step 0. \hfill\break We emphasize the columns indexed by $L_\Omega, \ldots, L_1, L_0$. Each horizontal row is a foldable sequence, and the rectangle between any two rows is a combing rectangle. The bottom row is a fold sequence, and the $T^0$ row from $T^0_{k_T}$ to $T^0_{L_0}$ is a fold sequence. Each peak of the $T_{L_0}$ column is the union of its two collapse graphs. Rows in this and subsequent diagrams will be indicated by stripping off subscripts, for instance the ``$T^0$ row'' refers to the foldable sequence $T^0_0 \mapsto\cdots\mapsto T^0_{L_0}$; similarly, columns are indicated by stripping off superscripts. Since each peak of column $T_{L_0}$ between rows $T^0$ and $T^D$ is the union of its two collapse graphs, it follows that each peak of each column $T_j$ between rows $T^0$ and $T^D$ is the union of its two collapse graphs, because the two collapse graphs at a column~$j$ peak $T^{2i}_j$ are the pullbacks under the foldable map $T^{2i}_j \mapsto T^{2i}_{L_0}$ of the two collapse graphs at the corresponding column $L_0$ peak $T^{2i}_{L_0}$.}
\label{FigureBigDiagram0}
\end{figure}

Proposition~\ref{PropFSUContraction} will be proved by explicitly transforming the Big Diagram, Step 0 into a projection diagram from $R$ onto $S_0 \mapsto\cdots\mapsto S_K$ of an appropriate depth $l$ needed to verify the conclusions of the proposition. This transformation is primarily an induction that uses the pushdown tools of Section~\ref{SectionPushingDownPeaks}, followed by an epilogue which uses the pushdown tools one more time. As the proof progresses we will consider the truncated fold sequences $T^0_{k_T} \mapsto\cdots\mapsto T^0_{L_\omega}$ for increasing values of $\omega$, but such truncation will not affect measurements of free splitting units between $T^0_i$ and $T^0_j$ as long as $k_T \le i \le j \le L_\omega$ (see the remark following Definition~\ref{DefGeneralFSU}).

\subparagraph{Induction.} We explain in detail how to carry out the first step of the induction. Under our assumption that $d(T,R) \ge 3$, the $T_{L_0}$ column of the Big Diagram, Step 0 has a peak at $T^2_{L_0}$. Assuming furthermore that $\Upsilon \ge b_1$, equivalently $\Omega \ge 1$, then $L_1$ is defined and there are $\ge b_1 = 4 \rank(F)-3$ free splitting units between $T^0_{L_1}$ and $T^0_{L_0}$. We may therefore apply Proposition~\ref{PropPushdownInToto} to the portion of the diagram between the $T^0$ and $T^4$ rows as follows: trim away all portions of the Big Diagram, Step 0 that lie to the right of the $T_{L_1}$ column and below the $T^D$ row, and use the conclusion of Proposition~\ref{PropPushdownInToto} to replace the combing rectangles between the $T^0$ and $T^4$ rows, to get the Big Diagram, Step 0.1, shown in Figure~\ref{FigureBigDiagram0point1}.

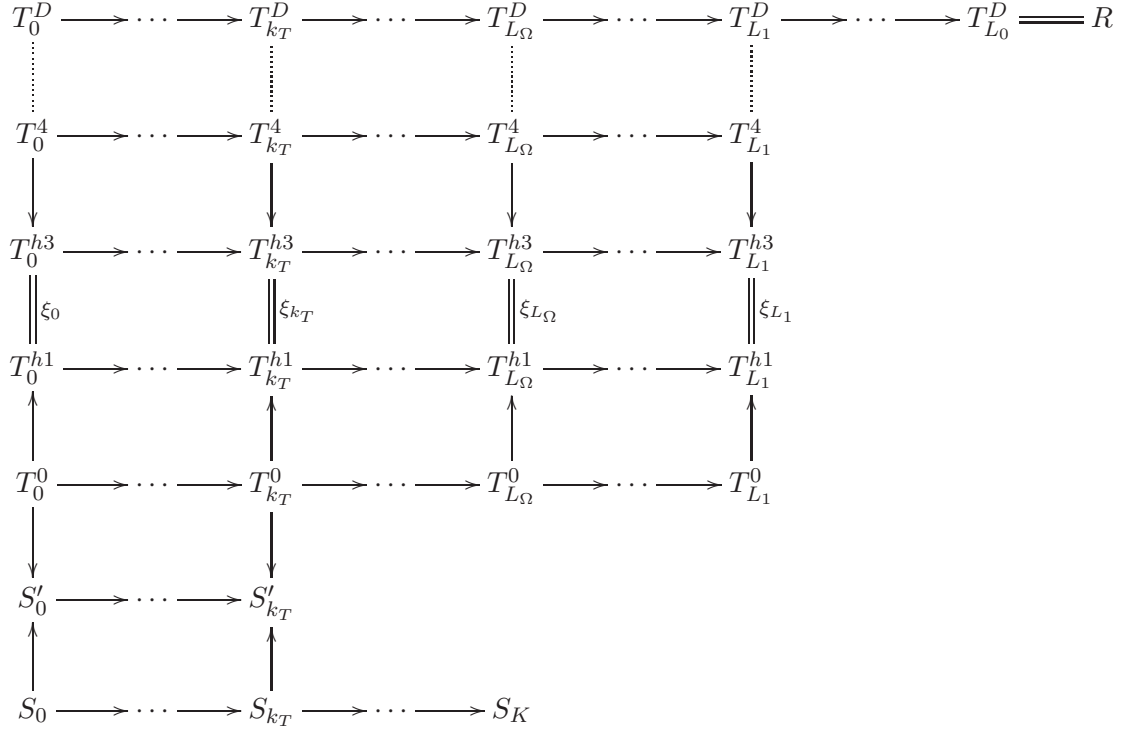
\begin{figure}[h]
$$\xymatrix{
T^D_0 \ar@{.}[d]\ar[r]&\cdots \ar[r]&T^D_{k_T} \ar@{.}[d] \ar[r] &  \cdots \ar[r]& T^D_{L_\Omega} \ar[r]\ar@{.}[d] & \cdots \ar[r] & T^D_{L_1} \ar[r] \ar@{.}[d] &  \cdots \ar[r] & T^D_{L_0} \ar@{=}[r] & R \\
T^4_0 \ar[r]\ar[d]&\cdots \ar[r]&T^4_{k_T}  \ar[d] \ar[r] &  \cdots \ar[r]& T^4_{L_\Omega} \ar[r]\ar[d] & \cdots \ar[r] & T^4_{L_1} \ar[d] \\ 
T^{h3}_0  \ar@{=}[d]^{\xi_0}    \ar[r]  & \cdots\ar[r]& T^{h3}_{k_T}  \ar@{=}[d]^{\xi_{k_T}}   \ar[r]  &\cdots\ar[r] & T^{h3}_{L_\Omega}  \ar@{=}[d]^{\xi_{L_\Omega}} 
          \ar[r] &\cdots \ar[r]  & T^{h3}_{L_1}   \ar@{=}[d]^{\xi_{L_1}}                 
\\ 
T^{h1}_0             \ar[r] & \cdots\ar[r]& T^{h1}_{k_T}             \ar[r]   &\cdots \ar[r] & T^{h1}_{L_\Omega}            
          \ar[r]  &\cdots \ar[r]  & T^{h1}_{L_1}                  
\\ 
T^0_0 \ar[r] \ar[d] \ar[u]& \cdots \ar[r] & T^0_{k_T} \ar[d]  \ar[u] \ar[r] &  \cdots \ar[r] & T^0_{L_\Omega} \ar[r]\ar[u] & \cdots \ar[r] & T^0_{L_1} \ar[u] \\
S'_0 \ar[r]          & \cdots \ar[r] & S'_{k_T} \\
S_0 \ar[r] \ar[u] & \cdots \ar[r] & S_{k_T} \ar[r] \ar[u]  & \cdots \ar[r] & S_K \\
}$$
\caption{The Big Diagram, Step 0.1.}
\label{FigureBigDiagram0point1}
\end{figure}
The rectangles of the Big Diagram, Step 0.1 between the $T^0$ and $T^{h1}$ rows and between the $T^{h3}$ and $T^4$ rows are combing rectangles. Each $\xi_j \from T^{h1}_j \to T^{h3}_j$ is a conjugacy, possibly nonsimplicial. Now we must pause to patch things up in order to make these conjugacies simplicial.

We claim that, by an operation of equivariant subdivision of simplicial structures and re-assignment of barycentric coordinates on edgelets, carried out over all free splittings in Big Diagram, Step 0.1, but \emph{without} changing any of the functions, we may assume that the conjugacies $\xi_i$ are indeed simplicial maps. Here are the details of this operation. 

\break

Consider first the conjugacy $\xi_{L_1} \from T^{h1}_{L_1} \to T^{h3}_{L_1}$. We may subdivide $T^{h1}_{L_1}$ at the pre-image of the vertex set of $T^{h3}_{L_1}$, and simultaneously subdivide $T^{h3}_{L_1}$ at the image of the vertex set of $T^{h1}_{L_1}$, to obtain new equivariant vertex sets on which $\xi_{L_1}$ is a bijection; it is also a bijection of edgelets, although it may not yet respect barycentric coordinates. We may then reassign the barycentric coordinates on one edgelet of $T^{h1}_{L_1}$ in each $F$-orbit, and move these coordinates around by the $F$-action, to obtain a new simplicial structure on~$T^{h1}_{L_1}$. We may then push these coordinates forward under the map $\xi_{L_1}$ to obtain new barycentric coordinates on the edgelets of $T^{h3}_{L_1}$. Having carried out these operations, the map $\xi_{L_1}$ is now a simplicial conjugacy. 

Now we move left one step: by a similar subdivision/re-assignment on $T^{h1}_{L_1-1}$, pulling back vertices and barycentric coordinates under the foldable map $T^{h1}_{L_1-1} \mapsto T^{h1}_{L_1}$, we may assume that this map is simplicial. Similarly, by a subdivision/re-assignment on $T^{h3}_{L_1-1}$, we may assume that the foldable map $T^{h3}_{L_1-1} \mapsto T^{h3}_{L_1}$ is simplicial. We have now verified that in the commutative diagram
$$\xymatrix{
T^{h3}_{L_1-1} \ar[r] \ar@{=}[d]^{\xi_{L_1-1}} & T^{h3}_{L_1} \ar@{=}[d]^{\xi_{L_1}} \\
T^{h1}_{L_1-1} \ar[r]                                      & T^{h1}_{L_1}
}$$
the top, bottom, and right sides are simplicial maps; by commutativity, the left side is therefore automatically simplicial.

Now we continue to move left: by similar subdivisions/re-assignments carried out one at a time on the trees in rows $T^{h1}$ and $T^{h3}$, moving to the left one at a time from the last map in each row, we may assume that these rows are simplicial; having done this, by commutativity each of the maps $\xi_j \from T^{h1}_j \to T^{h3}_j$ is automatically a simplicial conjugacy. Now we move up: by similar subdivisions/re-assignments carried out one at a time on the trees in rows $T^4$, \ldots, $T^D$, starting with the collapse maps $T^4_j \mapsto T^{h3}_j$ and moving upward, we may assume that each vertical arrow above row $T^{h3}$ is simplicial; having done this, each of the horizontal arrows from row $T^{h3}$ upward and between columns $T_0$ and $T_{L_1}$ is automatically simplicial. Now, from $T^D_{L_1}$ we move to the right: by similar subdivisions/re-assignments we may assume that each of the maps $T^D_{L_1}\mapsto\cdots\mapsto T^D_{L_0}=R$ is simplicial. Finally, in a similar manner moving down from row $T^{h3}$ to row $S$, then moving right from $S_{k_T}$ to $S_K$, we have proved the claim. 

Knowing now that we have \emph{simplicial} conjugacies $\xi_j \from T^{h1}_j \to T^{h3}_j$, and using commutativity of the rectangle between rows $T^{h1}$ and $T^{h3}$, we may identify $T^{h1}_j$ and $T^{h3}_j$ via the map $\xi_j$, replacing these two rows by a single row as shown in \emph{The Big Diagram,~Step~1.}

In summary, when $d(T,R) \ge 3$ and $\Upsilon \ge b_1$, we have completed the first iteration of the induction argument: starting from the Big Diagram Step 0, by applying Proposition~\ref{PropPushdownInToto}, trimming away everything to the right of column $T_{L_1}$ and below row $T^D$, and replacing everything between rows $T^0$ and $T^4$, we get the Big Diagram Step 0.1, and then by subdividing and re-assigning barycentric coordinates we may assume that the conjugacies between rows $T^{h1}$ and $T^{h3}$ are simplicial. Identifying rows $T^{h1}$ and $T^{h3}$, we obtain the Big Diagram Step~1, shown in Figure~\ref{FigureBigDiagram1}. In the process we have decreased by~$2$ the lengths of all vertical zig-zag paths and the number of combing rectangles between the $T^0$ and $T^D$ rows. Observe that the conjugacy class of the free splitting $R$, and the equivalence class of the fold sequence $S_0 \mapsto\cdots\mapsto S_K$, have not been altered by these subdivision/re-assigmnent operations. 
\begin{figure}[h]
$$\xymatrix{
T^D_0 \ar@{.}[d]\ar[r]&\cdots \ar[r]&T^D_{k_T} \ar@{.}[d] \ar[r] &  \cdots \ar[r]& T^D_{L_\Omega} \ar[r]\ar@{.}[d] & \cdots \ar[r] & T^D_{L_1} \ar[r] \ar@{.}[d] &  \cdots \ar[r] & T^D_{L_0} \ar@{=}[r] & R \\
T^4_0 \ar[r]\ar[d]&\cdots \ar[r]&T^4_{k_T}  \ar[d] \ar[r] &  \cdots \ar[r]& T^4_{L_\Omega} \ar[r]\ar[d] & \cdots \ar[r] & T^4_{L_1} \ar[d] \\ 
T^{h}_0             \ar[r] & \cdots\ar[r]& T^{h}_{k_T}             \ar[r]   &\cdots \ar[r] & T^{h}_{L_\Omega}            
          \ar[r]  &\cdots \ar[r]  & T^{h}_{L_1}                  
\\ 
T^0_0 \ar[r] \ar[d] \ar[u]& \cdots \ar[r] & T^0_{k_T} \ar[d]  \ar[u] \ar[r] &  \cdots \ar[r] & T^0_{L_\Omega} \ar[r]\ar[u] & \cdots \ar[r] & T^0_{L_1} \ar[u] \\
S'_0 \ar[r]          & \cdots \ar[r] & S'_{k_T} \\
S_0 \ar[r] \ar[u] & \cdots \ar[r] & S_{k_T} \ar[r] \ar[u]  & \cdots \ar[r] & S_K \\
}$$
\caption{The Big Diagram, Step 1}
\label{FigureBigDiagram1}
\end{figure}
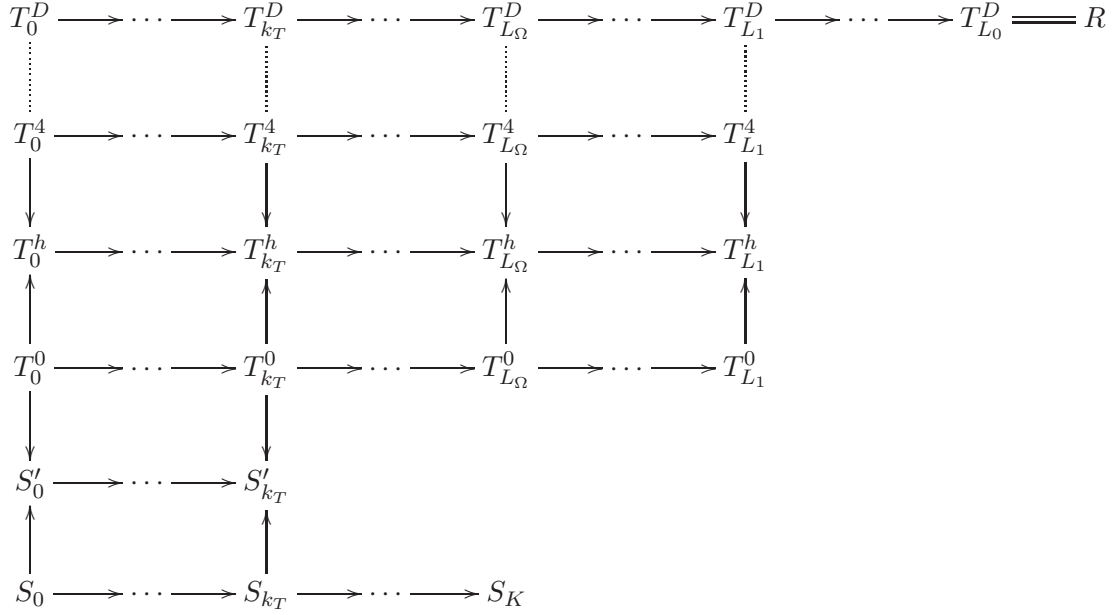

To complete the inductive step there is one last thing to do, namely to verify that along the zig-zag path in column $T_{L_1}$ on the right side of the Big Diagram, Step 1, each peak is the union of its two collapse graphs. This is evident for each peak from $T^6_{L_1}$ upward, since the collapse maps and collapse graphs are unchanged at those peaks from the Big Diagram, Step 0. For the peak at $T^4_{L_1}$, one of the collapse graphs is unchanged from the Big Diagram, Step 0, namely that of the map $T^4_{L_1} \mapsto T^5_{L_1}$. For the collapse graph of the map $T^4_{L_1} \mapsto T^h_{L_1}$, we use the part of the conclusion of Proposition~\ref{PropPushdownInToto} which tells us that the combing rectangle in the Big Diagram Step 1 between the $T^4$ and $T^h$ rows is obtained by an application of \emph{Composition of combing rectangles}, Lemma~\ref{LemmaCombingComp}, using the combing rectangle in the Big Diagram Step 0 between the $T^4$ and $T^3$ rows and between the $T_0$ and $T_{L_1}$ columns. What Lemma~\ref{LemmaCombingComp} allows us to conclude is that the collapse graph of the Step 0 map $T^4_{L_1} \mapsto T^3_{L_1}$ is contained in the collapse graph of Step 1 map $T^4_{L1} \mapsto T^h_{L_1}$. The union of the two collapse graphs of $T^4_{L_1}$ in the Big Diagram, Step 1 is therefore still equal to $T^4_{L_1}$.

\smallskip

\textbf{Remark.} The reader may wonder why the initial normalization step was necessary in the proof of Proposition~\ref{PropPushdownInToto}: we could have started with a \emph{normalized} zig-zag geodesic on the right side of the Big Diagram, Step 0. This would imply that the $T^4$ column in that diagram is normalized at $T^4_{L_1}$. Nonetheless it is possible that the $T^4$ column in the Big Diagram, Step 1 is not normalized at $T^4_{L_1}$, because the collapse graph for $T^4_{L_1} \mapsto T^h_{L_1}$ may be strictly larger than the collapse graph for $T^4_{L_1} \mapsto T^3_{L_1}$. If so then the normalization step of Proposition~\ref{PropPushdownInToto} is inescapable in the next step of the induction.

\smallskip

We now describe the induction step in general. From the hypothesis we have $d(T,R) \le \max\{2\Omega,1\}$. If $d(T,R) \le 2$ then we refer to the Epilogue below. Otherwise, under the assumption $d(T,R) \ge 3$, we have $D \le d(T,R) + 1 \le 2 \Omega + 1$, and so we may repeat the above argument inductively a total of $\lfloor \frac{D-1}{2} \rfloor$ times, removing the corrugation peaks one at a time. For each $\omega = 2,\ldots,\lfloor \frac{D-1}{2} \rfloor$, at the $\omega$ step of the induction we do the following. At the start of the $\omega$ step we have the Big Diagram, Step~$\omega-1$, analogous to the Big Diagram Step~$1$ but with $L_{\omega-1}$ in place of $L_1$ and $L^{2\omega}$ in place of~$L^4$, and with a stack of $D-2\omega+2$ combing rectangles between the $T^0$ and $T^D$ rows. We trim away the portion of the diagram to the right of column $T_{L_\omega}$, on or above row $T^0$, and below row $T^D$. We replace the four combing rectangles between rows $T^0$ and $T^{2\omega+2}$ by two combing rectangles and a commutative diagram of conjugacies. We carry out a subdivision/re-assignment operation which allows us to assume that the conjugacies are simplicial. We then collapse the commutative diagram of conjugacies, identifying its two rows into a single row. We have now produced the Big Diagram, Step~$\omega$, with a stack of $D-2\omega$ combing rectangles between the $T^0$ and $T^D$ rows: we have decreased by~$2$ the lengths of all vertical zig-zag paths between the $T^0$ and $T^D$ rows and decreased by~$1$ the number of corrugation peaks. Finally we verify that each peak along column $T_{L_\omega}$ is still the union of its two collapse graphs. 

At each stage of the induction, we have not altered the conjugacy class of $R$ nor the equivalence class of $S_0\mapsto\cdots\mapsto S_K$.

\subparagraph{Epilogue.} If $d(T,R) \ge 3$, when the induction process stops we have backed up to column~$T_{L_\omega}$ where $\omega = \lfloor \frac{D-1}{2} \rfloor$, and there are no remaining corrugation peaks above row~$T^0$. We obtain the Big Diagram, Step $\lfloor \frac{D-1}{2} \rfloor$, a not-so-big diagram that comes in two cases. The Case~1 diagram occurs when $D$ is even, and it has two combing rectangles between row $T^0$ and row $T^D$; see Figure~\ref{FigureNotSoBigCase1}. The Case~2 diagram occurs when $D$ is odd and has only one such combing rectangle; see Figure~\ref{FigureNotSoBigCase2}. In each of these diagrams, the conjugacy class of $R$ and the equivalence class of the fold sequence $S_1 \mapsto\cdots\mapsto S_K$ have not been changed from the initial setup in the Prologue.

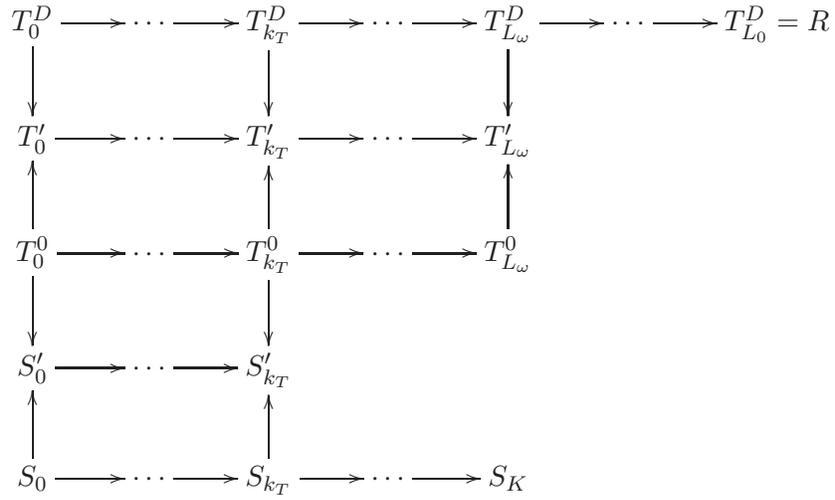
\begin{figure}
$$\xymatrix{
T^D_0 \ar[r] \ar[d]      & \cdots \ar[r] & T^D_{k_T} \ar[r] \ar[d]      & \cdots \ar[r] & T^D_{L_\omega} \ar[r] \ar[d] & \cdots \ar[r] & T^D_{L_0} = R \\
T'_0 \ar[r]                   & \cdots \ar[r] & T'_{k_T} \ar[r]                   & \cdots \ar[r] & T'_{L_\omega} \\
T^0_0 \ar[r] \ar[d] \ar[u] & \cdots \ar[r] & T^0_{k_T} \ar[r] \ar[d] \ar[u] & \cdots \ar[r] & T^0_{L_\omega}\ar[u] \\
S'_0 \ar[r]          & \cdots \ar[r] & S'_{k_T} \\
S_0 \ar[r] \ar[u] & \cdots \ar[r] & S_{k_T} \ar[r] \ar[u]  & \cdots \ar[r] & S_K \\
}$$
\caption{Case 1: a collapse--expand from $T^0$ to $T^D$}
\label{FigureNotSoBigCase1}
\end{figure}

\begin{figure}
$$\xymatrix{
T^D_0 \ar[r]      & \cdots \ar[r] & T^D_{k_T} \ar[r]     & \cdots \ar[r] & T^D_{L_\omega} \ar[r] & \cdots \ar[r] & T^D_{L_0} = R \\
T^0_0 \ar[r] \ar[d] \ar[u] & \cdots \ar[r] & T^0_{k_T} \ar[r] \ar[d] \ar[u] & \cdots \ar[r] & T^0_{L_\omega}\ar[u] \\
S'_0 \ar[r]          & \cdots \ar[r] & S'_{k_T} \\
S_0 \ar[r] \ar[u] & \cdots \ar[r] & S_{k_T} \ar[r] \ar[u]  & \cdots \ar[r] & S_K \\
}$$
\caption{Case 2: a collapse from $T^0$ to $T^D$.}
\label{FigureNotSoBigCase2}
\end{figure}
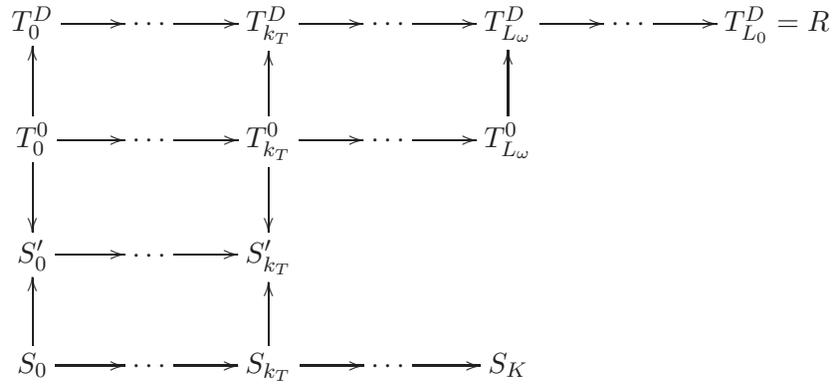

If $d(T,R) \le 2$ then, starting from the augmented projection diagram depicted in the prologue, and depending on the nature of the geodesic from $T$ to $R$, we proceed as follows. If $d(T,R)=1$ and there is a collapse $T \collapses R$, we comb the $T^0$ row along this collapse to obtain the Case~2 diagram with $\omega=0$ and $T^D_{L_\omega}=T^D_{L_0}=R$. If $d(T,R)=1$ and there is an expansion $T \expandsto R$ then we append an improper collapse $T \collapsesto T$ to get a length~2 collapse--expand zig-zag $T \collapsesto T \expandsto R$, and we comb the $T^0$ row along this collapse--expand to obtain the Case~1 diagram with similar notation changes. If $d(T,R)=2$ and there is a collapse--expand from $T$ to $R$ then, combing the $T^0$ row along this collapse--expand, we produce the Case~1 diagram with similar notation changes. Finally, if $d(T,R)=2$ and there is an expand--collapse from $T$ to $R$, then combing the $T^0$ row along this expand--collapse, we obtain the Case~3 diagram in Figure~\ref{FigureNotSoBigCase3}.

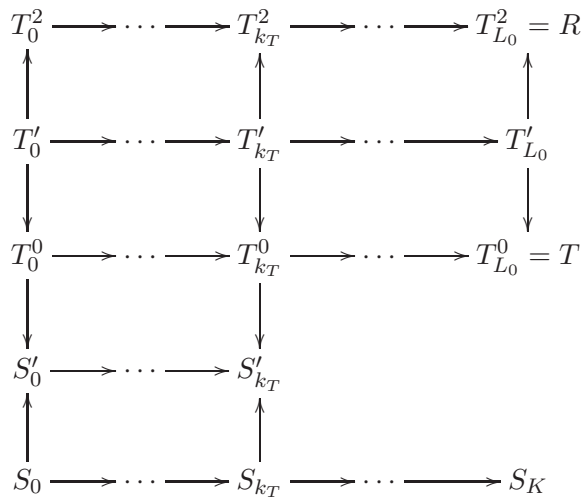
\begin{figure}
$$\xymatrix{
T^2_0 \ar[r]               & \cdots \ar[r] & T^2_{k_T} \ar[r]               & \cdots \ar[r] & T^2_{L_0} = R \\
T'_0 \ar[r] \ar[u]\ar[d] & \cdots \ar[r] & T'_{k_T} \ar[r] \ar[u]\ar[d] & \cdots \ar[r] & T'_{L_0} \ar[u]\ar[d]  \\
T^0_0 \ar[r]\ar[d]       & \cdots \ar[r] & T^0_{k_T} \ar[r] \ar[d]      & \cdots \ar[r] & T^0_{L_0} = T\\
S'_0 \ar[r]          & \cdots \ar[r] & S'_{k_T} \\
S_0 \ar[r] \ar[u] & \cdots \ar[r] & S_{k_T} \ar[r] \ar[u]  & \cdots \ar[r] & S_K \\
}$$
\caption{Case 3: an expand--collapse from $T^0$ to $T^2$.}
\label{FigureNotSoBigCase3}
\end{figure}

We now finish off Case~1; afterwards we shall reduce Cases~2 and~3 to Case~1. In the Case 1 diagram, trim off everything to the right of column~$T_{k_T}$, on or above row~$T^0$, and below row~$T^D$, to obtain the diagram shown in Figure~\ref{FigureCaseOneTrimmed}, which has a corrugation peak along the $T^0$ row. We must consider two subcases, depending on whether the peak $T^0_{k_T}$ of the W zig-zag in column $k_T$ is the union of its two collapse graphs $b_{k_T}$, $\rho_{k_T}$.
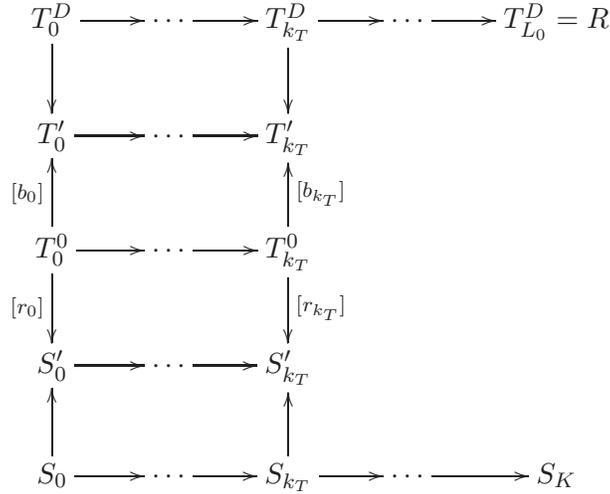
\begin{figure}[h]
$$\xymatrix{
T^D_0 \ar[r] \ar[d]      & \cdots \ar[r] & T^D_{k_T} \ar[r] \ar[d]      & \cdots \ar[r] & T^D_{L_0} = R \\
T'_0 \ar[r]                   & \cdots \ar[r] & T'_{k_T}  \\
T^0_0 \ar[r] \ar[d]_{[r_0]} \ar[u]^{[b_0]} & \cdots \ar[r] & T^0_{k_T} \ar[d]^{[r_{k_T}]} \ar[u]_{[b_{k_T}]}  \\
S'_0 \ar[r]          & \cdots \ar[r] & S'_{k_T} \\
S_0 \ar[r] \ar[u] & \cdots \ar[r] & S_{k_T} \ar[r] \ar[u]  & \cdots \ar[r] & S_K \\
}$$
\caption{The Case~1 diagram, trimmed down.}
\label{FigureCaseOneTrimmed}
\end{figure}

Suppose first that $T^0_{k_T} \ne b_{k_T} \union r_{k_T}$ in Figure~\ref{FigureCaseOneTrimmed}. For each $j=0,\ldots,k_T$, in the tree $T^0_j$ which is the peak of the W zig-zag in column $j$, the union of its two collapse graphs $b_j \union r_j$ is a proper subgraph, that subgraph being the inverse image of $b_{k_T} \union r_{k_T}$ under the foldable map $T^0_j \mapsto T^0_{k_T}$. We may therefore carry out the simplistic pushdown depicted in Figure~\ref{FigureSimplistic}, in parallel as $j$ varies from $0$ to $k_T$, resulting in a diagram of the form depicted in Figure~\ref{FigureCaseOnePushedDown}.
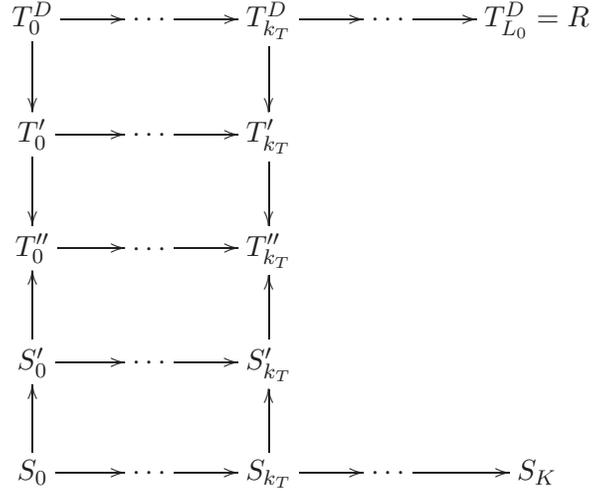
\begin{figure}
$$\xymatrix{
T^D_0 \ar[r] \ar[d]      & \cdots \ar[r] & T^D_{k_T} \ar[r] \ar[d]      & \cdots \ar[r] & T^D_{L_0} = R \\
T'_0 \ar[r] \ar[d]                & \cdots \ar[r] & T'_{k_T} \ar[d] \\
T''_0 \ar[r] & \cdots \ar[r] & T''_{k_T}  \\
S'_0 \ar[r] \ar[u]         & \cdots \ar[r] & S'_{k_T} \ar[u] \\
S_0 \ar[r] \ar[u] & \cdots \ar[r] & S_{k_T} \ar[r] \ar[u]  & \cdots \ar[r] & S_K \\
}$$
\caption{The result of a parallel simplistic pushdown on Figure \ref{FigureCaseOneTrimmed}, in the case when $T^0_{k_T} \ne \beta_{k_T} \union \rho_{k_T}$. Concatenating the upper two combing rectangles into a single one, and the same for the lower two, we obtain a projection diagram.}
\label{FigureCaseOnePushedDown}
\end{figure}
In Figure~\ref{FigureCaseOnePushedDown}, the $T''$ row is obtained by applying Proposition~\ref{PropCBC} \emph{Combing by collapse} using the collapse graphs $b_j \union r_j \subset T^0_j$, and the middle two combing rectangles are each obtained by an application of Lemma~\ref{LemmaCombingDecomp} \emph{Decomposition of combing rectangles}. By applications of Lemma~\ref{LemmaCombingComp} \emph{Composition of combing rectangles}, we may compose the lower two and the upper two combing rectangles of Figure~\ref{FigureCaseOnePushedDown} to produce a depth $k_T$ projection diagram from $R$ to $S_0 \mapsto\cdots\mapsto S_K$, and the proof of Proposition~\ref{PropFSUContraction} is complete in this case. 

Suppose next that $T^0_{k_T} = b_{k_T} \union r_{k_T}$ in Figure~\ref{FigureCaseOneTrimmed}. From the hypothesis of Proposition~\ref{PropFSUContraction}, there are $\ge b_1 = 4 \rank(F)-3$ free splitting units along the bottom row of the diagram between $S_0$ and~$S_{k_T}$. Let $\ell \in \{0,\ldots,k_T\}$ be the largest integer such that there are $\ge b_1$ free splitting units between $S_l$ and $S_{k_T}$, from which it follows that there are exactly $b_1$ free splitting units between $S_I$ and $S_{k_T}$. We may now carry out one last iteration of the Induction. Applying Proposition~\ref{PropPushdownInToto}, remove all portions of the diagram in Figure~\ref{FigureCaseOneTrimmed} to the right of column $l$, above the $S$ row, and below the $T^D$ row, and replace the four combing rectangles by two combing rectangles and a commutative diagram of conjugacies. After an operation of subdivision and re-assignment of barycentric coordinates, we may assume that the conjugacies are all simplicial. After collapsing the commutative diagram of conjugacies, identifying its two rows to a single row, we obtain the diagram depicted in Figure~\ref{FigureOneLastTime}, in which the conjugacy class of the free splitting $R$ and the equivalence class of the fold sequence $S_0 \mapsto\cdots\mapsto S_K$ remain unchanged. This is the desired projection diagram from the free splitting $R$ to the fold sequence $S_0 \mapsto\cdots\mapsto S_K$ which completes the proof of Proposition~\ref{PropFSUContraction} in case~1.
\begin{figure}
$$\xymatrix{
T^D_0 \ar[r] \ar[d]      & \cdots \ar[r] & T^D_{l} \ar[r] \ar[d]      & \cdots \ar[r] & T^D_{L_0} = R \\
S^h_0 \ar[r]                   & \cdots \ar[r] & S^h_{l}  \\
S_0 \ar[r] \ar[u] & \cdots \ar[r] & S_{l} \ar[r] \ar[u]  & \cdots \ar[r] & S_K \\
}$$
\caption{The projection diagram resulting from one last iteration of the Induction carried out on Figure \ref{FigureCaseOneTrimmed}, in the case when $T^0_{k_T} = \beta_{k_T} \union \rho_{k_T}$.}
\label{FigureOneLastTime}
\end{figure}
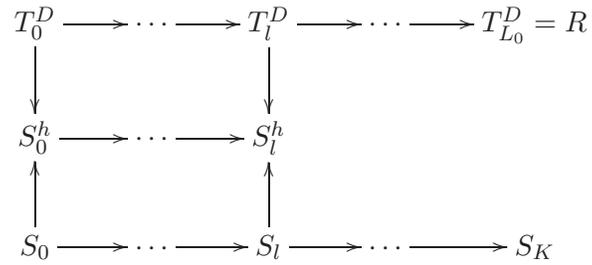

\subparagraph{Remark.} As was remarked earlier regarding the Big Diagram, Step 1 depicted in Figure~\ref{FigureBigDiagram1}, in the context of case~1 depicted in Figure~\ref{FigureOneLastTime}, the initial normalization step in the proof of Proposition~\ref{PropPushdownInToto} cannot be avoided, because there is no guarantee that the $S_{k_T}$ column is normalized at $T^0_{k_T}$.

\bigskip

We reduce case~2 to case~1 by producing a case~1 diagram: just attach an improper combing rectangle to the top of the case 2 diagram, by defining the foldable sequence $T'_0 \mapsto\cdots\mapsto T'_{L_\Omega}$ to equal the foldable sequence $T^D_0 \mapsto\cdots\mapsto T^D_{L_\Omega}$, and defining for each $j=0,\ldots,L_\Omega$ an improper collapse map $T^D_j \to T'_j$ which is just the identity map.

We also reduce case~3 to case~1. First trim away everything in the Case~3 diagram to the right of the $k_T$ column, on or above the $T^0$ row, and below the $T^2$ row. Next, apply Lemma~\ref{LemmaCombingComp}, \emph{Composition of combing rectangles}, to the two combing rectangles between the $S'$ row and the $T'$ row, concatenating them into a single combing rectangle. Finally, attach an improper combing rectangle to the top of the diagram as in case~2. The result is a case~1 diagram, completing the reduction.



\def\cprime{$'$} \def\cprime{$'$}
\providecommand{\bysame}{\leavevmode\hbox to3em{\hrulefill}\thinspace}
\providecommand{\MR}{\relax\ifhmode\unskip\space\fi MR }
\providecommand{\MRhref}[2]{%
  \href{http://www.ams.org/mathscinet-getitem?mr=#1}{#2}
}
\providecommand{\href}[2]{#2}


\begin{theindex}

  \item back greedy subsequence, 51
  \item baseball diagram, 61

  \indexspace

  \item collapse, 6, 8, \see{\emph{Glossary}}{85}
    \subitem map, \see{\emph{Glossary} under \emph{Map}}{85}
    \subitem proper and improper, 9
  \item combing rectangle, 27
  \item conjugacy, 1, 5, \see{\emph{Glossary}}{85}

  \indexspace

  \item derivative, 14
  \item direction, 13

  \indexspace

  \item edgelet, 6, \see{\emph{Glossary}}{85}
    \subitem of a foldable map, 15
  \item expansion, 8, \see{\emph{Glossary}}{85}

  \indexspace

  \item fold, 18
    \subitem full, 18
      \subsubitem improper, 18
      \subsubitem proper, 18
    \subitem type IA, 18
    \subitem type IIIA, 18
  \item fold factorization, 21
  \item fold map, 18, \see{\emph{Glossary} under \emph{Map}}{85}
  \item fold path, 21, 
		\see{\emph{Glossary} under \emph{Foldable sequence}}{85}
  \item fold sequence, 20, 
		\see{\emph{Glossary} under \emph{Foldable sequence}}{85}
    \subitem equivalence, 21
  \item foldable map, 14, \see{\emph{Glossary} under \emph{Map}}{85}
  \item foldable sequence, 27, \see{\emph{Glossary}}{85}
  \item free splitting, 1, 4
  \item free splitting unit, 50
  \item front greedy subsequence, 50

  \indexspace

  \item gate, 14

  \indexspace

  \item map, 5, \see{\emph{Glossary}}{85}
    \subitem collapse, 1, 6
    \subitem fold, 18
    \subitem foldable, 14

  \indexspace

  \item natural cell structure, 5, \see{\emph{Glossary}}{85}
  \item natural core, 8
  \item normalization diagram, 60

  \indexspace

  \item projection diagram, 28
    \subitem augmented, 54

  \indexspace

  \item W diagram, 59

  \indexspace

  \item zig-zag, 10, \see{\emph{Glossary}}{85}

\end{theindex}

\section*{Glossary}

\begin{description}
\item[Collapse and expansion.]\index{collapse|see{\emph{Glossary}}}\index{expansion|see{\emph{Glossary}}} Inverse relations amongst free splittings, denoted $S \collapses T$ and $T \expandsto S$ respectively, defined so that $T$ is obtained from $S$ by collapsing to a point each component of some proper, equivariant, natural subgraph of $S$.
\item[Conjugacy.]\index{conjugacy|see{\emph{Glossary}}} An equivariant homeomorphism between free splittings, which need not be a map.
\item[Edgelet.]\index{edgelet|see{\emph{Glossary}}} A 1-cell of some given simplicial structure on a tree. The term is also used in a relative sense --- given a foldable map $f \from S \to T$ and an edgelet $e$ of $T$, an \emph{$e$-edgelet of $f$} is any edgelet of $S$ mapped by $f$ to $e$.
\item[Foldable sequence.]\index{foldable sequence|see{\emph{Glossary}}} A sequence of maps of free splittings in which any composition of any subinterval of that sequence is a foldable map. 
\begin{itemize}
\item A \emph{fold sequence}\index{fold sequence|see{\emph{Glossary} under \emph{Foldable sequence}}} is a special kind of foldable sequence in which each map is a fold. 
\item A \emph{fold path}\index{fold path|see{\emph{Glossary} under \emph{Foldable sequence}}} is the sequence of vertices in $\FS'(F)$ obtained from the conjugacy classes of the free splittings along a fold sequence.
\end{itemize}
\item[Map.]\index{map|see{\emph{Glossary}}} An equivariant simplicial function between free splittings. Important types of maps include:
\begin{itemize}
\item A \emph{collapse map}\index{collapse!map|see{\emph{Glossary} under \emph{Map}}}
 collapses to a point each edge in an equivariant subgraph. 
\item A \emph{foldable map}\index{foldable map|see{\emph{Glossary} under \emph{Map}}}
 is injective on each natural edge, and has $\ge 3$ gates at each natural vertex.  
\item A \emph{fold map}\index{fold map|see{\emph{Glossary} under \emph{Map}}} is a foldable map defined by identifying initial segments of some pair of natural edges with the same initial vertex. 
\end{itemize}
\item[Natural cell structure.]\index{natural cell structure|see{\emph{Glossary}}} Every tree with no isolated ends and no valence~$1$ vertices --- in particular every free splitting of a free group of rank~$\ge 2$ --- has a natural cell structure, whose vertices are the points that (locally) separate the tree into some number of components~$\ge 3$. A natural subgraph is a subcomplex of the natural cell structure. Any other cell structure on the graph is a refinement of the natural cell structure.
\item[Zig-zag path.]\index{zig-zag|see{\emph{Glossary}}} An edge path in $\FS'(F)$ which alternates between expansions and collapses. Examples include all geodesic edge paths in $\FS'(F)$.
\end{description}

\end{document}